\documentclass[a4paper,10pt]{amsart}
\usepackage[utf8]{inputenc}
\usepackage{amsthm}
\usepackage{amsmath}
\usepackage{comment}
\usepackage{bm}
\usepackage{amsfonts}
\usepackage{amssymb}
\usepackage{mathtools}
\usepackage{stmaryrd}
\usepackage{mathrsfs}
\usepackage{setspace}
\usepackage{xcolor}
\usepackage{textgreek}
\usepackage{todonotes}
\usepackage{caption}
\usepackage{tikz}
\usepackage{pgfplots}
    \pgfplotsset{
        compat=1.15,
        width=8cm,
    }
\usepackage[a4paper, left=2.8cm, right=2.8cm, top=3cm, bottom=3cm]{geometry}

\usepackage{thmtools} 
\usepackage{bm}
\usepackage{keyval}

\usepackage[pdfdisplaydoctitle,colorlinks,breaklinks,urlcolor=blue,linkcolor=blue,citecolor=blue]{hyperref} 
\mathtoolsset{showonlyrefs}

\newcommand{\R}{\mathbb{R}}

\renewcommand{\L}{{L^{\eta,\gamma}}}

\newcommand{\Srad}{{\mathscr{S}_{\mathrm{rad}}}}

\let\div\relax 

\DeclareMathOperator{\div}{div}

\newcommand{\eps}{\varepsilon}

\newcommand{\CC}{\mathbb{C}}

\newcommand{\NN}{\mathbb{N}}

\newcommand{\RR}{\mathbb{R}}

\newcommand{\vertiii}[1]{{\left\vert\kern-0.25ex\left\vert\kern-0.25ex\left\vert #1 
    \right\vert\kern-0.25ex\right\vert\kern-0.25ex\right\vert}}

\numberwithin{equation}{section}
\newtheorem{theorem}{Theorem}[section]

\newtheorem{defi}[theorem]{Definition}
\newtheorem{cor}[theorem]{Corollary}

\newtheorem{lemma}[theorem]{Lemma}

\newtheorem{prop}[theorem]{Proposition}

\newtheorem{rmk}[theorem]{Remark}

\theoremstyle{definition}

\newenvironment{acknowledgements}{%
  \begin{abstract}
}{%
  \end{abstract}
}

\title[Instability and non-uniqueness for the Keller-Segel system]{Spectral instability and non-uniqueness for the Keller-Segel system}

\author[E. Luongo]{Eliseo Luongo}
\address{Fakultät für Mathematik, Universität Bielefeld, 33501 Bielefeld, Germany} 
\email{eluongo@math.uni-bielefeld.de  }
\author[U. Pappalettera]{Umberto Pappalettera}
\address{Departement Mathematik und Informatik, Universit\"at Basel, Spiegelgasse 1, CH-4051 Basel, Switzerland.} 
\email{umberto.pappalettera@unibas.ch}

\keywords{Keller-Segel, Spectral instability, Non-uniqueness.}
\date\today

\begin{document}

\begin{abstract}
We show that the Cauchy problem associated with the parabolic-elliptic Keller-Segel model is locally ill-posed in $L^q(\mathbb{R}^n)$ for dimensions $n \in \{3,\dots,9\}$ and throughout the supercritical range $q\in [1,\frac{n}{2})$.
An analog non-uniqueness result is given in the critical space of bounded functions taking values in  $L^{n/2,\infty}(\R^n)$.
The non-uniqueness is driven by an instability mechanism in self-similarity variables, in the spirit of the program proposed by Jia and Šverák \cite{JiSv15} for the three-dimensional Navier–Stokes equations.
\end{abstract}

\maketitle

\section{Introduction}
Chemotaxis is a biological phenomenon describing the displacement
of a population of bacteria in response to a secreted chemical signal that is spread in the environment. From a mathematical viewpoint, one of the most studied chemotactic systems is the parabolic–elliptic Keller–Segel model introduced in \cite{JaLu92}:
\begin{align} \label{eq:KS_physical}
\begin{dcases}
\partial_t c = \Delta c - \nabla \cdot (vc),
\\
v = \nabla (-\Delta)^{-1} c,
\end{dcases}
\end{align}
where the unknown is a scalar quantity $c : [0,T] \times \R^n \to \R$.
System \eqref{eq:KS_physical} can also be interpreted as a simplified model for self-gravitating matter in stellar dynamics, see for instance \cite{chandrasekhar1957introduction, wolansky1992steady, ascasibar2013approximate}.
The equation defining $v$ can be formally solved via convolution with the gradient of the Green function $K$:
\begin{align*}
   v(t,x) = \int_{\R^n} \nabla K(x-y) c(t,y) dy.
\end{align*}
In this paper we are mostly concerned with \emph{mild} solutions to \eqref{eq:KS_physical}, defined as follows.

\begin{defi}\label{def_mild_sol}
    Let $q\in [1,+\infty]$ and $T>0$. A mild $L^q$-solution to \eqref{eq:KS_physical} on the time interval $[0,T]$ and with initial condition $c_0\in L^q(\R^n)$ is a function $c  \in C([0,T] ; L^q(\R^n))$ which is a distributional solution to \eqref{eq:KS_physical} and satisfies the mild formulation
\begin{align} \label{eq:mild.formulation}
    c(t,\cdot) - e^{\Delta t}c_0 =-\int_0^t e^{\Delta (t-s)} \nabla \cdot (v(s,\cdot) c(s,\cdot)) ds,
\end{align}
as an equality in $L^q(\R^n)$, for every $t \in [0,T]$.
\end{defi}

In the definition above, we implicitly assume that the product $vc$ is a well-defined distribution on $(0,T)\times \R^n$ and the right-hand side of \eqref{eq:mild.formulation} belongs to $L^q(\R^n)$ for every $t \in [0,T]$.

Equation \eqref{eq:KS_physical} is invariant under the scaling 
\begin{align*}
    c(t,x) \mapsto \lambda^2 c(\lambda^2 t,\lambda x),
    \quad
    \lambda>0.
\end{align*}
Accordingly, the space $C([0,T] ; L^\frac{n}{2}(\R^n))$ is \emph{critical} for the Keller-Segel system \eqref{eq:KS_physical} under this scaling. 
In the \emph{subcritical} regime $q>\frac{n}{2}$, viscous effects dominate the dynamics in \eqref{eq:KS_physical}, and local existence and uniqueness of mild $L^q$-solutions are expected; see, for instance, \cite[Theorem 3.1, Lemma 3.1(ii)]{Ka99} and \cite[Theorem 2]{BiHeNa94} for results specific to \eqref{eq:KS_physical}, and \cite[Theorem 2.1]{PrSiWi18} for general well-posedness results for parabolic equations.
Some well-posedness results are also available in the critical case $q=\frac{n}{2}$, see for instance \cite{BeMa14} in dimension $n=2$ and \cite{KoSuYa12,Fe22} in dimensions $n \geq 3$; however, solutions may blow up in finite time if the critical $L^{\frac{n}{2}}(\R^n)$-norm of the initial datum is sufficiently large \cite{CoPeZa04}, \cite[Corollary 5.1.1]{BiHeNa94}; see also \cite{CaCoEb12,KiXu16,SoWi19,naito2021blow,GlSc24,HuKi24,NgNoZa25,HuKiYa25} for further blow-up/no blow-up criteria and results.

On the other hand, in the \emph{supercritical} regime $q<\frac{n}{2}$, nonlinear effects dominate and ill-posedness is expected for \eqref{eq:KS_physical}. However, to the best of our knowledge no example of non-uniqueness is available in the literature, neither in  $C([0,T] ; L^q(\R^n))$ nor in other supercritical spaces. In this paper, we prove existence and non-uniqueness of mild $L^q$-solutions for every supercritical exponent. More precisely we show that:
\begin{theorem} \label{thm:non.uniqueness}
Let $n \in \{3,\dots,9\}$. Then, for every $q \in [1,n/2)$, there exist nontrivial $c_0 \in L^q(\R^n)$, $T \in (0,\infty)$, and two distinct mild $L^q$-solutions of \eqref{eq:KS_physical} in the sense of \autoref{def_mild_sol} with initial condition $c_0$.
\end{theorem}

Moreover, we are able to construct non-unique weak solutions to \eqref{eq:KS_physical} in a scaling critical space, up to replacing the Lebesgue $L^{n/2}(\R^n)$ with the (scaling equivalent) Lorentz space $L^{n/2,\infty}(\R^n)$. Specifically we can construct distinct solutions in $L^\infty([0,T];L^{n/2,\infty}(\R^n))$. Here the lack of time continuity is not surprising, as the heat semigroup is not strongly continuous on the Lorentz space.

It is worth to mention that non-uniqueness in \emph{any} critical space is completely new for the Keller-Segel system in dimension $n \geq 3$ (see \cite{LuSuVe12} for a result on measure-valued solutions in dimension $n=2$), and our result is even more interesting if one considers that uniqueness of small solutions in $L^\infty([0,T];L^{n/2,\infty}(\R^n))$ is expected by fixed-point arguments à la Koch-Tataru \cite{koch2001well}, cf. the discussion at \cite[Page 4 and Remark 3.2 (ii)]{Fe22}. 
We have the following:
\begin{theorem}
    \label{thm:critical}
    Let $n \in \{3,\dots,9\}$. Then there exist nontrivial $c_0 \in L^{n/2,\infty}(\R^n)$, $T \in (0,\infty)$, and two distinct weak solutions of \eqref{eq:KS_physical} in $L^\infty([0,T];L^{n/2,\infty}(\R^n))$ attaining the initial datum $c_0$ in the sense of distributions.
\end{theorem} 

The proof of \autoref{thm:non.uniqueness} and \autoref{thm:critical} follows the program developed by Jia and Šverák in \cite{JiSv15} to prove non-uniqueness of Leray-Hopf weak solutions to 3D Navier-Stokes equations. Although the latter are different equations than the Keller-Segel system, the two models share some similarities: $c$ behaves like the Navier-Stokes vorticity in terms of scaling, the equations are nonlocal, and the nonlinearity depends on the gradient of the solution.

The program of \cite{JiSv15} consists of two fundamental (and relatively independent) blocks, that in our context look as follows.

The first is a form of \emph{spectral instability} for self-similar solutions. Rewrite \eqref{eq:KS_physical} in similarity variables $\Xi(\tau,\theta) := e^\tau c(e^\tau,e^{\tau/2}\theta)$ to get:
\begin{align} \label{eq:Xi}
    \partial_\tau \Xi &= \Delta \Xi + \frac{\theta}{2} \cdot \nabla \Xi + \Xi - \nabla \Xi \cdot \nabla (-\Delta)^{-1}\Xi + \Xi^2,
    \quad 
    (\tau,\theta) \in (-\infty,\log T] \times \R^n,
\end{align}
and linearize \eqref{eq:Xi} around a given $\tau$-independent profile $\bar \Xi$. The resulting equation for the perturbation $\Phi(\tau,\theta) := \Xi(\tau,\theta)-\bar \Xi(\theta)$ reads as:
\begin{align} \label{eq:Phi}
    \partial_\tau \Phi &= \Delta \Phi + \frac{\theta}{2} \cdot \nabla \Phi + \Phi - \nabla \bar\Xi \cdot \nabla (-\Delta)^{-1}\Phi - \nabla \Phi \cdot \nabla (-\Delta)^{-1} \bar\Xi + 2 \bar\Xi \Phi
    =: L_{\bar\Xi}[\Phi].
\end{align}
Then one expects that for small $\Phi$ with respect to a suitable Banach norm, the behavior of solutions to \eqref{eq:Xi} is determined by the linearized equation above, and more specifically by the spectral properties of the linearized operator $L_{\bar\Xi}$. In particular, existence of an eigenvalue of $L_{\bar\Xi}$ with positive real part and sufficiently fast decay at infinity of the corresponding eigenfunction are a strong indication of non-uniqueness of solutions to \eqref{eq:Xi}.

The main obstacle to the application of the Jia and Šverák program is, in fact, the verification of this spectral instability.
For the Navier-Stokes equations specifically, the first numerical evidence of this fact has been given in \cite{GuSv23}; 
building upon the groundbreaking work of Vishik \cite{vishik2018instability,vishik2018instability2} (see also \cite{lecturenotes24,Ca+25,DoMe25}), Albritton, Bru\`e and Colombo gave in \cite{ABC22} the first rigorous proof of spectral instability for the Navier-Stokes equation with an external forcing, made \emph{ad hoc} to help with the construction of unstable self-similar profiles, cf. also \cite{ABC23} on bounded domains; a computer assisted proof of spectral instability in the unforced case has recently been claimed by Hou, Wang, and Yang in \cite{HoWaYa25}. 

This work demonstrates that the program of Jia and Šverák can successfully be implemented for the Keller-Segel system. We are able to rigorously prove spectral instability for the linearization around a self-similar, radially symmetric solution of \eqref{eq:KS_physical}. Unlike the previous works, our argument is entirely analytic and does not require the introduction of external forcings into the equation or computer-assisted methods. We have the following: 
\begin{theorem} \label{thm:instability.intro}
Let $n \in \{3,\dots,9\}$ and fix $\eta \in (1,n/2)$, $\gamma \in [\eta,\infty)$, and $\bar\lambda>0$. Then there exists a $\tau$-independent profile $\bar \Xi$ such that the linearized operator $L_{\bar\Xi}$ admits a realization on the space $L^{\eta}_{\mathrm{rad}}(\R^n) \cap L^{\gamma}_{\mathrm{rad}}(\R^n)$ of radially symmetric functions with a positive eigenvalue $\lambda \in (0,\bar\lambda]$, and $\sigma(L_{\bar\Xi}) \subset \{ \mu\in \CC \, : \, \mathrm{Re}(\mu) \leq \lambda\}$.
Moreover, the eigenfunction corresponding to $\lambda$ decays at infinity exponentially fast.
\end{theorem}

For the proof of this result, in \autoref{sec:spectral} we preliminarily identify the eigenvalues of $L_{\bar\Xi}$ having positive real part with the eigenvalues of an auxiliary self-adjoint operator on a weighted $L^2(\R^{n+2})$ space. The latter can be ``counted'' using Sturm-Liouville oscillation theory as zeros of an associated ODE, which is studied in \autoref{sec:ODE}. Counting these zeros is where the main difficulties reside (as one might expect): for our approach we need some tools from dynamical systems to quantitatively analyze the ``tail'' of a properly rescaled self-similar profile, and then we invoke a comparison result for ODEs due to Picone \cite{picone1910sui}.
In high dimension $n>9$ and in dimension $n=2$, our approach fails and we suspect that radially symmetric self-similar solutions are, in fact, stable\footnote{ This is the case for the nonlinear heat equation, see \cite[Point 3 of Theorem 2.5]{GlHoLaLu25}. }.
As far as we know, the only other example of PDE where the program of Jia and Šverák has been carried out successfully (without forcing nor computer assistance) is the nonlinear heat equation by the first author and his collaborators in \cite{GlHoLaLu25}.

\autoref{thm:instability.intro} above is already sufficient to produce two distinct weak solutions to \eqref{eq:KS_physical} in the critical space $L^\infty([0,T];L^{n/2,\infty}(\R^n))$, thus proving \autoref{thm:critical}.
Indeed, spectral instability produces two distinct ``ancient'' solutions (called in this way since they are defined for times $\tau \in (-\infty,\log T]$ in similarity variables), one corresponding to $\Phi \equiv 0$ and the other obtained from the unstable eigenfunction of $L_{\bar \Xi}$, see \autoref{sec:strategy} below and the construction of \autoref{sec:ancient}.

On the other hand, these solutions do not decay fast enough at spatial infinity to lie in the $L^q(\R^n)$ class, and additional work is needed in order to prove \autoref{thm:non.uniqueness}.  
The second block of the Jia-Šverák program is a \emph{localization} procedure. Indeed, a localization of the initial condition is necessary in order to gain the desired integrability stated in \autoref{thm:non.uniqueness}. As mentioned in \cite{JiSv15}: \emph{The reason why this is not straightforward is that the equations for the ``remainders'' contain critically singular lower order terms which cannot be assumed to be small}.
We deal with the localization problem in \autoref{sec:localization}.
We also provide a variant of \autoref{thm:instability.intro} (see \autoref{thm:instability}) that somehow captures the ``minimal'' requirements to produce non-unique $L^q$-valued solutions of \eqref{eq:KS_physical} and in general simplifies the localization procedure; it also allows to cover the case $q=1$ that is left out in \autoref{thm:instability.intro}. 

\subsection{Overview of the arguments}\label{sec:strategy}
Before going to the actual proofs, let us first give an overview of our arguments that we think can help the reader orientate in the forthcoming sections. In the following, given a function $F$ that depends on time and space variables, we will often denote $F(t,\cdot)$ simply by $F(t)$.

\subsubsection{Reduced mass equation}
The two distinct mild $L^q$-solutions that we construct to prove \autoref{thm:non.uniqueness} have radial symmetry, namely $c(t,x)=c(t,x')$ whenever $|x|=|x'|$.
A main feature of \eqref{eq:KS_physical} is that it can be manipulated with a linear transformation $w := A[c]$ in such a way that a radially symmetric solution $c$ is mapped into a function $w$ solving a local PDE.
The precise transformation has already been used for instance in \cite{brenner1999diffusion} and is given by
\begin{align}
    A[c](t,y) := \frac{1}{2\omega_{n-1}|y|^n} \int_{B(0,|y|)} c(t,x)dx,
    \quad
    (t,y) \in [0,T] \times \R^{n+2},
\end{align}
where $\omega_{n-1}$ denotes the volume of the $(n-1)$-dimensional sphere; for radially symmetric $c(t,\cdot)$, the \emph{reduced mass} function $w(t,\cdot)$ is also radially symmetric and satisfies
\begin{align} \label{eq:w}
  \partial_t w = \Delta w + y\cdot \nabla (w^2) + 2n w^2,
  \quad
  (t,y) \in [0,T] \times \R^{n+2}.
\end{align}

One can try to investigate spectral instability of \eqref{eq:w} too, using the Jia-Šverák approach described above. 
Self-similar solutions $\Omega(\tau,\xi) := e^\tau w(e^\tau,e^{\tau/2}\xi)$ satisfy
\begin{align} 
\label{eq:Omega}
  \partial_\tau \Omega &= \Delta \Omega + \frac{\xi}{2} \cdot \nabla \Omega + \Omega + \xi \cdot \nabla (\Omega^2) + 2n\Omega^2,
    \quad
    (\tau,\xi) \in (-\infty,\log T] \times \R^{n+2}.
\end{align}
The equation for the perturbation $\Psi(\tau,\xi)=\Omega(\tau,\xi)-\bar\Omega(\xi)$ analog to \eqref{eq:Phi} is:
\begin{align} \label{eq:Psi}
    \partial_\tau \Psi &= \Delta \Psi + \frac{\xi}{2} \cdot \nabla \Psi + \Psi + 2\xi \cdot \nabla (\bar \Omega \Psi) + 4n\bar\Omega \Psi =: L_{\bar \Omega}[\Psi].
\end{align}

Denote $\tilde{A}$ the operator so that $\tilde{A}[\Xi]=\Omega$ for every $w=A[c]$. A quick computation reveals that $\tilde{A}$ and $A$ have the same formal expression: 
\begin{align}
    \tilde{A}[\Xi](\tau,\xi)
    =
    \frac{1}{2\omega_{n-1}|\xi|^n} \int_{B(0,|\xi|)} \Xi(\tau,\theta)d\theta.
\end{align}
Hereafter we identify $A=\tilde{A}$ with a little abuse of notation.

The relevant spaces for the study of spectral instability of the linearized equations \eqref{eq:Phi} and \eqref{eq:Psi} are the (weighted) Lebesgue spaces of radially symmetric functions, defined for $q \geq 1$:
\begin{align}
  X^q := L^q_{\mathrm{rad}}(\R^n,d\theta),
  \quad
  Y^q := L^q_{\mathrm{rad}}(\R^{n+2},|\xi|^{-2}d\xi),
  \quad
  Y_{\nabla}^q :=\{ \Omega \in  Y^q : \xi \cdot \nabla \Omega\in Y^q \},
\end{align}
and their intersections 
\begin{align}
  X^{\eta,\gamma} := X^\eta \cap X^\gamma,
  \quad
  Y^{\eta,\gamma} := Y^\eta \cap Y^\gamma,
  \quad
Y_\nabla^{\eta,\gamma} := Y_\nabla^\eta \cap Y_\nabla^\gamma.
\end{align}
It turns out (\autoref{lem:isomorphism}) that the map $A : X^{\eta,\gamma} \to Y_\nabla^{\eta,\gamma}$ is a linear isomorphism of Banach spaces for every $\eta,\gamma \in (1,\infty)$ and therefore most of the spectral properties of \eqref{eq:Phi} can be equivalently formulated in terms of \eqref{eq:Psi}. 
Moreover, the spectral analysis of \eqref{eq:Psi} can be carried out also in the borderline case $\eta=1$ where $A$ is not an isomorphism but $A^{-1}$ is well-defined and continuous, producing non-uniqueness of $L^1$-valued mild solutions to the Keller-Segel system in ``physical'' variables \eqref{eq:KS_physical} while bypassing a proof of spectral instability in $X^{1,\gamma}$. Therefore, from now on we will focus on studying properties of \eqref{eq:w}, deducing the corresponding results for \eqref{eq:KS_physical} by the map $A^{-1}$.

\subsubsection{Spectral instability of $L_{\bar\Omega}$} 
Being radially symmetric and $\tau$-independent, the profile $\bar\Omega$ around which the linearized equation \eqref{eq:Psi} is defined is uniquely determined by the parameter $\alpha := \bar\Omega(0)$, in which case we define for simplicity $L_{\alpha} := L_{\bar \Omega}$.

Hence, spectral instability of $L_{\bar\Omega}$ can be reformulated as the problem of finding a value of $\alpha \in \R$ such that the operator $L_{\alpha}$ is spectrally unstable (in suitable Banach spaces).

Let us describe our strategy to show this.
Notice that for radially symmetric objects like $\bar \Omega$ and $\Psi$ from \eqref{eq:Psi}, we can equivalently study the evolution of their radial profile.
For instance, $u_\alpha(|\xi|):=\bar \Omega (\xi)$ with $\bar{\Omega}(0)=\alpha$ solves the second order ODE 
\begin{align}\label{self-similar_ode_intro}
\begin{dcases}  u_\alpha'' 
+ 
\frac{n+1}{\rho}  u_\alpha' 
+ \frac{\rho}{2}  u_\alpha' + 
 u_\alpha  + 2 \rho  u_\alpha u_\alpha' + 2n  u_\alpha^2= 0, 
 \quad
 \rho > 0,
 \\
 u_\alpha(0)=\alpha,
 \\
 u_\alpha'(0)=0,
 \end{dcases}
\end{align}
and if $f(|\xi|)=\Psi(\xi)$ then $L_\alpha \Psi(\xi)$ equals, for $|\xi|=\rho$:
\begin{align}
L_\alpha \Psi =    f'' 
+ 
\frac{n+1}{\rho} f' +  \frac{\rho}{2}f' 
+ 
f +  2 \rho u_\alpha'  f +2 \rho u_\alpha  f' 
+ 4n u_\alpha f .
\end{align}
The key feature of the expression above is that we can rewrite the right-hand side in self-adjoint form if we consider functions $f$ living in a weighted space $L^2_{\pi,\mathrm{rad}}$.
This produces a new operator $L^{\pi}_\alpha$, with domain $D(L^\pi_\alpha) \subset L^2_{\pi,\mathrm{rad}}$, which: $i)$ coincides with $L_\alpha$ on smooth functions decaying at infinity sufficiently fast; and $ii)$ can be thoroughly studied by means of Sturm–Liouville theory, for instance to determine the number of positive eigenvalues. 
If one is able to establish a link between spectral instability of $L_\alpha$ and $L^\pi_\alpha$, then the door is open to apply results for self-adjoint Sturm-Liouville operators to the (non-self-adjoint!) operator $L_\alpha$.
As an example, in this paper we prove that the spectrum of $L_\alpha$ to the right of the imaginary axis coincides with that of $L_\alpha^\pi$ and is made of eigenvalues, see \autoref{cor:eigenvalues} and the discussion at the beginning of \autoref{subsec_unstable_eigen}, and therefore spectral instability of $L_\alpha$ and $L^\pi_\alpha$ are equivalent.

To show the existence of an unstable eigenvalue for $L^\pi_\alpha$, by Sturm-Liouville theory one can alternatively find a value of $\alpha$ such that the unique solution of 
\begin{align}
\begin{dcases}
f'' 
+ 
\frac{n+1}{\rho} f' +  \frac{\rho}{2}f' 
+ 
f +  2 \rho u_\alpha'  f +2 \rho u_\alpha  f' 
+ 4n u_\alpha f
=0,
\\
f(0)=1,
\\
f'(0)=0,
\end{dcases}
\end{align}
vanishes at some point $\rho_0>0$. 
So spectral instability boils down to studying zeros of an ODE, which is the content of \autoref{sec:ODE} and more specifically \autoref{prop:f.zero}.
The reader should not think that this procedure ``trivializes'' the problem: since the profile $u_\alpha$ is not explicit, finding zeros of $f$ is not an easy task.
For instance, to show spectral instability for the nonlinear heat equation, the authors of \cite{GlHoLaLu25} had to rely on several results from previous papers \cite{HaWe82,Naito} despite the fact that their equation is simpler because of the missing gradient nonlinearity $\rho u_\alpha'  f + \rho u_\alpha  f'$; in addition, exactly because of this additional term in the equation, the aforementioned results cannot be extended to our setting and we need to develop a new approach to study the zeros of $f$ that relies on Picone comparison principle and some theory of dynamical systems. 

Not every value of $\alpha \in \R$ leads to spectral instability for Keller-Segel. For example, in the proof of \autoref{thm:instability} we show that if $\alpha \in [0,\frac{1}{16n})$ then $f(\rho)>0$ for every $\rho \geq 0$ and the operator $L_\alpha$ is stable.
In fact, our proof of spectral instability holds for every $\alpha$ larger than an implicit value $\alpha_0=\alpha_0(n)$ depending only on the physical dimension $n$.

Let us explain why this is the case.
What we do is to rescale the parameter $\rho$ in the following way
\begin{align}
    \tilde u_\alpha(\rho) := \frac{u_\alpha(\rho/\sqrt{\alpha})}{\alpha},
    \quad
    \tilde{f}_\alpha(\rho)
    :=
    f(\rho/\sqrt{\alpha}),
\end{align}
so that $\tilde u_\alpha$, $\tilde f_\alpha$ solve the system
\begin{align}
\begin{dcases}
%
%
\tilde{u}_\alpha'' 
+ 
\frac{n+1}{\rho} \tilde{u}_\alpha' 
+ 
\frac{\rho}{2\alpha} \tilde{u}_\alpha'
+ 
\frac{1}{\alpha} \tilde{u}_\alpha  
+ 
2 \rho \tilde{u}_\alpha \tilde{u}_\alpha'
+ 
2n \tilde{u}_\alpha^2 
 = 0,
\\
%
%
\tilde{f}_\alpha'' 
+ 
\frac{n+1}{\rho} \tilde{f}_\alpha' 
+ 
\frac{\rho}{2\alpha} \tilde{f}_\alpha'
+ 
\frac{1}{\alpha}\tilde{f}_\alpha 
+ 
2 \rho \tilde{u}_\alpha \tilde{f}_\alpha'
+ 
2 \rho \tilde{f}_\alpha \tilde{u}_\alpha' 
+ 
4n \tilde{u}_\alpha \tilde{f}_\alpha 
= 0,
\end{dcases}
\end{align}
and in the limit $\alpha \to \infty$ we get the limit ODE system
\begin{align} 
\begin{dcases}
%
%
\tilde{u}'' 
+ 
\frac{n+1}{\rho} \tilde{u}' 
+ 
2 \rho \tilde{u} \tilde{u}' 
+ 
2n \tilde{u}^2 
= 0,
\\
%
%
\tilde{f}'' 
+ 
\frac{n+1}{\rho} \tilde{f}' 
+ 
2 \rho \tilde{u} \tilde{f}'
+ 
2 \rho \tilde{f} \tilde{u}'
+ 
4n \tilde{u} \tilde{f} 
 = 0.
\end{dcases}
\end{align}
Since $\tilde{f}_\alpha \to \tilde{f}$ uniformly on compacts and $\tilde{f}$ is not  identically zero, if we find a positive root for $\tilde{f}$ then $f_\alpha$ will vanish somewhere for every $\alpha$ sufficiently large.
By ODE analysis of \autoref{sec:ODE}, $u_\alpha(\rho) \sim \rho^{-2}$ and $u'_\alpha(\rho) \sim \rho^{-3}$ for $\rho \gg 1$.
The reason why showing the existence of roots of $\tilde{f}$ is easier than doing the same for $f$ is that, after rescaling, the dominant term $\frac{\rho}{2} f'+f$ has disappeared from the equation and the equation can be rigorously ``compared'' to a homogeneous equation of the form $y'' + \frac{A}{\rho} y' + \frac{B}{\rho^2}y=0$ with Picone comparison principle \cite{picone1910sui}. Indeed, if the condition $(A-1)^2-4B<0$ holds, this equation has oscillating solutions of the form
\begin{align}
    y(\rho) =  C_1 |\rho|^\frac{1-A}{2}\sin(\mu \log|\rho|)
    +
    C_2 |\rho|^\frac{1-A}{2}\cos(\mu \log|\rho|) ,
    \quad
    \mu:= \frac12 |(A-1)^2-4B|^{1/2}.
\end{align}
The condition above on $A,B$ is ultimately responsible for the constraint $n \in \{3,\dots,9\}$ on the physical dimensions. The precise numbers $A,B$ with which to apply the comparison principle are determined by the values of the limits
\begin{align}
    \lim_{\rho \to \infty} \rho^2 \tilde{u}(\rho),
    \quad
    \lim_{\rho \to \infty} \rho^3 \tilde{u}'(\rho),
\end{align}
that we study in \autoref{ssec:tail}.
It turns out that the quantity $z(t) := e^{2t} \tilde{u}(e^{t})$ solves the second-order autonomous ODE 
\begin{align} \label{eq:Emden.intro}
    \ddot{z} + (n-4+2z) \dot{z} -  2(n-2)(z-z^2) =0, \quad
   t \in \R,
\end{align}
whose limit behavior as $t \to \infty$ can be studied with tools from two-dimensional dynamical systems, such as Poincaré-Bendixson and Bendixson–Dulac theorems. 
The analysis of \eqref{eq:Emden.intro} is particularly rich when $n=3$ (which in fact requires a different proof than the case $n \geq 4$), as hinted by the investigation of the so-called \emph{termination points} of Emden's problem done in \cite{JoLu73}; as a matter of fact, \eqref{eq:Emden.intro} alone is not ``enough'' to uniquely determine the limit behavior of $(z,\dot{z})$ when $t \to \infty$, and in our proof we \emph{must} incorporate the information $\tilde{u} = \lim_{\alpha \to \infty} \tilde{u}_\alpha$.  
We remark that studying the tail behavior of $z_\alpha(t) := e^{2t} \tilde{u}_\alpha(e^{t})$ with finite $\alpha$ would be much more difficult, since the resulting ODE for $z_\alpha(t)$ is non-autonomous.

\subsubsection{Construction of the non-uniqueness} Building on the spectral instability of $L_{\alpha}$ for suitable choices of $\alpha>0$, one can construct two distinct solutions of \eqref{eq:w}, and consequently of \eqref{eq:KS_physical}, as described below.

Recall the self-similar change of variables $\Omega(\tau,\xi) := e^\tau w(e^\tau,e^{\tau/2}\xi)$ and that we aim to establish non-uniqueness for \eqref{eq:Omega}. For each $\alpha>0$ and profile $u_\alpha$ solving \eqref{self-similar_ode_intro}, the function $\bar\Omega := u_\alpha(|\cdot|)$ is a stationary solution of \eqref{eq:Omega}. The key idea in the program of Jia and Šverák is that, if $L_{\alpha}$ has an eigenvalue with positive real part, then it is possible to construct a second solution $\tilde{\Omega}(\tau)$ of \eqref{eq:Omega} satisfying 
\begin{align*}
    \tilde{\Omega}(\tau)\rightarrow \bar\Omega,
    \quad\text{as }\tau\rightarrow-\infty.
\end{align*}
This construction is carried out rigorously in \autoref{sec:ancient}; here we only outline how it is possible. In our setting, the part of $\sigma(L_{\alpha})$ lying in $\{z \in \CC \,:\, \mathrm{Re}(z) \geq 0\}$ is made by only finitely many real eigenvalues. Let $\lambda_{\alpha}$ denote the largest one, and let $\bar\Omega^{lin}$ be a corresponding eigenfunction. Both $\bar\Omega^{lin}$ and its gradient decay exponentially in space; in particular $\bar\Omega^{lin}\in Y^{q,r}_{\nabla}$ for every $r>\frac{n}{2}$. We look for $\tilde{\Omega}$ of the form
\begin{align*}
    \tilde{\Omega}(\tau):=\bar\Omega+\Psi(\tau):=\bar\Omega+e^{\lambda_{\alpha}\tau}\bar\Omega^{lin}+\Omega^{per}.
\end{align*}
Clearly,
\begin{align*}
  \Omega^{lin}(\tau):= e^{\lambda_{\alpha}\tau}\bar\Omega^{lin}\rightarrow  0,
  \quad\text{as }\tau\rightarrow-\infty,
\end{align*}
while $\Omega^{per}$ is required to solve:
\begin{align}\label{system_ancient_intro}
    \begin{dcases}
\partial_{\tau}\Omega^{per}=
L_{\alpha}\Omega^{per}+N(\Omega^{per})+N(\Omega^{lin})+V(\Omega^{per}),
&\text{for }\tau \in (-\infty,\log T],
\\
\Omega^{per}(\tau)\rightarrow 0, &\text{as }\tau\rightarrow-\infty,
    \end{dcases}
\end{align}
where we define
\begin{align}\label{nonlinearity_ancient}
    N(\omega)&:=2n \omega^2+\xi\cdot\nabla (\omega^2),
    \quad 
    V(\omega):=4n \omega\Omega^{lin}+2\xi\cdot\nabla (\omega\Omega^{lin}).
\end{align}
Since $L_{\alpha}$ generates a semigroup on $Y^{q,r}$, up to taking $T$ sufficiently small one can construct solutions to \eqref{system_ancient_intro} taking values in $Y^{q,r}_{\nabla}$ via a fixed-point argument, by analyzing jointly the evolution of $\Omega^{per}$ and $|\xi|\Omega^{per}$. Moreover, thanks to the regularizing properties of the semigroup $L_{\alpha}$ on $Y^{q,r}$ analogous to those of the heat semigroup on $X^{q,r}$ (see \autoref{cor:S_alpha}), it is possible to show that
\begin{align*}
    \Omega^{per}(\tau),\  |\xi|\Omega^{per}(\tau)\in Y^{q,r}_1:=Y_1^q\cap Y_1^r,
\end{align*} 
where, for $p\in [1,+\infty)$, we have denoted
\begin{align*}
   Y_1^p:=\left\{f\in Y^{p}:\, \int_{\R^{n+2}}\frac{|\nabla f(y)|^p}{|y|^2}dy<+\infty\right\}
\end{align*}
endowed with its natural norm. A naive approach would be to study directly the evolution of $\Omega^{per}$ in $Y^{q,r}_{\nabla}$, without introducing the additional quantity $|\xi|\Omega^{per}$. However, it is unclear whether this strategy can succeed, since the nonlinearity $N(\cdot)$ requires control not only of $\Omega^{per}$ and its derivatives, but also of their spatially weighted counterparts. In particular, controlling $N(\Omega^{per})$ in $Y_{\nabla}^{q,r}$ necessitates bounds on $|\xi|\Omega^{per}$. Moreover, our approach is \emph{minimal} in terms of spectral instability assumptions. The fixed-point argument only requires instability in $Y^{q,r}$, rather than $Y^{q,r}_{\nabla}$; the additional regularity is recovered a posteriori from the equation satisfied by $\Omega^{per}$ and from the regularizing properties of the semigroup generated by $L_{\alpha}$. This feature is important not only in light of the discussion above, but also because it allows us to cover the left endpoint $q=1$; indeed, the construction of ancient solutions directly for \eqref{eq:Phi}, instead of \eqref{eq:Omega}, would require a spectral instability for \eqref{eq:Phi} in the space $X^{1,r}$, which is left out by \autoref{thm:instability.intro} since $A$ is not an isomorphism between $X^1$ and $Y_{\nabla}^1$ (essentially due to the failure of Hardy inequality in $L^1$).

Most importantly, due to the quadratic structure of the right-hand side of \eqref{system_ancient_intro}, one can show that, in suitable function spaces,
\begin{align}\label{eq:decaying}
    \|\Omega^{per}(\tau)\|\ll \|\Omega^{lin}(\tau)\|, \quad \text{as }\tau\rightarrow -\infty,
\end{align}
which implies 
\begin{align*}
    \tilde{\Omega}(\tau)\rightarrow \bar{\Omega},\quad \text{as }\tau\rightarrow -\infty, \quad \mbox{and }\tilde{\Omega}\neq  \bar{\Omega}.
\end{align*}
Inverting the self-similar change of variables yields two radial distributional solutions $\tilde{w}_1,\tilde{w}_2$ of \eqref{eq:w} arising from the same initial datum
\begin{align*}
    \tilde{w}_0(y):=\frac{\ell_{\alpha}}{|y|^2},\quad \text{for some }\ell_{\alpha}>0.
\end{align*}
One solution is the self-similar expander
\begin{align*}
   \tilde{w}_1(t,y):=\frac{1}{t}\bar{\Omega}\left(\frac{y}{\sqrt{t}}\right), 
\end{align*}
while the other one is obtained via the perturbation $\Psi$ with the formula: 
\begin{align*}
    \tilde{w}_2(t,y):=\frac{1}{t}\bar{\Omega}\left(\frac{y}{\sqrt{t}}\right)+\frac{1}{t}\Psi\left(\log t,\frac{y}{\sqrt{t}}\right).
\end{align*}

By applying the map $A^{-1}$, one can check that $\tilde{w}_1, \tilde{w}_2$ give two distinct solutions satisfying the requirements of \autoref{thm:critical}. We refer to \autoref{ssec:proof.critical} for details. 

Notice, however, that the initial datum $\tilde{w}_0$ decays too slowly as $|y|\rightarrow +\infty$ to have the desired integrability stated in \autoref{thm:non.uniqueness}; in particular $\tilde{w}_0\notin Y_{\nabla}^q$. Therefore, $\tilde{w}_1, \tilde{w}_2$ cannot directly yield two distinct mild $L^q$-solutions of \eqref{eq:KS_physical} via the map $A^{-1}$. 
To enforce integrability and prove \autoref{thm:non.uniqueness}, we modify $ \tilde{w}_0$ far from the origin and deform the corresponding solutions accordingly. This is carried out in \autoref{subsec:localization_PDE}. We decompose $\tilde{w}_0=w_0+h_0$, where $w_0\in Y_{\nabla}^q$ with $w_0(y) = \frac{\ell_\alpha}{ |y|^{2}}$ close to the origin, and the remainder $h_0$ satisfies  $h_0\in Y_1^r\cap Y^r_{\nabla}$ and $|y|h_0\in Y^p$ for all $p>n$. We then analyze the evolution of $h_0$ in order to subtract it from the two solutions above, thereby obtaining two solutions $w_1,w_2$ of \eqref{eq:w} with common initial datum $w_0$. The corresponding Cauchy problem is 
	\begin{align}\label{nonlinear PDE_intro}
		\begin{dcases}
			\partial_t h=\Delta h+4n\bar{w}h+2y\cdot\nabla(\bar{w}h)+f(h),\\
			h(0)=h_0,
		\end{dcases}
	\end{align}
	where the potential term is self-similar, i.e.
	\begin{align*}
		\bar{w}(t,y)=\frac{1}{t}\bar{\Omega}\left(\frac{y}{\sqrt{t}}\right),
	\end{align*}
	and the forcing term $f$ depends on an auxiliary function $\psi$ and is given by
	\begin{align*}
		f(h)&:=4n\psi h+2y\cdot\nabla(\psi h)-2nh^2-y\cdot \nabla (h^2).
	\end{align*}
	We denote by $h_1$ the solution to \eqref{nonlinear PDE_intro} corresponding to $\psi=0$, and by $h_2$ the one corresponding to \begin{align*}
	    \psi(t,y)=\frac{1}{t}\Psi\left(\log t,\frac{y}{\sqrt{t}}\right).
	\end{align*} 
    Constructing such solutions is highly nontrivial, due to the time-dependent and singular nature of the potential term at time $t=0$. 
    Nevertheless, we develop a well-posedness theory for problems of the form \eqref{nonlinear PDE_intro}. In particular, we prove local existence and uniqueness in a subspace contained in $Y^r_{\nabla}$ provided that the maximal eigenvalue $\lambda_{\alpha}$ is small enough, namely:
	\begin{equation}\label{Eq:la_bar}
		\lambda_{\alpha} < 1-\frac{n}{2r}.
	\end{equation}
	Note that the right-hand side of \eqref{Eq:la_bar} can be arbitrarily small, forcing correspondingly small values of $\lambda_{\alpha}$. Values of $\alpha$ for which \eqref{Eq:la_bar} holds are available by our \autoref{thm:instability}. As in the construction of $\Omega^{per}$, the analysis requires control not only of $h$, but also of $|y|h$, which is possible since $|y|h_0\in Y^p$ for $p>n.$\\
    Finally, in \autoref{subsec:end_proof}, we construct two solutions $c_1,c_2$ to the Keller-Segel system \eqref{eq:KS_physical} from 
    \begin{align*}
        w_1:=\tilde{w}_1-h_1,\quad w_2:=\tilde{w}_2-h_2,
    \end{align*}
    via the operator $A^{-1}$, and show that they are mild $L^q$-solutions of \eqref{eq:KS_physical} with common initial datum $c_0\in L^q(\R^n)$. At this stage, it suffices that $A^{-1}$ is bounded from $Y^q_{\nabla}$ to $X^q$, which holds also for $q=1$. 
    Moreover, using a quantitative version of the decay estimate \eqref{eq:decaying} obtained in \autoref{sec:ancient}, we conclude that $c_1\neq c_2,$ thereby completing the proof.
\subsubsection{Final Remarks}
As is apparent from the strategy outlined above, and as will become even clearer in the forthcoming sections, \autoref{thm:non.uniqueness} can be readily extended to cover more extreme cases of non-uniqueness for \eqref{eq:KS_physical}, both in terms of the number of distinct solutions arising from the same initial datum and in terms of the \emph{genericity} of such data.\\
Given $k\in \mathbb{N}$ and a collection $\{a_j\}_{j\in \{1,\dots,k\}}\subset \R_{>0}$ such that $a_j\neq a_{j'} $ whenever $j\neq j'$, it is possible to find $T'<0$ and $k$ distinct solutions of \eqref{eq:Omega} on $(-\infty,T']$ of the form
\begin{align*}
\tilde{\Omega}_j(\tau):=\bar\Omega+\Psi_j(\tau):=\bar\Omega+\Omega^{lin}_j+\Omega^{per}_j,\quad \text{ where}\quad \Omega^{lin}_j(\tau):=a_j e^{\lambda_{\alpha}\tau}\bar\Omega^{lin},
\end{align*}
and $\Omega^{per}_j$ solves \eqref{system_ancient_intro} with $\Omega^{lin}$ replaced by $\Omega^{lin}_j$. The localization procedure described in the previous subsection then produces $k$ different local mild $L^q$-solutions of \eqref{eq:KS_physical}. A similar remark holds for \autoref{thm:critical} in the Lorentz case $L^{n/2,\infty}(\R^n)$, without requiring any localization. \\
The initial datum $c_0$ leading to the non-uniqueness is radial and behaves like
\begin{align*}
    c_0(x)=\frac{2\ell_{\alpha}(n-2)}{|x|^2}
\end{align*}
near $x=0$, while remaining uniformly bounded elsewhere. Moreover, non-uniqueness persists for every initial condition $\tilde{c}_0$ of the form
\begin{align*}
    \tilde{c}_0(x)=c_0(x)+\bar{c}_0(x), 
\end{align*}
where $\bar{c}_0\in X^{q,r}$ and $A[\bar{c}_0]$ satisfies the assumptions required in \autoref{thm:nonlinear_localization}; this is achieved by simply incorporating $A[\bar{c}_0]$  into the initial condition of \eqref{nonlinear PDE_intro}. This holds, for instance, if $\bar{c}_0\in C^{\infty}_c(\R^n)$ and is radial. Such perturbations do not remove the singular behavior of $c_0$ near the origin.

\begin{acknowledgements}
The authors are grateful to Irfan Glogić for suggesting this problem to them and for the interesting discussions.
EL and UP have received funding from the European Research Council (ERC) under the European Union’s Horizon 2020 research and innovation programme (grant agreement No. 949981).
UP has received funding from the Swiss National Science Foundation under the SNSF Ambizione grant No. 233216.
\end{acknowledgements}

\section{ODE analysis} \label{sec:ODE}
This section contains the ODE results needed to show \autoref{prop:f.zero} on the existence of zeros of $f$.
More specifically, in \autoref{subsec:selfsimilarprofile} we study the ODE for the radial profile $u_\alpha$ of the expanding self-similar solution, deriving properties such as smoothness, global existence, and decay at infinity; in \autoref{ssec:linearized.system} we introduce the system of ODEs obtained by linearizing \eqref{eq:ODE} around a profile $u_\alpha$, and we show \autoref{lem:principal_eigen} to characterize the different scenarios in which $f$ remains positive or admits at least a root $\rho_0 \in (0,\infty)$; \autoref{ssec:zeros} is devoted to the proof of \autoref{prop:f.zero}, using \autoref{prop:tails} on the tails of $\tilde{u}$ proved in \autoref{ssec:tail}.

Hereafter we use the  notation $a \lesssim b$ if there is some unimportant finite constant $C$ such that $a \leq Cb$. The symbol $\lesssim_p$ is used to stress that $C=C(p)$ depends upon a parameter $p$. 

\subsection{Profile of expanding self-similar solutions  $u_\alpha$}\label{subsec:selfsimilarprofile}
This subsection is devoted to the analysis of the ODE satisfied by the radial profile $u_\alpha$ of the expanders.
More specifically, for fixed $n \in \NN$, $n \geq 3$, we consider on the semiaxis $\rho \geq 0$:
\begin{align} \label{eq:ODE}
    \begin{dcases}
u_\alpha'' 
+ 
\frac{n+1}{\rho} u_\alpha' 
+ 
u_\alpha + \frac{\rho}{2} u_\alpha' + 2n u_\alpha^2 + 2 \rho u_\alpha u_\alpha' = 0,
\\
u_\alpha(0)=\alpha>0,
\\
u_\alpha'(0) = 0.
    \end{dcases}
\end{align}

\begin{prop} \label{prop:local.existence.uniqueness}
    For every $\alpha>0$ there exists a unique local solution $u_\alpha \in C^2([0,\tau_\alpha))$ to \eqref{eq:ODE} up to a maximal existence time $\tau_\alpha$. $u_\alpha$ is analytic in a right neighborhood of $0$. Moreover, if $I$ is a compact subset of $[0,\tau_{\alpha_\star})$ for some $\alpha_\star >0$, then $I \subset [0,\tau_\alpha)$ for every $\alpha$ in a neighborhood of $\alpha_\star$ and the map $\alpha \mapsto u_\alpha$ is continuous at $\alpha_\star$ with respect to the $C^1$ convergence in $I$.
\end{prop}
\begin{proof}
We recall the following formula, that we will use extensively throughout this section. If we rewrite \eqref{eq:ODE} as $u_\alpha'' +  h  u_\alpha' + g = 0$, where the functions $h = h(\rho)$ and $g = g(\rho)$ are locally integrable, then for any function $H=H(\rho)$ satisfying $H' = hH$ and $H(\rho)>0$ for every $\rho>0$, we can rewrite \eqref{eq:ODE} as $(H u_\alpha')' + H g = 0$. By direct integration between $0 \leq \rho_1 < \rho_2$ one obtains 
\begin{align} \label{eq:HOmega_alpha'}
   H(\rho_2) u_\alpha'(\rho_2) = H(\rho_1)u_\alpha'(\rho_1) - \int_{\rho_1}^{\rho_2} H(\rho) g(\rho) d\rho.
 \end{align}
We apply \eqref{eq:HOmega_alpha'} with $h(\rho) := \frac{n+1}{\rho}$, $g(\rho) := u_\alpha(\rho)+\frac{\rho}{2} u_\alpha'(\rho)+2nu_\alpha^2(\rho) + 2\rho u_\alpha(\rho)u_\alpha'(\rho)$, $H(\rho) := \rho^{n+1}$, and $\rho_1:=0$. Taking the integral on the interval $\rho_2 \in [0,\rho]$ and integrating by parts we get a fixed point formulation
\begin{align}
u_\alpha(\rho)
=
\mathcal{T}[u_\alpha](\rho)
:=
\alpha
&-\int_0^\rho \rho_2^{-n-1}\int_0^{\rho_2}
s^{n+1} \left( 
-\frac{n}{2} u_\alpha(s) 
+ 
(n-2) u_\alpha^2(s)  
\right)
ds d\rho_2
\\
&-
\int_0^\rho \rho_2 \left( \frac{1}{2} u_\alpha(\rho_2)+u_\alpha^2(\rho_2) \right) d\rho_2.
\end{align}

By contraction principle this produces a unique local solution of class $C([0,\overline{\tau}_\alpha])$ for some $\overline{\tau}_\alpha>0$ depending on $\alpha$, and using \eqref{eq:ODE} it is easy to check that $u_\alpha \in C^2([0,\overline{\tau}_\alpha])$. After $\overline{\tau}_\alpha$ the coefficients of \eqref{eq:ODE} are smooth and locally bounded, giving $u_\alpha \in C^2([0,\tau_\alpha))$ up to a maximal existence time $\tau_\alpha \geq \overline{\tau}_\alpha>0$. Moreover, $\tau_\alpha<\infty$ if and only if $\lim_{\rho \uparrow \tau_\alpha}|u_\alpha(\rho)|=\infty$.

To see that $u_\alpha$ is analytic around $0$, let us introduce the power series $\sum_{k}a_k \rho^k$ where $a_0 := \alpha$, $a_1 := 0$, and $\{a_k\}_{k \in \NN}$ is defined by recursion
\begin{align}
    a_{k+2} := -\frac{1}{(k+2)(k+n+2)} \left( \frac{k+2}{2}a_k + 2\sum_{h=0}^k (n+h) a_h a_{k-h} \right).
\end{align}
By direct verification, the power series $\sum_{k} a_k \rho^k$ has a positive convergence radius $R$ and is a fixed point of $\mathcal{T}$. By uniqueness of the fixed point, $u_\alpha(\rho) = \sum_{k}a_k \rho^k$ for every $\rho \in [0,\overline{\tau}_\alpha \wedge R)$.

The second part of the proposition follows again by the fixed point formulation above, in its version with continuous dependence on a parameter $\alpha$, and smoothness of the coefficients of \eqref{eq:ODE} away from $\rho = 0$. More precisely, contraction principle gives $\overline{\tau}_{\alpha_\star}>0$ and local existence and uniqueness of a solution $u_\alpha \in C^2([0,\overline{\tau}_{\alpha_\star}])$ for every $\alpha$ in a neighborhood of $\alpha_\star$, with continuity of the map $\alpha \mapsto u_\alpha$ at the point $\alpha=\alpha_\star$ with respect to the uniform convergence in $[0,\overline{\tau}_{\alpha_\star}]$. 
Then, \eqref{eq:HOmega_alpha'} gives continuity of the map $\alpha \mapsto u_\alpha'$ as well, and in particular $(u_\alpha(\overline{\tau}_{\alpha_\star}), u_\alpha'(\overline{\tau}_{\alpha_\star})) \to (u_{\alpha_\star}(\overline{\tau}_{\alpha_\star}), u_{\alpha_\star}'(\overline{\tau}_{\alpha_\star}))$ as $\alpha \to \alpha_\star$. 
Finally, since the coefficients of \eqref{eq:ODE} are smooth and bounded on $I \setminus [0,\overline{\tau}_{\alpha_\star}]$, up to taking a smaller neighborhood of $\alpha_\star$ we can extend existence, uniqueness and continuous dependence on $\alpha$ to the whole interval $I$.  
\end{proof}

In fact, the maximal existence time $\tau_\alpha$ in the previous proposition is infinite for every $\alpha>0$. To show this, we need to show that the solution $u_\alpha$ does not blow up to either plus or minus infinity.
The first case can be immediately ruled out by observing that the ``energy'' 
\begin{align}
    E_{\alpha} := 
    \frac{(u_\alpha')^2}{2} 
    +
    \frac{u_\alpha^2}{2} 
    +
    \frac{2nu_\alpha^3}{3} 
\end{align}
satisfies $E_\alpha'(\rho) \leq 0$ for every $\rho>0$ such that $u_\alpha(\rho) \geq -1/4$, as can be shown by direct differentiation.
However, in principle $E_\alpha$ is not bounded from below, meaning that blow up to minus infinity cannot be excluded by this kind of considerations and requires more work.
Here we prove the more precise result:
\begin{prop} \label{prop:Omega_alpha>0}
   For any $\alpha > 0$ it holds $u_\alpha(\rho)>0$ and $G_{\alpha}(\rho) := 2nu_{\alpha}(\rho) + 2\rho u_{\alpha}'(\rho) > 0$ for every $\rho \geq 0$.
   Moreover, $u_\alpha(\rho) < \alpha(1+\frac{\alpha}{2}\rho^2)^{-1}$ for every $\rho > 0$ and there exists $\ell_\alpha := \lim_{\rho \to \infty} \rho^2 u_\alpha(\rho) \in (0,2]$.
\end{prop}

Before moving to the proof of \autoref{prop:Omega_alpha>0} we need to collect some auxiliary results.

\begin{lemma} \label{lem:decreasing}
Fix $\alpha>0$ and suppose there exists $\rho_0 \in (0,\tau_\alpha)$ such that $u_\alpha(\rho) \geq 0$ for every $\rho \in [0,\rho_0]$. Then $u_\alpha'(\rho)<0$ for every $\rho \in (0,\rho_0]$.
\end{lemma}
\begin{proof}
    We use \eqref{eq:HOmega_alpha'}, choosing $h(\rho):= \frac{n+1}{\rho} + \frac{\rho}{2}+2\rho u_\alpha(\rho)$ and $g(\rho) := u_\alpha(\rho)+2nu_\alpha^2$, $0<\rho_1 \ll 1$ such that $u_\alpha(\rho) \geq  0$ for every $\rho \in [0,\rho_1]$ and $u_\alpha'(\rho_1)<0$  (notice that this is always possible since $u_\alpha''(0)<0$), and $H(\rho) = \exp \int_{\rho_1}^\rho h(s)ds$.
    Then \eqref{eq:HOmega_alpha'} implies that $u_\alpha'(\rho)<0$ for every $\rho \in [\rho_1,\rho_0]$. We conclude by arbitrariness of $\rho_1$.
\end{proof}

\begin{lemma} \label{lem:bound.rho.m.Omega}
    Fix $\alpha>0$ and suppose $\tau_\alpha = \infty$. If for some $m  \geq 2$ and $C< \infty$ it holds $\sup_{\rho \geq 0}(1+\rho)^m |u_\alpha (\rho)| \leq C$, then $\sup_{\rho \geq 0}(1+\rho)^{m+1} |u_\alpha'(\rho)| \lesssim C+C^2$, with implicit constant independent of $\alpha$.
    Moreover, if $\lim_{\rho \to \infty} \rho^m u_\alpha(\rho)=0$ then $\lim_{\rho \to \infty} \rho^{m+1} u_\alpha'(\rho)=0$. 
\end{lemma}
\begin{proof}
By \eqref{eq:HOmega_alpha'} applied as in \autoref{prop:local.existence.uniqueness}, and multiplying both sides by $(1+\rho)^{m+1}$, we have
\begin{align}
    (1+\rho)^{m+1}|u_\alpha'(\rho)| 
    &\leq
    (1+\rho)^{m+1}\rho^{-n-1} e^{-\frac{\rho^2}{4}} \int_0^{\rho} s^{n+1}e^{\frac{s^2}{4}} |u_\alpha(s) + (n-2-s^2/2) u_\alpha^2(s) |ds
    \\
    &\quad+
    (1+\rho)^{m+2} u_\alpha^2(\rho);
\end{align}
By assumption $(1+\rho)^{m+2} u_\alpha^2(\rho) \leq C^2 (1+\rho)^{2-m} \leq C^2$. On the other hand, $|u_\alpha(s) + (n-2-s^2/2) u_\alpha^2(s)| \lesssim C(1+s)^{-m}+ C^2(1+s)^{2-2m} \leq (C+C^2)(1+s)^{-m}$ and thus we can bound 
\begin{align}
 \int_0^{\rho} &s^{n+1}e^{\frac{s^2}{4}} |u_\alpha(s) + (n-2-s^2/2) u_\alpha^2(s) |ds
 \\
 &\leq
 (C+C^2)  \int_0^{\rho} s^{n+1}e^{\frac{s^2}{4}} (1+s)^{-m} ds
 \\
 &\leq
 (C+C^2) (\rho/2)^{n+2}e^{\frac{\rho^2}{16}}
 + 
 (C+C^2)\int_{\rho/2}^{\rho} s^{n+1}e^{\frac{s^2}{4}} (1+s)^{-m} ds
 \\
 &\leq
 (C+C^2) (\rho/2)^{n+2}e^{\frac{\rho^2}{16}}
 + 
 (C+C^2) \rho^{n+1}(1+\rho/2)^{-m-1}\int_{\rho/2}^{\rho} e^{\frac{s^2}{4}} (1+s) ds.
\end{align}
Hence,
\begin{align}
    (1+\rho)^{m+1}|u_\alpha'(\rho)| 
    &\lesssim
    (C+C^2)
    \left( 
    (1+\rho)^{m+2} e^{-\frac{3\rho^2}{16}}
    +
    e^{-\frac{\rho^2}{4}}
    \int_{\rho/2}^{\rho} e^{\frac{s^2}{4}} (1+s) ds
    +
    1 \right)
    \lesssim
    C+C^2,
\end{align}
as for fixed $m$ the quantity between parentheses is uniformly bounded in $\rho \geq 0$.

The second part of the statement follows by the same computation as above, noticing that if $|u_\alpha(\rho)| \leq \varepsilon \rho^{-m}$ for $0<\varepsilon \ll 1$ and every $\rho  \geq \rho_\varepsilon$, then the previous line reads
\begin{align}
    (1+\rho)^{m+1}|u_\alpha'(\rho)| 
    &\lesssim
    (C+C^2)
    \left( 
    (1+\rho)^{m+2} e^{-\frac{3\rho^2}{16}}
    +
    \varepsilon e^{-\frac{\rho^2}{4}}
    \int_{\rho/2}^{\rho} e^{\frac{s^2}{4}} (1+s) ds
    +
    \varepsilon^2 \right),
\end{align}
for every $\rho \geq 2\rho_\varepsilon$. Since $\varepsilon$ is arbitrary and $\lim_{\rho \to \infty} (1+\rho)^{m+2} e^{-\frac{3\rho^2}{16}}=0$, we get $(1+\rho)^{m+1}|u_\alpha'(\rho)| \to 0$, as desired.
\end{proof}

\begin{proof}[Proof of \autoref{prop:Omega_alpha>0}]
Our strategy is inspired by the proof of \cite[Propositions 3.6 and 3.7]{HaWe82}.

\emph{Step 1: $\alpha \in (0,1/4)$}.
By elementary computations one can verify that the derivative of $u_\alpha$ can be expressed for every $\rho>0$ as
\begin{align} \label{eq:Omega_alpha'}
    \rho^{n+1}u_\alpha'(\rho)
    =
    -\frac{\rho^{n+2}}{2} u_\alpha(\rho) + \int_0^\rho s^{n+1} u_\alpha(s) \left( \frac{n}{2 } -G_\alpha(s) \right) ds.
 \end{align}
On the other hand, $G_{\alpha}$ is a $C^1$ function, and by \eqref{eq:ODE} its derivative is given by
\begin{align} \label{eq:G_alpha'}
    G_{\alpha}'(\rho)
    =
    - 2\rho  \left( G_{\alpha}(\rho)  \left(u_{\alpha}(\rho) + \frac14 \right) + \left( 1-\frac{n}2\right)u_{\alpha}(\rho)   \right).
\end{align}

Let us preliminarily assume $\alpha<1/4$ and suppose by contradiction that at least one of $u_\alpha(\rho)$ and $G_\alpha(\rho)$ is not strictly positive for all $\rho \in [0,\infty)$, and define $\rho_0>0$ as the first value of $\rho$ for which either $u_\alpha(\rho) = 0$ or $G_\alpha(\rho)=0$.

By \autoref{lem:decreasing}, $u_\alpha'(\rho)<0$ for every $\rho \in (0,\rho_0]$ and therefore $\max_{s \in [0,\rho_0]} G_\alpha(s) = G_\alpha(0) = 2n\alpha < \frac{n}{2}$.
Thus, if we evaluate \eqref{eq:Omega_alpha'} at $\rho = \rho_0$, the left-hand side is always negative, whereas the integral on the right-hand side is always positive: we deduce that $u_\alpha(\rho_0)>0$ and therefore the alternative $G_\alpha(\rho_0)=0$ must hold.
However, in this case we evaluate \eqref{eq:G_alpha'} at $\rho = \rho_0$ and we get $G_\alpha'(\rho_0)>0$, contradicting  minimality of $\rho_0$.
We deduce that no such $\rho_0$ can exist and therefore $u_\alpha$, $G_\alpha$ are always positive.

\emph{Step 2: $\alpha>1/4$}.
Let $A := \{ \alpha >0 \,:\, \exists \rho \geq 0 \,\, \mbox{s.t.}\,\,  u_\alpha(\rho)=0 \}$ and denote $\alpha_\star := \inf A$. We have just shown that $\alpha_\star \geq 1/4$; suppose by contradiction $\alpha_\star < \infty$. 
Notice that if $u_{\alpha_\star}(\rho_0)=0$ for some $\rho_0 > 0$, then we must necessarily have $u_{\alpha_\star}'(\rho_0) \neq 0$ (otherwise $u_{\alpha_\star} \equiv 0$ by uniqueness of solutions to \eqref{eq:ODE}, which is incompatible with the initial condition $u_{\alpha_\star}(0)=\alpha_\star$) and in particular $u_{\alpha_\star}(\rho)<0$ for some $\rho$ in a neighborhood of $\rho_0$.
Therefore, by continuity of the map $\alpha \mapsto u_\alpha$ at the point $\alpha_\star$ given by \autoref{prop:local.existence.uniqueness}, we must have $u_{\alpha_\star}(\rho)>0$ for every $\rho \geq 0$, otherwise $u_\alpha$ would keep having zeros for every $\alpha$ in a neighborhood of $\alpha_\star$, contradicting minimality of $\alpha_\star$.

Next we claim that $\lim_{\rho \to \infty} \rho^2 u_{\alpha_\star}(\rho) =: \ell_{\alpha_\star}$ exists and $\ell_{\alpha_\star} \in [0,2]$. The lower bound is clear, since $u_{\alpha_\star}>0$; as for the upper bound, we can consider the quantity:
\begin{align}
    J_{\alpha_\star}(\rho)
    :=
    \rho^{n+1} e^{\frac{\rho^2}{4}}
    \left( u_{\alpha_\star}'(\rho) + \rho u_{\alpha_\star}^2(\rho)\right).
\end{align}
Notice that $J_{\alpha_\star}(0)=0$ and $J_{\alpha_\star}$ is non-positive in a right neighborhood of the origin, by Taylor expanding $u_{\alpha_\star}$.
We claim that $J_{\alpha_\star}(\rho) \leq 0$ for every $\rho \geq 0$.
Indeed, by explicit computation it holds
\begin{align} \label{eq:J'}
    J'_{\alpha_\star}(\rho) =
    -\rho^{n+1} e^{\frac{\rho^2}{4}}
    \left( u_{\alpha_\star}(\rho) + (n-2-\rho^2/2) u_{\alpha_\star}^2(\rho)\right).
\end{align}
Then we deduce that if $[0,\rho_0]$ is the maximal interval such that $J_{\alpha_\star}(\rho) \leq 0$ for every $\rho \in [0,\rho_0]$, then we have by comparison 
\begin{align} \label{eq:rho2Omega_alpha.leq.2}
    u_{\alpha_\star}'(\rho) \leq - \rho u_{\alpha_\star}^2(\rho)
\quad
    \Longrightarrow
    \quad
    0< u_{\alpha_\star}(\rho) \leq \frac{\alpha_\star}{1+\frac{\alpha_\star}{2}\rho^2} \leq \frac{2}{\rho^2}
\end{align}
for every $\rho \in [0,\rho_0]$, implying $J'_{\alpha_\star}(\rho_0)<0$ by \eqref{eq:J'}. Thus we can choose $\rho_0 = \infty$ and therefore by \eqref{eq:rho2Omega_alpha.leq.2} we have $\limsup_{\rho \to \infty} \rho^2 u_{\alpha_\star}(\rho) \le 2$. 
To show that the limit exists, differentiate the quantity $\rho^{2} u_{\alpha_\star}(\rho)
    + 2\rho u_{\alpha_\star}'(\rho)$ to obtain
\begin{align} \label{eq:rho2Omega+2rhoOmega'}
    \rho^{2} u_{\alpha_\star}(\rho)
    =
    -2\rho u_{\alpha_\star}'(\rho)
    +
    2n \alpha_\star -
    2n u_{\alpha_\star}(\rho)
    -\int_0^\rho 2s u_{\alpha_\star}(s) G_{\alpha_\star}(s)ds.
\end{align}
Then by \autoref{lem:bound.rho.m.Omega} the right-hand side above admits a limit when $\rho \to \infty$, hence so does the left-hand side. 

In the following, we consider separately the two cases $\ell_{\alpha_\star} = 0$ and $\ell_{\alpha_\star}>0$, seeking a contradiction in each case. This will contradict $\alpha_\star < \infty$ and conclude the proof that $u_\alpha>0$. In turn, this implies $G_\alpha>0$ arguing as in Step 1.

\emph{Step 3. $\ell_{\alpha_\star} = 0$}.
For the sake of contradiction, suppose $\ell_{\alpha_\star} = 0$.
Then we can show that 
\begin{align} \label{eq:rho.m.Omega.to.0}
    \lim_{\rho \to \infty} \rho^m u_{\alpha_\star}(\rho) = 0,
    \quad
    \forall m \geq 2.
\end{align}
We argue by induction on $m$. The base case $m=2$ is true by assumption. Then, suppose \eqref{eq:rho.m.Omega.to.0} is true for a given $m \geq 2$: by \autoref{lem:bound.rho.m.Omega} we have $\lim_{\rho \to \infty} \rho^{m+1} u_{\alpha_\star}'(\rho) = 0$, and therefore $\lim_{\rho \to \infty} \rho^{m}G_{\alpha_\star}(\rho)=0$ as well.
Then, using \eqref{eq:rho2Omega+2rhoOmega'} we get
\begin{align}
    \rho^{m+2} u_{\alpha_\star}(\rho)
    =
    -2\rho^{m+1} u_{\alpha_\star}'(\rho)
    +
    n \rho^{m} u_{\alpha_\star}(\rho)
    +\rho^{m}\int_\rho^\infty 2s u_{\alpha_\star}(s) G_{\alpha_\star}(s)ds,
\end{align}
and by $\lim_{\rho \to \infty} \rho^{m+1} u_{\alpha_\star}'(\rho) = \lim_{\rho \to \infty} \rho^{m}G_{\alpha_\star}(\rho)=0$ we obtain $\rho^{m+2} u_{\alpha_\star}(\rho) \to 0$. Therefore, \eqref{eq:rho.m.Omega.to.0} holds.
But this leads to a contradiction: indeed, since $u_{\alpha_\star}(\rho)>0$ for every $\rho \geq 0$, the arguments in Step 1 give $G_{\alpha_\star}(\rho)>0$ for every $\rho \geq 0$. Hence by comparison
\begin{align}
    \rho u_{\alpha_\star}'(\rho) > - n u_{\alpha_\star}(\rho)
    \quad
    \Longrightarrow
    \quad
    u_{\alpha_\star}(\rho) \geq u_{\alpha_\star}(1) \rho^{-n},
\end{align}
in conflict with \eqref{eq:rho.m.Omega.to.0}. 

\emph{Step 4. $\ell_{\alpha_\star}>0$}.
By \eqref{eq:rho2Omega_alpha.leq.2} and \autoref{lem:bound.rho.m.Omega} we can find $\rho_0>1$ large enough so that for some unimportant constant $C>0$ it holds
\begin{align}
    \frac{\ell_{\alpha_\star}}{2\rho_0^2}<u_{\alpha_\star}(\rho_0)<\frac{2\ell_{\alpha_\star}}{\rho_0^2},
    \quad
    -\frac{C}{\rho_0^3}<u_{\alpha_\star}'(\rho_0) < 0.
\end{align}
In particular, without loss of generality we might choose $\rho_0$ such that
\begin{align} \label{eq:condition.alpha.star}
 u_{\alpha_\star}(\rho_0)<\frac14,
 \quad0<G_{\alpha_\star}(\rho_0)<\frac{n}{2},
 \quad
 \frac{\rho_0}{2}u_{\alpha_\star}(\rho_0)+u_{\alpha_\star}'(\rho_0)>0.
\end{align}
By \autoref{prop:local.existence.uniqueness}, there exists a neighborhood $U_{\alpha_\star}$ of $\alpha_\star$ such that for every $\alpha \in U_{\alpha_\star}$ we have $u_\alpha(\rho)>0$ for every $\rho \leq \rho_0$ and \eqref{eq:condition.alpha.star} holds with $\alpha_\star$ replaced by $\alpha$.

Now we take any $\alpha \in U_{\alpha_\star}$, $\alpha>\alpha_\star$ and look at equation \eqref{eq:Omega_alpha'} between $\rho_0$ and $\rho>\rho_0$:
\begin{gather}
 \rho^{n+1}u_\alpha'(\rho) = \rho_0^{n+1}u_\alpha'(\rho_0) 
    + 
    \frac{\rho_0^{n+2}}{2} u_\alpha(\rho_0)
    -
    \frac{\rho^{n+2}}{2} u_\alpha(\rho)
    +
    \int_{\rho_0}^\rho s^{n+1} u_\alpha(s) \left( \frac{n}{2 } -G_\alpha(s) \right) ds.  
\end{gather}
We want to show that $u_\alpha$ remains strictly positive for every $\rho > \rho_0$, contradicting $\alpha_\star = \inf A$. To do this, we can argue as in Step 1: Indeed, as long as $u_\alpha$ is positive then $u_\alpha'<0$ by \autoref{lem:decreasing}, and therefore $G_\alpha(\rho)<2nu_\alpha(\rho) \leq 2nu_\alpha(\rho_0) < n/2$. Thus the integral in the equation above is positive, and 
\begin{align}
    u_\alpha'(\rho) 
    > 
    \rho^{-n-1} \rho_0^{n+1} \left( u_\alpha'(\rho_0) 
    + 
    \frac{\rho_0}{2} u_\alpha(\rho_0) \right)
    -
    \frac{\rho}{2} u_\alpha(\rho),
\end{align}
implying $u_\alpha'(\rho)>0$ at the first $\rho$ for which $u_\alpha(\rho)=0$. This leads to the same contradiction encountered in Step 1, and completes the proof of $u_\alpha,G_\alpha>0$.

\emph{Step 5: Conclusion}. We are left to show $u_\alpha(\rho) < \alpha(1+\frac{\alpha}{2}\rho^2)^{-1}$ for every $\rho > 0$ and the existence of the limit $\ell_\alpha := \lim_{\rho \to \infty} \rho^2 u_\alpha(\rho) \in (0,2]$. But since we now know that $u_\alpha>0$, we can repeat the arguments presented in Steps 2-3 for $\alpha=\alpha_\star$ and recover the same results for every $\alpha>0$.
\end{proof}
\subsection{Linearized equation} \label{ssec:linearized.system}
Let us introduce the ODE obtained by linearization of \eqref{eq:ODE} around a profile $u_\alpha$, which reads as follows:
\begin{align} \label{eq:ODE.linearized}
    \begin{dcases}
f_\alpha'' 
+ 
\left( \frac{n+1}{\rho} +  \frac{\rho}{2} +2 \rho u_\alpha \right) f_\alpha' 
+ 
f_\alpha
+ 4n u_\alpha f_\alpha +  2 \rho u_\alpha'  f_\alpha = 0,
\\
f_\alpha(0)=1,
\\
f_\alpha'(0) = 0.
    \end{dcases}
\end{align}
In the following lemma we study how the solution of \eqref{eq:ODE.linearized} behaves as its coefficients change.
Denote $\mathcal{V}_\alpha := 4n u_\alpha + 2\rho u_\alpha'$.
For fixed $\alpha>0$, by \autoref{prop:Omega_alpha>0} we have $0 < \mathcal{V}_\alpha(\rho) \lesssim \rho^{-2}$ for every $\rho \geq 0$ and $\rho^2 \mathcal{V}_\alpha(\rho) > n \ell_\alpha>0$ for every $\rho$ sufficiently large. Moreover, by \autoref{prop:local.existence.uniqueness} and the equation satisfied by $u_\alpha$, we have that $\mathcal{V}_\alpha$ is smooth.

Fix $\lambda>0$ and consider the \emph{eigenvalue problem}
\begin{equation} \label{eq:eigen}
\begin{dcases}
    f'' + \left( \frac{n+1}{\rho} +  \frac{\rho}{2} +2 \rho u_\alpha \right) f'  + f  + \mu \mathcal{V}_\alpha f =  \lambda f,
    \quad
    \mu \in \R,
    \\
    f(0) = 1,
    \\
    f'(0) = 0.
\end{dcases}
    \end{equation} 

For every $\alpha,\lambda>0$ and $\mu \in \R$, given previous \autoref{prop:Omega_alpha>0} and \autoref{lem:bound.rho.m.Omega}, the coefficients of \eqref{eq:eigen} are smooth and locally bounded far from the origin, where the singularity $\sim 1/\rho$ can be treated arguing similarly to \autoref{prop:local.existence.uniqueness}. Moreover, finite time blow up cannot happen by linearity of the ODE. As a consequence, a unique $C^2([0,\infty))$ solution of \eqref{eq:eigen} exists.

\begin{lemma} \label{lem:principal_eigen}
Let $\alpha,\lambda>0$ be fixed. Then the following hold:
\begin{itemize}
    \item[$i$)] there exists $\bar \mu = \bar\mu(\alpha,\lambda) \in (0,\infty)$ such that the solution $f$ of \eqref{eq:eigen} with $\mu=\bar \mu$  satisfies $f(\rho) >0$ and 
\begin{align} \label{eq:decay_f}
\int_1^\infty \frac{d\rho}{H(\rho) f^2(\rho)} = \infty,
\quad
H(\rho) := \exp \int_1^\rho \left( \frac{n+1}{s} + \frac{s}{2} + 2 su_\alpha(s)\right) ds;
\end{align}
    \item[$ii$)]
    If $\mu>\bar \mu$ then the associated solution $f$ of \eqref{eq:eigen} admits at least a root $\rho_0 \in (0,\infty)$;
    \item[$iii$)] 
    If $\mu < \bar \mu$ then the associated solution $f$ of \eqref{eq:eigen} is strictly positive and 
\begin{align} 
\int_1^\infty \frac{d\rho}{H(\rho) f^2(\rho)} < \infty.
\end{align}
\end{itemize}
\end{lemma}
\begin{proof}
    Part $i)$ follows by variational principle within the weighted Sobolev space $H^1_\pi(\R^{n+2})$ with weight $\pi$ given by $\pi(y) := \exp \int_0^{|y|}(\frac{s}{2}+2su_\alpha(s))ds$, similarly to the proof of \cite[Proposition D.1]{Naito}.
More precisely, let us consider the weighted spaces 
\begin{align*}
L_\pi^2(\R^{n+2}) 
&:= 
\left\{ F \in L^2(\R^{n+2}) : \int_{\R^{n+2}}   |F(y)|^2 \pi(y) dy < \infty \right\},
\\
H_\pi^1(\R^{n+2}) 
&:= 
\left\{ F \in H^1(\R^{n+2}) : \int_{\R^{n+2}} \left( |\nabla F(y)|^2 + |F(y)|^2 \right) \pi(y) dy < \infty \right\}.
\end{align*}

Notice that by \autoref{prop:Omega_alpha>0} the weight $\pi$ satisfies
\begin{align} 
   \theta(y) &:=  \left(\Delta \log \pi + \frac12 |\nabla \log \pi|^2 \right)(y)
    \\
    &= \frac{n+2}{2} + 2(n+2) u_\alpha(|y|) + 2|y|u'_{\alpha}(|y|) + \frac{|y|^2}{2}\left( \frac12 +2u_\alpha(|y|)\right)^2
    \\
    \label{eq:theta}
    &\geq \frac{n+2}{2} + 4 u_\alpha(|y|)  + \frac{|y|^2}{2}\left( \frac12 +2u_\alpha(|y|)\right)^2,
\end{align}
in particular $\theta(y) \to +\infty$ when $|y| \to \infty$.
Thus, by \cite[Proposition 1.1]{EsKa87} the embedding $H^1_\pi(\R^{n+2}) \subset L_\pi^2(\R^{n+2})$ is compact. Moreover, by \cite[Lemma 1.5]{EsKa87} we have for every $F \in H^1_\pi(\R^{n+2})$ 
\begin{align} \label{eq:weighted_poincarè}
\int_{\R^{n+2}} |\nabla F(y)|^2  \pi(y) dy 
\geq 
\frac12 
\int_{\R^{n+2}} |F(y)| ^2  \theta(y) \pi(y) dy.
\end{align}

Then we can consider the variational problem for the \emph{principal eigenvalue} 
\begin{align*}
\bar\mu := \inf_{F \in H^{1}_\pi(\R^{n+2})} \left\{ 
\int_{\R^{n+2}} \left( |\nabla F(y)|^2 - (1-\lambda) |F(y)|^2 \right) \pi(y) dy
\,:\,
\int_{\R^{n+2}} |F(y)|^2 \mathcal{V}_\alpha(|y|) \pi(y) dy = 1   \right\}.
\end{align*}

\emph{Step 1: $\bar \mu \in (0,\infty)$}.
The set defining the infimum is non-empty, since any non-zero $F \in C^\infty_c(\R^{n+2})$ can be suitably renormalized to satisfy $\int_{\R^{n+2}} |F(y)|^2 \mathcal{V}_\alpha(|y|) \pi(y) dy = 1$.
In particular, $\bar \mu<\infty$.

To see that $\bar \mu> 0$, let us preliminarily  observe that in virtue of \autoref{prop:Omega_alpha>0} and \autoref{lem:decreasing}, for every $\alpha \in (0,\infty)$ it holds 
\begin{align}
    \| \mathcal{V}_\alpha(|\cdot|)\|_{L^\infty(\R^{n+2})} = \mathcal{V}_\alpha(0) = 4n \alpha.
\end{align}
Hence, for every $F$ satisfying the constraints of the variational problem we have
\begin{align} \label{eq:lower.bound.F^2}
    1 = \int_{\R^{n+2}} |F(y)|^2 \mathcal{V}_\alpha(|y|) \pi(y) dy  \leq
    4n \alpha \int |F(y)|^2 \pi(y)dy,
\end{align}
and using \eqref{eq:weighted_poincarè}, \eqref{eq:theta}, $n \geq 3$, and \eqref{eq:lower.bound.F^2} we have:
\begin{align}
\int_{\R^{n+2}} \left( |\nabla F(y)|^2 - (1-\lambda) |F(y)|^2 \right) \pi(y) dy
&\geq
\int_{\R^{n+2}}    |F(y)| ^2  \left(\frac{\theta(y)}{2} -1\right) \pi(y) dy
\\
&\geq \label{eq:lower.bound.barmu}
\frac14 \int_{\R^{n+2}}    |F(y)| ^2  \pi(y) dy
\geq   \frac{1}{16n \alpha }
    > 0.
\end{align}

\emph{Step 2: Finding a minimizer}.
Since the embedding $H^1_\pi(\R^{n+2}) \subset L_\pi^2(\R^{n+2})$ is compact, given a minimizing sequence $\{F_k\}_{k \in \NN}$ approaching the infimum of the variational problem we can extract a subsequence converging weakly in $H^1_\pi(\R^{n+2})$ and strongly in $L_\pi^2(\R^{n+2})$ towards a limit $\tilde F \in H^1_\pi(\R^{n+2})$.
Since $\mathcal{V}_\alpha(|\cdot|) \in L^\infty(\R^{n+2})$, the constraint $\int_{\R^{n+2}} |\tilde F(y)|^2 \mathcal{V}_\alpha(|y|) \pi(y) dy = 1$ is satisfied, and by lower semicontinuity of the norm with respect to weak convergence, $\tilde F$ is a minimizer of the variational problem. 
Moreover, using that $|\nabla F|^2 \leq |\nabla \tilde F|^2$ for $F:=|\tilde F|$ and that the variational problem is rotationally invariant, we can additionally assume that the minimizer satisfies $F(y) = f(|y|)$ for some non-negative function $f$.

\emph{Step 3: Euler-Lagrange equations}.
Since $F$ realizes the infimum defining $\bar \mu$, 
it solves the Euler-Lagrange equation for every smooth test function $\varphi \in C^\infty_c(\R^{n+2})$
\begin{align} \label{eq:EL_F}
\int_{\R^{n+2}} \left( \nabla F(y) \cdot \nabla \varphi(y)  - (1-\lambda) F(y) \varphi(y) \right) \pi(y)dy = \bar \mu \int_{\R^{n+2}} F(y) \varphi(y) \mathcal{V}_\alpha(|y|) \pi(y) dy;
\end{align}
namely, $F \in H^1_\pi(\R^{n+2})$ is a weak solution of the elliptic PDE
\begin{align}
    \mbox{div}(\pi \nabla F) + (1-\lambda) F\pi + \bar \mu F \mathcal{V}_\alpha(|\cdot|) \pi = 0.
\end{align}
Notice that $\mathcal{V}_\alpha(|\cdot|)$ is smooth as a consequence of the power series expansion of $u_\alpha$ around the origin given in \autoref{prop:local.existence.uniqueness} and smoothness of the coefficients of \eqref{eq:ODE}. 
Since $\pi$ is also smooth and strictly positive, by elliptic regularity theory \cite[Section 6.3.1, Theorem 3]{Evans} and Harnack inequality we deduce that $F \in C^2_{loc}(\R^{n+2})$ and $F(y)>0$ for every $y \in \R^{n+2}$.
Moreover, since $F$ is radial and $C^2_{loc}(\R^{n+2})$ we have $\nabla F(0) = 0  $. In particular, $f \in C^2_{loc}([0,\infty))$ satisfies $f(\rho) > 0$ for every $\rho \geq0 $ and $f'(0)=0$.
Furthermore, by change of variables, integration by parts, and using that $F$ solves \eqref{eq:EL_F}, we have that $f$ solves the ODE
\begin{align*}
f'' + \left( \frac{n+1}{\rho} +  \frac{\rho}{2} +2 \rho u_\alpha \right) f'  + f + \bar\mu \mathcal{V}_\alpha f = \lambda f 
\end{align*}
with initial conditions $f(0)=F(0)>0$ and $f'(0)=0$.
By linearity, we can suppose without loss of generality
that $f(0)=1$. Putting all together, $f$ is the unique solution of system \eqref{eq:eigen}.

\emph{Step 4: Decay of $f$ at $\infty$.}
Since $|F|^2 \mathcal{V}_\alpha(|\cdot|) \pi \in L^1(\R^{n+2})$ and $\rho^2 \mathcal{V}_\alpha(\rho) > n \ell_\alpha>0$ for every $\rho$ sufficiently large, it holds 
\begin{align} \label{eq:aux001}
    \infty 
    &> 
    \int_{\{ |y| \geq 1\}}
    |F(y)|^2 \mathcal{V}_\alpha(|y|) \pi(y) dy 
    \\
    &= C_n
    \int_1^\infty f^2(\rho) \mathcal{V}_\alpha(\rho) H(\rho)  d\rho
    \gtrsim 
    \int_1^\infty f^2(\rho) H(\rho) \rho^{-2} d\rho ,
\end{align}
for some dimensional constant $C_n$.
By Cauchy-Schwarz inequality we get
\begin{align}
    \infty = \int_1^\infty \frac{1}{\rho} d\rho
    =
    \int_1^\infty \frac{f(\rho) H^{1/2}(\rho) \rho^{-1}}{f(\rho) H^{1/2}(\rho)} d\rho
    \lesssim
    \left(
    \int_1^\infty  f^2(\rho) H(\rho) \rho^{-2} d\rho
    \right)^{1/2}
\left(
    \int_1^\infty \frac{d\rho}{H(\rho) f^2(\rho)}
    \right)^{1/2},
\end{align}
and \eqref{eq:aux001} yields
\begin{align}
\int_1^\infty \frac{d\rho}{H(\rho) f^2(\rho)} = \infty.
\end{align}

Parts $ii)$ and $iii)$ of the lemma follow by the comparison principle given in \cite[Proposition A.1]{Naito}.
More specifically, point $ii)$ follows by \cite[Proposition A.1 i)]{Naito}, while under the assumption of point $iii)$ we have by \cite[Proposition A.1 ii)]{Naito} that $f>0$ and 
\begin{align}
    \int_1^\infty \frac{d\rho}{H(\rho) f^2(\rho)} < \infty.
\end{align}
\end{proof}

\subsection{Zeros of linearized ODE}
\label{ssec:zeros}
The main goal of this subsection is to prove the following 

\begin{prop} \label{prop:f.zero}
  Let $n \in \{3, \dots,9\}$. Then for every $k \in \NN$ there exists $\alpha_k>0$ such that for every $\alpha>\alpha_k$ the solution $f_\alpha$ of \eqref{eq:ODE.linearized} vanishes at least in $k$ different points $0<\rho_1<\dots<\rho_k$.
\end{prop}

The proof of this proposition is not straightforward and relies on a delicate argument based on Sturm–Picone comparison Theorem, which requires a detailed description of the ``tail'' of $u_\alpha$ for large values of $\alpha$.
To study this, it is convenient to introduce the rescaled quantities
\begin{align}
    \tilde{u}_\alpha(\rho) := \frac{u_\alpha(\rho/\sqrt{\alpha})}{\alpha},
    \quad
    \tilde{f}_\alpha(\rho) := f_\alpha(\rho/\sqrt{\alpha}),
\end{align}
so that $\tilde{u}_\alpha$ remains of order one uniformly in $\alpha$. Clearly $f_\alpha$ vanishes in $k$ different points if and only if $\tilde{f}_\alpha$ does.
It is immediate to verify that $(\tilde{u}_\alpha,\tilde{f}_\alpha)$ satisfies the system
\begin{align} \label{eq:rescaled.ODE.linearized}
\begin{dcases}
%
%
\tilde{u}_\alpha'' 
+ 
\frac{n+1}{\rho} \tilde{u}_\alpha' 
+ 
\frac{1}{\alpha} \tilde{u}_\alpha 
+ 
\frac{\rho}{2\alpha} \tilde{u}_\alpha' 
+ 
2n \tilde{u}_\alpha^2 
+ 
2 \rho \tilde{u}_\alpha \tilde{u}_\alpha' = 0,
\\
%
%
\tilde{f}_\alpha'' 
+ 
\frac{n+1}{\rho} \tilde{f}_\alpha' 
+ 
\frac{1}{\alpha}\tilde{f}_\alpha 
+ 
\frac{\rho}{2\alpha} \tilde{f}_\alpha' 
+ 
4n \tilde{u}_\alpha \tilde{f}_\alpha 
+ 
2 \rho \tilde{u}_\alpha \tilde{f}_\alpha'
+ 
2 \rho \tilde{f}_\alpha \tilde{u}_\alpha' = 0,
\\
%
%
\tilde{u}_\alpha(0)=\tilde{f}_\alpha(0)=1,
\\
\tilde{u}_\alpha'(0)=\tilde{f}_\alpha'(0) = 0.
\end{dcases}
\end{align}
To understand the behavior of the pair $(\tilde{u}_\alpha,\tilde{f}_\alpha)$ when $\alpha \to \infty$, we can argue as in \autoref{prop:local.existence.uniqueness} and write solutions to \eqref{eq:rescaled.ODE.linearized} in the following form:
\begin{align}
\tilde{u}_\alpha(\rho)
=
1
&-\int_0^\rho \rho_2^{-n-1}\int_0^{\rho_2}
s^{n+1} \left( 
-\frac{n}{2\alpha} \tilde{u}_\alpha(s) 
+ 
(n-2) \tilde{u}_\alpha^2(s)  
\right)
ds d\rho_2
\\
&-
\int_0^\rho \rho_2 \left( \frac{1}{2\alpha} \tilde{u}_\alpha(\rho_2)+\tilde{u}_\alpha^2(\rho_2) \right) d\rho_2;
\end{align}
and
\begin{align}
\tilde{f}_\alpha(\rho) 
=
1
&-
\int_0^\rho
\rho_2^{-n-1}
\int_0^{\rho_2}s^{n+1} 
\left(  
-\frac{n}{2\alpha} \tilde{f}_\alpha(s) 
+ 
2(n-2) \tilde{u}_\alpha(s) \tilde{f}_\alpha(s) 
\right)
ds d\rho_2
\\
&-
2 \int_0^\rho \rho_2 \tilde{u}_\alpha(\rho_2) \tilde{f}_\alpha(\rho_2)d\rho_2
-
\frac{1}{2\alpha}\int_0^\rho \rho_2 \tilde{f}_\alpha(\rho_2)d\rho_2.
\end{align}
We can interpret $(\tilde{u}_\alpha,\tilde{f}_\alpha)$ as the fixed point of a contraction map, depending on a parameter $\frac{1}{\alpha}$, on a space $C([0,\overline{\tau}_\alpha])$ for some $\overline{\tau}_\alpha>0$. Moreover, the interval $[0,\overline{\tau}_\alpha]$ can be chosen uniformly in $\alpha>1$, call it $[0,\overline{\tau}]$. 
Hence, contraction principle and the expressions above give continuity of the map $\alpha \mapsto (\tilde{u}_\alpha,\tilde{f}_\alpha)$ with respect to the $C^1([0,\overline{\tau}])$ convergence.
Moreover, smoothness of the coefficients far from $\rho=0$ entails locally $C^1$ convergence of the pair $(\tilde{u}_\alpha,\tilde{f}_\alpha)$, as $\alpha \to \infty$, towards $(\tilde{u},\tilde{f})$ satisfying 
\begin{align}
\tilde{u}(\rho)
=
1
&-\int_0^\rho \rho_2^{-n-1}\int_0^{\rho_2}
s^{n+1} 
(n-2) \tilde{u}^2(s)  
ds d\rho_2
-
\int_0^\rho \rho_2 \tilde{u}^2(\rho_2) d\rho_2;
\end{align}
and
\begin{align}
\tilde{f}(\rho) 
=
1
&-
\int_0^\rho
\rho_2^{-n-1}
\int_0^{\rho_2}s^{n+1} 
\left(  
2(n-2) \tilde{u}(s) \tilde{f}(s) 
\right)
ds d\rho_2
-
2 \int_0^\rho \rho_2 \tilde{u}(\rho_2) \tilde{f}(\rho_2)d\rho_2.
\end{align}
Differentiating twice the expressions above, it is easy to check that $(\tilde{u},\tilde{f})$ solves the ODE system
\begin{align} \label{eq:limit.ODE.linearized}
\begin{dcases}
%
%
\tilde{u}'' 
+ 
\frac{n+1}{\rho} \tilde{u}' 
+ 
2n \tilde{u}^2 
+ 
2 \rho \tilde{u} \tilde{u}' = 0,
\\
%
%
\tilde{f}'' 
+ 
\frac{n+1}{\rho} \tilde{f}' 
+ 
4n \tilde{u} \tilde{f} 
+ 
2 \rho \tilde{u} \tilde{f}'
+ 
2 \rho \tilde{f} \tilde{u}' = 0,
\\
%
%
\tilde{u}(0)=\tilde{f}(0)=1,
\\
\tilde{u}'(0)=\tilde{f}'(0) = 0,
\end{dcases}
\end{align}
as one could have guessed by simply taking the formal limit $\alpha \to \infty$ in \eqref{eq:rescaled.ODE.linearized}.

By uniqueness of solutions to the equation above, we deduce that $\tilde{f}'(\rho_0) \neq 0$ for every $\rho_0$ such that $\tilde{f}(\rho_0)=0$, otherwise we would have $\tilde{f} \equiv0$ in contrast with the boundary conditions at $\rho=0$.  
Therefore, if we can prove that $\tilde{f}$ vanishes in at least $k$ different points $0<\tilde{\rho}_1 < \dots < \tilde{\rho}_k$, then we can use the uniform convergence of $\tilde{f}_\alpha$ towards $\tilde{f}$ on the compact $[0,\tilde{\rho}_k+1]$ to deduce that also $\tilde{f}_\alpha$ must vanish in at least $k$ different points for every $\alpha$ sufficiently large. 
Thus, in the following we turn to looking for zeros of $\tilde{f}$.

\subsubsection{Tails of $\tilde{u}$} \label{ssec:tail}
One key ingredient needed in the proof of \autoref{prop:f.zero} is the following:
\begin{prop} \label{prop:tails}
    Let $n \geq 3$. Then it holds
    \begin{align} \label{eq:limits.tildeOmega}
        \lim_{\rho \to \infty} 
        \rho^2 \tilde{u}(\rho) = 1,
        \quad
        \lim_{\rho \to \infty} 
        \rho^3 \tilde{u}'(\rho) = -2.
    \end{align}
\end{prop}
First of all, notice that by \autoref{prop:Omega_alpha>0}, \autoref{lem:decreasing}, and \autoref{lem:bound.rho.m.Omega} we know that for every $\alpha>0$ and $\rho \geq 0$ it holds 
\begin{align}
   0 \leq \rho^2 \tilde{u}_\alpha(\rho) \leq 2,
   \quad
   -1 \lesssim \rho^3 \tilde{u}_\alpha'(\rho) \leq 0, 
\end{align}
with implicit constant independent of $\alpha$ and $\rho$.
By $C^1$ convergence $\tilde{u}_\alpha \to \tilde{u}$ on compacts, the same inequalities hold uniformly in $\rho \geq 0$ for the limit $\tilde{u}$.
From this we also get $\tilde{u}(\rho) > 0$ for every $\rho \geq 0$: indeed, if $\tilde{u}(\rho_0)=0$ for some $\rho_0>0$, then $\tilde{u}'(\rho_0)=0$ and therefore $\tilde{u} \equiv 0$ by \eqref{eq:limit.ODE.linearized}, in contradiction with $\tilde{u}(0)=1$. 

By \autoref{prop:Omega_alpha>0} we also know that, for fixed $\alpha>0$, the limit $\lim_{\rho \to \infty} \rho^2 \tilde{u}_\alpha(\rho) =: \ell_\alpha$ exists, however we cannot use this information to deduce existence of the limit $\lim_{\rho \to \infty} \rho^2 \tilde{u}(\rho)$ and we have to prove it separately.

In order to study the limits \eqref{eq:limits.tildeOmega}, let us introduce
\begin{align}
    z(t) := e^{2t}\tilde{u}(e^{t}),
    \quad
    \dot z(t) := \frac{d}{dt}z(t) = 2z(t) +e^{3t} \tilde{u}'(e^{t}),
    \quad
    t \in \R.
\end{align}
By the previous observations we know that there exists a finite constant $C$ such that:
\begin{align} \label{eq:bounds.z.y}
    0<z(t) <C,
    \quad
    -C < \dot z(t) < C,
    \quad
    \forall t \in \R.
\end{align}
Using the equation \eqref{eq:limit.ODE.linearized} for $\tilde{u}$ we get a closed ODE for $z=z(t)$ (derivatives with respect to $t$ denoted with dots):
\begin{align} \label{eq:Emden.z}
   \ddot{z} + (n-4+2z) \dot{z} -  2(n-2)(z-z^2) =0, \quad
   t \in \R, 
\end{align}
with boundary conditions at $t=-\infty$:
\begin{align} \label{eq:limits.z}
    \lim_{t \to -\infty} z(t) = 0,
    \quad
    \lim_{t \to -\infty} \dot{z}(t) = 0,
    \quad
    \lim_{t \to -\infty} \frac{\dot{z}(t)}{z(t)} = 2.
\end{align}

\begin{lemma} \label{lem:z'}
  Let $t_0 \in \R$ be given and suppose $0<z(t)\leq 1$ for every $t \in (-\infty,t_0)$. Then $\dot{z}(t) > 0$ for every $t \in (-\infty,t_0]$.
\end{lemma}
\begin{proof}
By the third limit in \eqref{eq:limits.z} and $z(t)>0$ for every $t \in \R$, there exists $t_1 \in \R$ such that $\dot{z}(t)>0$ for every $t \in (-\infty,t_1]$.
If $t_1 \geq t_0$ we are done, otherwise rewrite \eqref{eq:Emden.z} in the form $\ddot{z} + h \dot{z} + g=0$ with $h(t) := n-4+2z(t)$, $g(t) := -2(n-2)(z(t)-z^2(t))$, and $H(t) := \exp \int_{t_1}^t (n-4+2z(s))ds$.
Using the analog of \eqref{eq:HOmega_alpha'}, we get for every $t \in (t_1,t_0]$ 
\begin{align} \label{eq:dotz0}
    \dot{z}(t) = \frac{1}{H(t)} \left( \dot{z}(t_1)+2(n-2)\int_{t_1}^t H(s) (z(s)-z^2(s))ds\right)  \geq \frac{\dot{z}(t_1)}{H(t)}>0.
\end{align}
\end{proof}

\begin{lemma} \label{lem:dotz.to.0}
    Let $a,b \in \R$ with at least one of the two $\neq 0$, and suppose
    $\lim_{t \to \infty} (az(t)+b) \dot{z}^2(t) = 0$.
    Then $\dot{z}(t)$ converges as $t \to \infty$ and $\lim_{t \to \infty} \dot{z}(t)=0$.
\end{lemma}
\begin{proof}
Without loss of generality we may assume $a>0$.
Suppose by contradiction that $\dot{z}(t)$ does not converge to zero: then there exists a $\delta_0>0$ such that $|\dot z(t)|> \delta$ frequently for every $0<\delta<\delta_0$. 
Now pick $0<\delta<\delta_0$ and $0<\varepsilon \ll 1$, and let $t_0=t_0(\delta,\varepsilon)$ be such that $|\dot z(t_0)|> \delta$ and $|az(t)+b| \dot{z}^2(t) < \varepsilon$ for every $t \geq t_0$. 
We distinguish two cases: either $\dot z(t_0)>\delta$ or $\dot z(t_0)<-\delta$.

\emph{Case 1: $\dot z(t_0)>\delta$}.
By \eqref{eq:bounds.z.y} and \eqref{eq:Emden.z}, $|\ddot z|$ is bounded by a finite constant $M$, and therefore we have $\dot z(t)>\dot z(t_0)/2> \delta/2$ for every $t \in [t_0,t_0+ \delta/2M]$, meaning 
\begin{align}
  z(t_0)+\frac{\delta^2}{8M}
  < 
  z(t_0+\delta/4M)
  <
  z(t_0+\delta/2M) 
  -\frac{\delta^2}{8M}.
\end{align}
This leads to a contradiction when $\varepsilon \ll 1$: Indeed, if $az(t)+b \geq 0$ at time $t:=t_0+\delta/4M$ we have
\begin{align}
    \frac{a\delta^4}{32M} -\frac{\varepsilon}{4}
    \leq
    \left(az(t_0)+\frac{a\delta^2}{8M}+b\right) \frac{\dot{z}^2(t_0)}{4} 
    \leq 
    (az(t)+b) \dot{z}^2(t) 
    < 
    \varepsilon,
\end{align}
while if $az(t)+b < 0$ we have $\dot{z}(t)>\dot{z}(t_0+\delta/2M)-\delta/4 > \dot{z}(t_0+\delta/2M)/2$ and thus
\begin{align}
    \frac{a\delta^4}{128M} -\frac{\varepsilon}{4}
    \leq
    -\left(az(t_0+\delta/2M)-\frac{a\delta^2}{8M}+b\right) \frac{\dot{z}^2(t_0+\delta/2M)}{4} 
    \leq 
    -(az(t)+b) \dot{z}^2(t) 
    < 
    \varepsilon.
\end{align}

\emph{Case 2: $\dot{z}(t_0)<-\delta$}. 
The proof is similar to the previous case; now we have $\dot{z}(t)<\dot{z}(t_0)/2<-\delta/2$ for every $t \in [t_0,t_0+\delta/2M]$ and
\begin{align}
    z(t_0+\delta/2M)+\frac{\delta^2}{8M}
  < 
  z(t_0+\delta/4M)
  <
  z(t_0) 
  -\frac{\delta^2}{8M}.
\end{align}
Then, if $az(t)+b \geq 0$ at time $t:=t_0+\delta/4M$ we have $\dot{z}(t) < \dot{z}(t_0+\delta/2M)+\delta/4<\dot{z}(t_0+\delta/2M)/2$ and therefore
\begin{align}
\frac{a\delta^4}{128M} -\frac{\varepsilon}{4}
\leq
\left(az(t_0+\delta/2M)+\frac{a\delta^2}{8M}+b\right) \frac{\dot{z}^2(t_0+\delta/2M)}{4}   
\leq 
(az(t)+b) \dot{z}^2(t) 
< 
\varepsilon,
\end{align}
while if $az(t)+b < 0$
\begin{align}
\frac{a\delta^4}{32M} -\frac{\varepsilon}{4}
\leq
-\left(az(t_0)-\frac{a\delta^2}{8M}+b\right) \frac{\dot{z}^2(t_0)}{4}   
\leq 
-(az(t)+b) \dot{z}^2(t) 
< 
\varepsilon.
\end{align}
\end{proof}

\begin{lemma} \label{lem:z.to.0or1}
If $\lim_{t \to \infty} \dot{z}(t) =0$, then ${z}(t)$ converges as $t \to \infty$ and $\lim_{t \to \infty} z(t) \in \{0,1\}$.
\end{lemma}
\begin{proof}
Take $0<\delta \ll 1$ and $t_\delta \in \R$ large enough so that $|\dot z(t)| < \delta$ for every $t \geq t_\delta$, and suppose per absurdum that exists $t_0>t_\delta$ such that $|z(t_0)-z^2(t_0) |> \delta^{1/3}$;
We focus on the case $z(t_0)-z^2(t_0)> \delta^{1/3}$, the other being similar.

Since $|z|,|\dot{z}|$ are bounded, we have
$z(t)-z^2(t) > \delta^{1/3}/2$ for every $t \in [t_0,t_0+\delta^{1/3}/2M]$ for some constant $M>0$. 
We intend to apply \eqref{eq:dotz0} with $t_1=t_0$ and $t=t_0+\delta^{1/3}/2M$: in this case, we have $e^{-(t-t_0)} \leq H(t) \leq e^{K(t-t_0)}$ for some constant $K>0$ and therefore 
\begin{align}
    \dot{z}(t_0+\delta^{1/3}/2M) \geq e^{-\frac{K\delta^{1/3}}{2M}} \left(-\delta  + \frac{n-2}{M}\delta^{2/3} e^{-\frac{\delta^{1/3}}{2M}} \right) \gtrsim \delta^{2/3}>\delta,
\end{align}
reaching a contradiction for $\delta$ sufficiently small.
\end{proof}

The proof of \autoref{prop:tails} is divided into the two cases $n \geq 4$ and $n=3$.

\begin{proof}[Proof of \autoref{prop:tails}, case $n \geq 4$]
Equation \eqref{eq:Emden.z} admits an ``energy''
    \begin{align} \label{eq:energy}
        E(t) := \frac{1}{2} \dot{z}(t)^2 +2(n-2) \left( \frac13 z^3(t) -\frac12 z^2(t) \right)
    \end{align}
that is bounded from below and decreases with respect to $t$: indeed using \eqref{eq:Emden.z} we have
    \begin{align} \label{eq:dot.E}
        \dot{E}(t) 
        &= \dot{z} \ddot{z} + 2(n-2) (z^2-z) \dot z
        =-(n-4+2z) \dot{z}^2 \leq 0.
    \end{align}
Here we have used $n-4+2z \geq 0$. Moreover, since $z(t)>0$ for every $t \in \R$, the inequality is strict unless $\dot{z}(t)=0$. In particular, by the third limit in \eqref{eq:limits.z} there exists $t_1 \in \R$ with $\dot{E}(t_1)<0$, and the energy cannot be globally constant.

Since $E(t)$ and $|\ddot{z}|$ are bounded, it must be $(n-4+2z) \dot{z}^2 \to 0$ as $t \to \infty$. By \autoref{lem:dotz.to.0} and \autoref{lem:z.to.0or1}, we know that $(z,\dot z)$ converges for $t \to \infty$ and the limit is either $(0,0)$ or $(1,0)$.
However the first case cannot occur, as it would imply that $E$ is constant since $\lim_{t \to -\infty} E(t)=0$ by \eqref{eq:limits.z}.
\end{proof}

It remains to treat the case $n=3$. 
In this case, \eqref{eq:Emden.z} reduces to 
\begin{align}
    \ddot{z} = (1-2z) \dot{z} + 2(z-z^2).
\end{align}
It is convenient to introduce the new variable $\zeta(t):= \dot{z}(t)+z(t)$, which allows to rewrite the above as a system of ODEs
    \begin{align} \label{eq:system.ODE.zu}
    \begin{dcases}
      \dot{z} = \zeta-z,
      \\
      \dot{\zeta} = 2(1-z) \zeta.  
    \end{dcases}
    \end{align}
The key observation following from this formulation is that the quantity $\zeta(t)>0$ for every $t \in \R$; Indeed, it holds $\zeta(t_1)>0$ for at least one $t_1 \in \R$ (cf. for instance the first line of the proof of \autoref{lem:z'}) and the second equation in \eqref{eq:system.ODE.zu} guarantees $\zeta(t) \neq 0$ for every $t \in \R$, by uniqueness of solutions. 
Since $\zeta(t)>0$ for every $t \in \R$, we can equivalently rewrite the system for $z(t)$ and $v(t) := \log \zeta(t)$ in the equivalent form
\begin{align} \label{eq:system.ODE.zv}
    \begin{dcases}
      \dot{z} = e^v-z,
      \\
      \dot{v} = 2(1-z). 
    \end{dcases}
\end{align}

\begin{lemma} \label{lem:limsup<1}
If $n=3$, then it is impossible to have $\limsup_{t \to \infty} z(t) <1$.
\end{lemma}
\begin{proof}
Suppose per absurdum $\limsup_{t \to \infty} z(t) = a<1$, then for every $\varepsilon \in (0,1-a)$ there exists $t_0 \in \R$ such that $z(t) \leq 1-\varepsilon$ for every $t \geq t_0$. Then the second equation in \eqref{eq:system.ODE.zv} informs us that
\begin{align}
     v(t) \geq v(t_0) + 2\varepsilon(t-t_0),
     \quad
     \forall t \geq t_0.
\end{align}
Sending $t \to \infty$ we obtain $v(t) \to +\infty$, or equivalently $\zeta(t) \to + \infty$ as $t \to \infty$; but this is absurd because $\zeta = \dot{z}+z$ must be bounded. 
\end{proof}

\begin{lemma} \label{lem:no.periodic.orbits}
    Let $n=3$, $t_0 \in \R$ and denote $\mathbf{z} := (z(t_0),\dot{z}(t_0))$ and $\omega(\mathbf{z})$ the $\omega$-limit set of $\mathbf{z}$.
    If $\omega(\mathbf{z})$ is a periodic orbit, then $\omega(\mathbf{z})  = \{(1,0)\}$.
\end{lemma}
\begin{proof}
Since we know that $u(t):=\dot{z}(t)+z(t)>0$ for every $t \in \R$, we must have
\begin{align}
    \omega(\mathbf{z}) \subset U:= \{(x,y) \in \R^2 : x+y \geq 0 \}.
\end{align}
Any periodic orbit of the ODE \eqref{eq:Emden.z} contained in $U$ corresponds to a periodic orbit of the system \eqref{eq:system.ODE.zu} contained in the region of space $\{\zeta \geq 0\}$, hence it suffices to characterize the periodic orbits of the latter.

We claim that any non-constant periodic orbit of \eqref{eq:system.ODE.zu} contained in $\{\zeta \geq 0\}$ must satisfy $\zeta(t) \equiv 0$ for every $t \in \R$.
First of all, notice that if $\zeta(t)=0$ for some $t \in \R$ then $\zeta(t) \equiv 0$, thus let us suppose that $\zeta(t)>0$ for every $t \in \R$ and the variable $v(t) := \log \zeta(t)$ is well-defined.
Since $\partial_z (e^v-z) + \partial_v (2(1-z)) \equiv -1$, Bendixson–Dulac Theorem excludes the existence of a non-constant periodic solution of \eqref{eq:system.ODE.zu} lying entirely within the region $\zeta \neq 0$.

Thus we know that, if $\omega(\mathbf{z})$ is a periodic orbit, it is either constant or must be contained in the region of space $\omega(\mathbf{z}) \subset \{(x,y) \in \R^2 : x+y = 0 \}$. But in both cases it is easy to check, for instance by \autoref{lem:dotz.to.0} and \autoref{lem:z.to.0or1}, that it must be $\omega(\mathbf{z}) \subset \{(0,0),(1,0)\}$. Finally, the case $\omega(\mathbf{z})  = \{(0,0)\}$ implies $(z,\dot{z}) \to (0,0)$ and is ruled out by \autoref{lem:limsup<1}.
\end{proof}

The following is a direct consequence of \autoref{lem:no.periodic.orbits}.
\begin{cor} \label{cor:no.intersections}
Let $n=3$ and suppose there exist two distinct times $\tau_0, \tau_1 \in \R$ in which the trajectory of $(z,\dot{z})$ self-intersects at the point $z(\tau_0)=z(\tau_1) =: z_{\times}$ and $\dot{z}(\tau_0)=\dot{z}(\tau_1)=: \dot{z}_\times$. Then $(z(t),\dot{z}(t)) \equiv (z_{\times},\dot{z}_{\times}) = (1,0)$.
\end{cor}
\begin{proof}
Since the system \eqref{eq:Emden.z} governing the evolution of $(z,\dot{z})$ is autonomous with respect to the variable $t$ and solutions are unique, the assumption of the corollary implies that the map $t \mapsto (z(t),\dot{z}(t))$ is periodic with period $|\tau_1-\tau_0|$. In this case, the $\omega$-limit set $\omega(\mathbf{z})$ coincides with the whole trajectory $\Gamma := \{ (z(t),\dot{z}(t)) \in \R^2 \,:\, t \in \R\}$, and the thesis follows by \autoref{lem:no.periodic.orbits}.
\end{proof}

Observe that the critical points of \eqref{eq:Emden.z}, seen as a dynamical system in the $(z,\dot{z})$ variables, are $(0,0)$ (a saddle) and $(1,0)$ (a sink).

Let us fix any $t_0 \in \R$ and let us introduce $\mathbf{z}$ and $\omega(\mathbf{z})$ as in \autoref{lem:no.periodic.orbits} above.
Since the orbit of $\mathbf{z}$ is bounded by \eqref{eq:bounds.z.y}, then by Poincaré–Bendixson Theorem one of the following holds:
\begin{itemize}
    \item[$i$)] $\omega(\mathbf{z})$ is a critical point;
    \item[$ii$)] $\omega(\mathbf{z})$ is a union of a finite number of critical points and homoclinic loops or heteroclinic orbits;
    \item[$iii$)] $\omega(\mathbf{z})$ is a periodic orbit.
\end{itemize}
By \autoref{lem:no.periodic.orbits}, the third scenario above actually coincides with the first one. 
Moreover, as $(0,0)$ is the only saddle point for $(z,\dot{z})$, the closure of each heteroclinic orbit or homoclinic loop must contain $(0,0)$. If we can prove that $(0,0) \notin \omega(\mathbf{z})$, then scenario $ii$) is impossible and the only remaining alternative is case $i)$.
By \autoref{lem:limsup<1}, the first scenario can happen if and only if $\omega(\mathbf{z}) = \{(1,0)\}$, in other words $(z,\dot{z}) \to (1,0)$.

\begin{proof}[Proof of \autoref{prop:tails}, case $n=3$]
By the discussion above, it is sufficient to check $(0,0) \notin \omega(\mathbf{z})$.

\emph{Step 1}.
First, by \autoref{lem:limsup<1} we have $\tau_0 := \inf \{ t \in \R : z(t)=1/2\} < \infty$ and by \autoref{lem:z'} it holds $\dot z(\tau_0)>0$.
Let us define $t_0 :=\inf \{ t >\tau_0 : \dot{z}(t)=0\}$.
Clearly, $t_0>\tau_0$, since $\dot z(\tau_0)>0$.
If $t_0$ is infinite then $\dot{z}(t)>0$ and $z(t) \geq 1/2$ for every $t \geq \tau_0$, and there is nothing to prove; so let us assume $t_0<\infty$.
By a version of \eqref{eq:dotz0} applied between times $\tau_0$ and $t_0$ and by minimality of $t_0$, it must be $1 \leq z(t_0)$.
If $1 = z(t_0)$ then $(z(t_0),\dot{z}(t_0))=(1,0)$ and therefore $(z(t),\dot{z}(t)) \equiv(1,0)$, so $1 < z(t_0)$ must be true.
Again by \eqref{eq:dotz0} this implies that $\dot{z}(t)<0$ for $t$ in a right neighborhood of $t_0$.

Now let us introduce $t_1 :=\inf \{ t >t_0 : \dot{z}(t)=0\}>t_0$.
We distinguish two cases: If $t_1$ is infinite, then $\dot{z}(t)<0$ for every $t > t_0$, and since $z$ is bounded this implies $\dot{z}(t) \to 0$ as $t \to \infty$; by \autoref{lem:z.to.0or1} and \autoref{lem:limsup<1} it must be $z(t) \to 1$ and the thesis follows; if instead $t_1<\infty$, then $0 <z(t_1) < 1$ (arguing as above, using \eqref{eq:dotz0} and the minimality of $t_1$) and $\dot{z}(t)>0$ for $t$ in a right neighborhood of $t_1$.

\emph{Step 2}. 
Next, we want to iterate the construction of Step 1.
With a little abuse of notation, denote $t_{-1}:= -\infty$ and ${z}(t_{-1})=\dot{z}(t_{-1})=0$.

Suppose we are given times $t_{2k-3}<t_{2k-2}<t_{2k-1}$ such that $\dot{z}(t_{2k-3})=\dot{z}(t_{2k-2})=\dot{z}(t_{2k-1})=0$, $\dot{z}(t)>0$ for every $t \in (t_{2k-3},t_{2k-2})$, and $\dot{z}(t)<0$ for every $t \in (t_{2k-2},t_{2k-1})$.
Suppose in addition $z(t_{2k-3})<z(t_{2k-1})<1$ and $1<z(t_{2k-2})$.

Introduce $t_{2k} := \inf \{ t >t_{2k-1} : \dot{z}(t)=0\}>t_{2k-1}$; if $t_{2k}$ is infinite, then $\dot{z}(t)>0$ for every $t > t_{2k-1}$, thus $\dot{z}(t) \to 0$ as $t \to \infty$ and hence $z(t) \to 1$ by \autoref{lem:z.to.0or1} and \autoref{lem:limsup<1}.
If $t_{2k}< \infty$, then the same argument as Step 1 gives $1<z(t_{2k})$ and $\dot{z}(t)<0$ for $t$ in a right neighborhood of $t_{2k}$.
Moreover, by \autoref{cor:no.intersections} it must be
\begin{align}
    1<z(t_{2k}) < z(t_{2k-2}),
\end{align}
because otherwise the curves $[t_{2k-3},t_{2k-2}] \ni t \mapsto (z(t),\dot{z}(t))$ and $[t_{2k-1},t_{2k}] \ni s \mapsto (z(s),\dot{z}(s))$ would intersect, implying $(z,\dot{z}) \equiv (1,0)$.

Similarly, $t_{2k+1} := \inf \{ t >t_{2k} : \dot{z}(t)=0\}>t_{2k}$; if $t_{2k+1}$ is infinite, then $\dot{z}(t)<0$ for every $t > t_{2k}$ and $(z,\dot{z})\to (1,0)$ by the usual argument;
On the other hand, if $t_{2k+1}< \infty$ then
\begin{align}
    z(t_{2k-1}) < z(t_{2k+1})<1,
\end{align}
and $\dot{z}(t)>0$ for $t$ in a right neighborhood of $t_{2k+1}$. 

\emph{Step 3}.
Given the construction of Step 2, we distinguish two cases: either there exists a finite index $m \in \NN$ such that $t_m$ is infinite, in which case we have proved that $(z,\dot{z})\to (1,0)$ and the thesis is true, or we have an infinite sequence of times $\{t_m\}_{m \in \NN}$ that is strictly increasing and has the properties listed above.

Now there are two subcases: if $t_m \to t_{\infty} \in \R$ as $m \to \infty$, then by continuity it holds $\dot{z}(t_\infty)=0$ and
\begin{align}
\lim_{k \to \infty} z(t_{2k+1}) \leq 1 \leq \lim_{k \to \infty} z(t_{2k}).
\end{align}
But since $|\dot{z}|$ is bounded neither of the two inequalities above can be strict, and therefore $z(t_m) \to z(t_{\infty})=1$.
Since we are assuming $t_{\infty} \in \R$ this implies $(z,\dot{z}) \equiv (1,0)$, which is a contradiction.
We deduce that the sequence $\{t_m\}_{m \in \NN}$ must monotonically increase to infinity.
Now suppose the point $(0,0) \in \omega(\mathbf{z})$, then for every $\varepsilon>0$ there exists a sequence of times $\{s_m\}_{m \in \NN}$ diverging to infinity such that $|(z(s_m),\dot{z}(s_m))| < \varepsilon$ for every $m$.
Then for every $m$ there exists $k = k(m)$ such that $s_m \in [t_{2k-1},t_{2k+1})$.
But by construction
\begin{align}
    0<z(t_1)<\dots <z(t_{2k-1}) \leq z(s),
    \quad
    \forall s \in [t_{2k-1},t_{2k+1}).
\end{align}
Hence $|(z(s_m),\dot{z}(s_m))| \geq z(s_m) \geq z(t_1)$, reaching a contradiction for $\varepsilon \ll 1$.
\end{proof}

\subsubsection{Proof of \autoref{prop:f.zero}}
As discussed above, it is sufficient to show that the limit $\tilde{f}$ vanishes in at least $k$ different points $0<\rho_1<\dots<\rho_k$.

By \autoref{prop:tails}, we know that for every $\varepsilon>0$ there exists $\rho_\varepsilon>0$ such that for every $\rho \geq \rho_\varepsilon$ it holds
\begin{align} \label{eq:key.assumption}
    \frac{1-\varepsilon}{\rho^2} < \tilde{u}(\rho) < \frac{1+\varepsilon}{\rho^2},
    \quad
    -2\frac{1+\varepsilon}{\rho^3} < \tilde{u}'(\rho) 
    < -2\frac{1-\varepsilon}{\rho^3}.
\end{align}
In particular, \eqref{eq:key.assumption} above implies
\begin{align}
0
<
\frac{n+1+2(1-\varepsilon)}{\rho} 
&< 
\frac{n+1}{\rho} +2 \rho\tilde{u}(\rho) 
< 
\frac{n+1+2(1+\varepsilon)}{\rho},
\\
\frac{4n(1-\varepsilon)-4(1+\varepsilon)}{\rho^2} 
&< 
4n\tilde{u}(\rho) + 2 \rho \tilde{u}'(\rho).     
\end{align}
Let us therefore fix $0<\varepsilon' \ll 1$ and compare the ODE satisfied by $\tilde{f}$: 
\begin{align}
    \tilde{f}''
    +
    \left( \frac{n+1}{\rho} +2 \rho\tilde{u} \right) \tilde{f}'
    +
    \left( 4n\tilde{u}
    +
    2 \rho \tilde{u}'\right) \tilde{f}
    =
    0,
\end{align}
with the ODE
\begin{align} \label{eq:Euler}
    y''
    +
    \frac{A}{\rho}y'
    +
    \frac{B}{\rho^2} y
    =
    0,
    \quad
    A := n+3,
    \quad
    B := 4(n-1) -\varepsilon'.
\end{align}
Using the transformations
\begin{align}
    \tilde{f}_\star (\rho) 
    := 
    \tilde{f}(\rho) \exp \left( \frac12 \int_{\rho_\varepsilon}^\rho \left( \frac{n+1}{s} +2 s\tilde{u}(s) \right)ds \right),
    \quad
    y_\star (\rho) 
    := 
    y(\rho) \exp \left( \frac12 \int_{\rho_\varepsilon}^\rho \frac{A}{s} ds \right),
\end{align}
both equations for $\tilde{f}$ and $y$ can be written in normal form 
\begin{align}
    \tilde{f}_\star'' + \tilde{V} \tilde{f}_\star = 0,
    \quad
    y_\star'' + V y_\star = 0,
\end{align}
without changing their positive roots, with $\tilde{V},V$ given by
\begin{align}
\tilde{V}(\rho) &:= 
4n\tilde{u} 
+
2 \rho \tilde{u}'
-
\frac12 \left( \frac{n+1}{\rho} +2 \rho\tilde{u} \right)'
-
\frac14 \left( \frac{n+1}{\rho} +2 \rho\tilde{u} \right)^2
\\
&=
4n\tilde{u} 
+
2 \rho \tilde{u}'
+\frac{n+1}{2\rho^2} - \tilde{u} - \rho\tilde{u}' 
 -\frac{(n+1)^2}{4\rho^2} - (n+1)\tilde{u} - \rho^2\tilde{u}^2 ,
    \\
V(\rho) &:= \frac{B}{\rho^2} + \frac{A}{2\rho^2}- \frac{A^2}{4\rho^2}
=
\left(B+\frac{A}{2}-\frac{A^2}{4} \right) \rho^{-2}.
\end{align}

By definition of $A$ and $B$, for every $0<\varepsilon \ll 1$ we can choose $\varepsilon'>0$ such that $\tilde{V} (\rho)\geq V(\rho)$ for every $\rho \geq \rho_\varepsilon$.
Therefore, if we are able to show that $y_\star$ (equivalently $y$) vanishes in at least $k+1$ distinct points $\rho_\varepsilon<\tilde{\rho}_1<\dots<\tilde{\rho}_{k+1}$ we can apply Sturm-Picone comparison Theorem to deduce that $\tilde{f}_\star$ (equivalently $\tilde{f}$) must vanish in at least $k$ distinct points $ {\rho}_1<\dots< {\rho}_{k}$ satisfying $\tilde{\rho}_i \leq \rho_i \leq \tilde{\rho}_{i+1}$.

Equation \eqref{eq:Euler} for $y$ is called \emph{Euler equation}, see
\cite[Chapter 2.1.2, Equation 123]{PoZa03}.
As already mentioned in the introduction, if $(A-1)^2-4B<0$ it has oscillating solutions of the form
\begin{align}
    y(\rho) =  C_1 |\rho|^\frac{1-A}{2}\sin(\mu \log|\rho|)
    +
    C_2 |\rho|^\frac{1-A}{2}\cos(\mu \log|\rho|) ,
    \quad
    \mu:= \frac12 |(A-1)^2-4B|^{1/2},
    \quad
    C_1,C_2 \in \R.
\end{align}
In particular, we can apply Sturm-Picone comparison Theorem \cite[pp.\,19-21]{picone1910sui} with 
\begin{align}
   y_\star(\rho) := \rho^\frac{1-A}{2}\sin(\mu \log \rho) \exp \left( \frac12 \int_{\rho_\varepsilon}^\rho \frac{A}{s} ds \right), 
   \quad
   \rho \geq \rho_\varepsilon,
\end{align}
which has infinitely many roots after $\rho_\varepsilon$.
Let us see when the condition $(A-1)^2-4B<0$ holds. Since $\varepsilon$ is arbitrarily small, it suffices to verify the condition with $\varepsilon=\varepsilon'=0$; in this case we have $A=n+3$, $B=4(n-1)$ and
\begin{align}
    (A-1)^2-4B
    =
    (n+2)^2 - 16(n-1) < 0
    \iff
    2 < n < 10.
\end{align}

\section{Spectral analysis and instability}\label{sec:spectral}

In this section, we prove \autoref{thm:instability.intro} and establish additional properties of the function spaces introduced in \autoref{sec:strategy}, as well as of $S_{\alpha}$, the semigroup generated by $L_{\alpha}$.
More precisely, in \autoref{subsec:preliminaries} we show that $A$ is an isomorphism and derive Sobolev and Hardy inequalities in $Y^q$. 
Next, in \autoref{sec:semigruoup}, we prove that $L_\alpha$ generates a strongly continuous semigroup on both $Y^{q,r}$ and $Y_{\nabla}^{q,r}$, whose infinitesimal generator differs from $L_0$ only by a relatively compact perturbation. This result, combined with the ODE analysis from the previous section and Sturm–Liouville oscillation theory, is then used in \autoref{subsec_unstable_eigen} to reduce the existence of unstable eigenvalues of $L_{\alpha}$ to counting the zeros of \eqref{eq:ODE.linearized}. The latter are guaranteed to exist by \autoref{prop:f.zero}, ultimately yielding the proof of \autoref{thm:instability.intro}. Finally, once the optimal growth bound for $S_{\alpha}$ in the presence of unstable eigenvalues has been established, we investigate its regularizing properties in \autoref{prop:smoothing}.

\subsection{Isomorphism properties and functional inequalities in $Y^q$}\label{subsec:preliminaries}
Let $p\in [1,+\infty]$. In the following, we denote by
\begin{align}
 \iota : Y^{p} \to L^{p}(\R_+,r^{n-1}dr) \to X^{p}   
\end{align}
the isomorphism obtained by identifying radial functions with their radial profiles. More precisely, for a radial function $f\in Y^p$, we define
\begin{align*}
    (\iota f)(\theta) := f(|\theta|e_1)\qquad\theta\in \R^n,
\end{align*}
where $e_1\in \R^{n+2}$ is any fixed unit vector. With a slight abuse of notation, we do not indicate explicitly the dependence of $\iota$ on $p$, and we use the same symbol to denote the corresponding isomorphism between $Y^{p_1,p_2}$ and $X^{p_1,p_2}$, defined analogously.
\\ Next we need to study the properties of the operator $A$ defined on smooth, radial functions of $\R^n$ as 
\begin{align}
    A[c](y) := \frac{1}{2\omega_{n-1}|y|^n} \int_{B(0,|y|)} c(x)dx,
    \quad
    y\in \R^{n+2}.
\end{align}
This is the content of the following lemma, see also \autoref{rmk:isomorphism}.
\begin{lemma} \label{lem:isomorphism}
For every $p \in (1,\infty)$ the map $A : X^p \to Y_{\nabla}^p$ is a linear isomorphism of Banach spaces.
\end{lemma}
\begin{proof}
For notational simplicity we drop the superscript $p$ below.
Given $\Xi \in X$ and $\Omega \in Y$, we denote the radial parts
$\Xi(\theta)=:\Xi_{\mathrm{rad}}(|\theta|)$ and $\Omega(\xi) =: \Omega_{\mathrm{rad}}(|\xi|)$.
The map $\Xi \mapsto \Xi_{\mathrm{rad}}$ (resp. $\Omega \mapsto \Omega_{\mathrm{rad}}$) is a linear isomorphism between $X$ (resp. $Y$) and $L^{p}(\R_+,r^{n-1}dr)$, as can be immediately verified by integration in polar coordinates.

By definition, the map $A$ is linear and $\Omega:=A[\Xi]$ is a radial function for every $\Xi \in L^1_{loc}(\R^n,d\theta)$. Let us check $\Omega \in Y$, or equivalently $\Omega_{\mathrm{rad}} \in L^{p}(\R_+,r^{n-1}dr)$.
First of all, notice that 
\begin{align} \label{eq:definition.T}
\Omega_{\mathrm{rad}}(r) &= \frac{1}{2r^n} \int_0^r \Xi_{\mathrm{rad}}(s) s^{n-1} ds;
\end{align}
let us rewrite the integral above via a change of variables $h:= r^n$ and $k := s^n$ to obtain
\begin{align}
\Omega_{\mathrm{rad}}(h^{1/n}) 
&= 
\frac{1}{2h} \int_0^{h^{1/n}} \Xi_{\mathrm{rad}}(s) s^{n-1} ds
= 
\frac{1}{2hn} \int_0^{h} \Xi_{\mathrm{rad}}(k^{1/n}) dk.
\end{align}
Therefore, by the previous expression and Hardy inequality we have 
\begin{align}
 \int_0^\infty |\Omega_{\mathrm{rad}}(r)|^p r^{n-1}dr
 &=
 \frac{1}{n} \int_0^\infty |\Omega_{\mathrm{rad}}(h^{1/n})|^p dh
\\
&=
\frac{1}{n(2n)^p} \int_0^\infty \left| \frac{1}{h}\int_0^h \Xi_{\mathrm{rad}}(k^{1/n})dk\right|^p dh
\\
&\leq
\frac{p^p}{n(2n(p-1))^p} \int_0^\infty  |\Xi_{\mathrm{rad}}(k^{1/n})|^p dk
\\
&=
\frac{p^p}{(2n(p-1))^p} \int_0^\infty  |\Xi_{\mathrm{rad}}(s)|^p s^{n-1}ds.
\end{align}
The inequality above shows that $A$ maps $X$ into $Y$ continuously. In order to verify $\xi \cdot \nabla \Omega \in Y$ we can equivalently check that $r\Omega_{\mathrm{rad}}' \in L^{p}(\R_+,r^{n-1}dr)$; by \eqref{eq:definition.T} it holds $r\Omega_{\mathrm{rad}}'(r) = -n\Omega_{\mathrm{rad}}(r) + \Xi_{\mathrm{rad}}(r)/2 \in L^{p}(\R_+,r^{n-1}dr)$, and therefore $\Omega \in Y_\nabla$ and the map $A:X \to Y_\nabla$ is well-defined and continuous.

Let us now show that $A$ is bijective with continuous inverse.
Given $\Omega \in Y_\nabla$ with radial part $\Omega_{\mathrm{rad}}$, let $\Xi_{\mathrm{rad}}(r) := 2n \Omega_{\mathrm{rad}}(r)+2r\Omega_{\mathrm{rad}}'(r) \in L^{p}(\R_+,r^{n-1}dr)$; then the function $\Xi$ with radial part $\Xi_{\mathrm{rad}}$ belongs to $X$ and it holds $A[\Xi]=\Omega$, hence $A$ is surjective. The same formula $\Xi_{\mathrm{rad}}(r) = 2n \Omega_{\mathrm{rad}}(r)+2r\Omega_{\mathrm{rad}}'(r)$ also shows that $\mathrm{ker}(A) = \{0\}$ and thus $A$ is injective. 
Continuity of the inverse is also a direct consequence of the explicit formula above, indeed:
\begin{align} \label{eq:A-1}
    \| A^{-1}[\Omega]\|_{X}
    =
    \| 2n \Omega_{\mathrm{rad}}+2r\Omega_{\mathrm{rad}}'\|_{L^{p}(\R_+,r^{n-1}dr)}
    \lesssim \|\Omega\|_{Y_\nabla}.
\end{align}
\end{proof}

\begin{rmk}\label{rmk:isomorphism}
Clearly if $\eta,\gamma \in (1,\infty)$ then $A : X^{\eta,\gamma} \to Y_\nabla^{\eta,\gamma}$ remains a linear isomorphism of Banach spaces. 
Moreover, when $\eta=1$ one has from the explicit formula \eqref{eq:A-1} that the operator $A^{-1} :Y^{1,\gamma}_\nabla \to X^{1,\gamma}$ is well-defined and continuous, although in general $A: X^{1,\gamma} \nsubseteq Y^{1,\gamma}$: for instance, if $\chi$ is the indicator function of the unit ball of $\R^n$, then $\chi \in X^{1,\gamma}$ but $A[\chi] \notin Y^1$.
\end{rmk}

In the following we will also employ the embeddings between weighted spaces below.
\begin{lemma}\label{weighted_sobolev}
    Let $p\in [1,n)$. Then $Y_1^{p}\hookrightarrow Y^{p^*}$ continuously with $p^*$ satisfying 
    \begin{align*}
        \frac{1}{p^*}=\frac{1}{p}-\frac{1}{n}.
    \end{align*}
    Moreover 
    \begin{align*}
        \int_{\R^{n+2}} \frac{1}{|y|^{p+2}}|f(y)|^p dy\lesssim \|f\|_{Y^{p}_1}^p.
    \end{align*}
\end{lemma}
\begin{proof}
    For $f\in Y_1^{p}$, it is easy to check that
    \begin{align*}
        \iota f\in W^{1,p}(\R^n).
    \end{align*}
Therefore, by standard Sobolev embedding and Hardy's inequality the claims follow since
    \begin{align*}
        \|f\|_{Y^{p^*}}&\lesssim\|\iota f\|_{X^{p^*}}\lesssim \|\iota f\|_{W^{1,p}(\R^n)}\lesssim \|f\|_{Y_1^{p}},\\
        \int_{\R^{n+2}} \frac{1}{|y|^{p+2}}|f(y)|^p dy&\lesssim\int_{\R^n}\frac{1}{|x|^p}|\iota f(x)|^p dx\lesssim  \|\iota f\|_{W^{1,p}(\R^n)}\lesssim \|f\|_{Y_1^{p}}.
    \end{align*}
\end{proof}
Lastly, we need some properties of the heat semigroup on $Y^p$ spaces.
\begin{lemma}\label{lem:heat_semigroup_weighted}

Let $p\in [1,+\infty)$ and $f\in Y^p$, then 
\begin{align} \label{eq:heat.bounded}
 e^{\Delta t}f\in C([0,\infty);Y^p),
 \quad
 \sup_{t \geq 0} \|e^{\Delta t}f\|_{Y^p}&\lesssim \|f\|_{Y^p}.
\end{align}

Moreover, if $p' \in [p,+\infty)$ then we have the ultracontractivity and smoothing estimates
\begin{align} \label{eq:heat.ultracontractive}
   \sup_{t \geq 0} t^{\frac{n}{2}\left(\frac{1}{p}-\frac{1}{p'}\right)} \|e^{\Delta t}f\|_{Y^{p'}}
   &\lesssim
   \|f\|_{Y^p},
   \\
   \label{eq:heat.derivative.bounded}
   \sup_{t \geq 0} t^{\frac{k}{2}+\frac{n}{2}\left(\frac{1}{p}-\frac{1}{p'}\right)} \||\nabla|^k e^{\Delta t}f\|_{Y^{p'}}
   &\lesssim
   \|f\|_{Y^p},
   \quad
   k =1,2,3. 
\end{align}

Furthermore, if $f\in Y^p_1$ then  
\begin{align} \label{eq:heat.bounded.Y1}
 e^{\Delta t}f\in C([0,\infty);Y^p_1),
 \quad
 \sup_{t \geq 0} \|e^{\Delta t}f\|_{Y_1^p}&\lesssim \|f\|_{Y_1^p}.
\end{align}

\end{lemma}
\begin{proof}
Let us denote $\mathcal{G}_t(\xi) := (4\pi t)^{-\frac{n+2}{2}} e^{-\frac{|\xi|^2}{4t}}$ the heat kernel on $\R^{n+2}$ and $w(\xi) := |\xi|^{-2}$. Let us preliminarily prove
\begin{align} \label{eq:bound.weight}
    (\mathcal{G}_t \ast w)(\xi) \lesssim w(\xi),
    \quad
    \forall t > 0, \, \forall \xi \neq 0,
\end{align}
with implicit constant independent of $t,\xi$. 
By changing variables $(\mathcal{G}_t \ast w)(\xi) = t^{-1}(\mathcal{G}_1 \ast w)(\xi/\sqrt{t})$ and therefore it suffices to prove the claim for $t=1$.
Let us decompose
\begin{align}
(\mathcal{G}_t \ast w)(\xi) \lesssim     \int_{\R^{n+2} } e^{-\frac{|\xi-\zeta|^2}{4}} w(\zeta) d\zeta
    &\lesssim
    \int_{|\zeta| \leq |\xi|/2}  e^{-\frac{|\xi-\zeta|^2}{4}} w(\zeta) d\zeta
    +
    \int_{|\zeta| > |\xi|/2}e^{-\frac{|\xi-\zeta|^2}{4}} w(\zeta) d\zeta
    \\
    &\lesssim
   e^{-\frac{|\xi|^2}{16}}  \int_{|\zeta| \leq |\xi|/2}  w(\zeta) d\zeta
    +
    |\xi|^{-2} 
    \\
    &\lesssim
   e^{-\frac{|\xi|^2}{16}} |\xi|^n 
    +
    |\xi|^{-2}
    \lesssim w(\xi).
\end{align}
By Fubini Theorem, \eqref{eq:bound.weight} above, and since the heat semigroup is positivity preserving we get
\begin{align}
    \int_{\R^{n+2}} |\mathcal{G}_t \ast f|(\xi) w(\xi) d\xi
    &\lesssim
    \int_{\R^{n+2}} (\mathcal{G}_t \ast |f|)(\xi) w(\xi) d\xi
    \\&=
    \int_{\R^{n+2}} |f(\xi)| (\mathcal{G}_t \ast w)(\xi)  d\xi
    \lesssim
    \int_{\R^{n+2}} |f|(\xi) w(\xi)  d\xi,
\end{align}
establishing boundedness of the heat semigroup in $Y^1$. Since the heat semigroup is bounded in $L^\infty$, Stein-Weiss interpolation Theorem \cite[Theorem 14.3.1]{HyvNVeWe23} gives boundedness in every $Y^p$. Strong continuity in $Y^p$ for $p<\infty$ then follows by this uniform bound, density of smooth functions, and dominated convergence.

As for \eqref{eq:heat.ultracontractive}, by interpolation it suffices to show $\sup_{t \geq 0} t^\frac{n}{2} \| e^{\Delta t} f \|_{L^\infty} \lesssim \| f \|_{Y^1}$.
To prove this bound, we use radiality of $f$ and represent the heat semigroup by means of the transition semigroup of the Bessel process in $\R^{n+2}$, which has density:
\begin{align}
    p_{2t}^{n+2}(r,s) :=
    (2t)^{-1} (s/r)^{n/2} s e^{-\frac{(r^2+s^2)}{4t}} I_{n/2}\left( \frac{rs}{2t} \right),
    \quad
    \forall r,s >0,
\end{align}
where $I_{n/2}$ is the modified Bessel function of the first kind of index $n/2$, cf. \cite[Chapter XI, page 446]{revuzyor99}. Therefore, denoting $f_{\mathrm{rad}}$ the radial part of $f$, we have for every $r:=|\xi|>0$
\begin{align} \label{eq:bessel.auxiliary}
|(\mathcal{G}_t \ast f)(\xi)|
\leq
\int_{\R^{n+2}} (4\pi t)^{-\frac{n+2}{2}} e^{-\frac{|\xi-\zeta|^2}{4t}} |f(\zeta)| d\zeta
=
\int_0^\infty p_{2t}^{n+2}(r,s) |f_{\mathrm{rad}}(s)| ds.
\end{align}
Since we want to obtain a bound in terms of the $Y^1$ norm of $f$, we artificially introduce a weight $s^{n-1}$ in the last integral above and define
\begin{align}
    q_{2t}^{n+2}(r,s) := s^{1-n} p_{2t}^{n+2}(r,s).
\end{align}
By simple algebraic manipulations we get 
\begin{align}
    q_{2t}^{n+2}(r,s) = (2t)^{-n/2} q_{1}^{n+2}(r/\sqrt{2t},s/\sqrt{2t}),
\end{align}
and hence $\sup_{r,s>0} q_{2t}^{n+2}(r,s) \lesssim t^{-n/2}\sup_{r,s>0} q_{1}^{n+2}(r,s)$.
Using the bound $I_{n/2}(z) \lesssim z^{n/2}e^z(1+z)^{-\frac{n+1}{2}}$, see for instance \cite[Eqs. 10.30.1 and 10.30.4]{OlverEtAl2010}, we can control the last supremum by
\begin{align}
\sup_{r,s>0} q_{1}^{n+2}(r,s)
&\lesssim
\sup_{r,s>0}  (1/sr)^{n/2} s^2 e^{-\frac{(r^2+s^2)}{2}} I_{n/2}\left( rs \right)
\\
&\lesssim
\sup_{r,s>0}  \frac{s^2}{(1+rs)^{\frac{n+1}{2}}} e^{-\frac{|r-s|^2}{2}}.
\end{align}
If $r > s/2$ then we can control the quantity above with
\begin{align}
\sup_{s>0,\, r>s/2} q_{1}^{n+2}(r,s)
\lesssim
\sup_{s>0} \frac{s^2}{(1+s^2)^\frac{n+1}{2}} < \infty;
\end{align}
On the other hand, if $r \leq s/2$ then 
\begin{align}
\sup_{s>0,\, 0<r\leq s/2} q_{1}^{n+2}(r,s)
\lesssim
\sup_{s>0} s^2 e^{-\frac{s^2}{8}}< \infty.
\end{align}
All in all we obtain $\sup_{r,s>0} q_{2t}^{n+2}(r,s) \lesssim t^{-n/2}$. 
Inserting this bound in \eqref{eq:bessel.auxiliary} above, we get
\begin{align}
 |(\mathcal{G}_t \ast f)(\xi)|
 \lesssim
 t^{-\frac{n}{2}} 
 \int_0^\infty |f_{\mathrm{rad}}(s)| s^{n-1}ds
 \lesssim
 t^{-\frac{n}{2}} \| f \|_{Y^1}.
\end{align}

The bound \eqref{eq:heat.derivative.bounded} follows by the previous bounds \eqref{eq:heat.bounded}, \eqref{eq:heat.ultracontractive} and observing that for $k=1,2,3$
\begin{align}
    ||\nabla|^k \mathcal{G}_t (\xi)| 
    \lesssim 
    t^{-\frac{k}{2}} \left( 1+ |\xi/\sqrt{t}|^3 \right)  \mathcal{G}_t (\xi)
    \lesssim 
    t^{-\frac{k}{2}}  \mathcal{G}_{2t} (\xi),
\end{align}
with implicit constant independent of $t>0$ and $\xi \in \R^{n+2}$.

Finally, the bound in \eqref{eq:heat.bounded.Y1} follows by the relation $|\nabla e^{\Delta t} f| = |e^{\Delta t} \nabla f| \leq e^{\Delta t} |\nabla f| $, valid for smooth functions $f$, and \eqref{eq:heat.bounded}. Strong continuity in $Y^{p}_1$ follows by usual density arguments.
\end{proof}

\subsection{Semigroup generation and perturbation structure of $L_{\alpha}$}\label{sec:semigruoup}
Let $\Srad$ denote the space of Schwartz radial functions on $\R^{n+2}$ and introduce the operator
\begin{align} \label{eq:def.mathcalL0}
\tilde{\mathcal{L}}_0 :f &\mapsto \Delta f + \frac{\xi}{2} \cdot \nabla f +f,
\quad 
D(\tilde{\mathcal{L}}_0)=\Srad.
\end{align}
\begin{lemma} \label{lem:L_0}
Fix $\eta\in [1,\infty)$ and $\gamma \in [\eta,\infty)$.  
Then the operator $(\tilde{\mathcal{L}}_0,D(\tilde{\mathcal{L}}_0))$ is closable in $Y^{\eta,\gamma}$ and its closure $(\mathcal{L}_0,D(\mathcal{L}_0))$ generates a strongly continuous semigroup $S_0$ on $Y^{\eta,\gamma}$ that satisfies the bound:
\begin{align} \label{eq:growth.bound.Y}
    \| S_0 (\tau) f \|_{Y^{\eta,\gamma}} \lesssim e^{\omega_{\mathrm{ess}} \tau} \| f \|_{Y^{\eta,\gamma}},
    \quad
    \omega_{\mathrm{ess}} := 1-\frac{n}{2\eta};
\end{align}
A similar statement holds true replacing the space $Y^{\eta,\gamma}$ with $Y_\nabla^{\eta,\gamma}$.
\end{lemma}
\begin{proof}
Closability of $(\tilde{\mathcal{L}}_0,D(\tilde{\mathcal{L}}_0))$ in $Y^{\eta,\gamma}$ follows by the fact that it admits an adjoint defined on $\Srad$. Indeed, if $\eta>1$ then the adjoint is densely defined. On the other hand, if $\eta=1$ and $\{f_k\}_{k \in \NN}$ is a sequence such that $f_k \to 0$ and $\tilde{\mathcal{L}}_0 f_k \to g$ in $Y^1$, then $\int_0^\infty r^{n-1} g_{\mathrm{rad}}(r) \phi_{\mathrm{rad}}(r) dr=0$ for every $\phi \in \mathscr{S}_{rad}$, hence $g = 0$ by the fundamental lemma of calculus of variations.
Since the topology of $Y_\nabla^{\eta,\gamma}$ is stronger than that of $Y^{\eta,\gamma}$, $(\tilde{\mathcal{L}}_0,D(\tilde{\mathcal{L}}_0))$ is also closable in $Y_\nabla^{\eta,\gamma}$.
With a little abuse of notation, in the following we denote $(\mathcal{L}_0,D(\mathcal{L}_0))$ the closure of $(\tilde{\mathcal{L}}_0,D(\tilde{\mathcal{L}}_0))$ with respect to either $Y^{\eta,\gamma}$ or $Y_\nabla^{\eta,\gamma}$, depending upon the context. 

Let us now denote $G_\tau$ the Gaussian kernel on $\R^{n+2}$, rescaled as in \cite[Proposition 2.2]{GlHoLaLu25}:
\begin{align}
    G_\tau(\xi) := (4\pi \alpha(\tau))^{-\frac{n+2}{2}}e^{-\frac{|\xi|^{2}}{4\alpha(\tau)}},
    \quad
    \xi \in \R^{n+2},
    \quad
    \alpha(\tau) := e^\tau-1.
\end{align}
Arguing as therein, we can explicitly compute the heat semigroup in self-similar variables and notice that $(\mathcal{L}_0,D(\mathcal{L}_0))$ generates the semigroup on $Y^{\eta,\gamma}$:
\begin{align} \label{eq:definition.semigroup.Y}
    (S_0(\tau) f)(\xi) := e^{\tau} (G_\tau \ast f) (e^{\tau/2} \xi), \quad
    \xi \in \R^{n+2}.
\end{align}
Strong continuity of $S_0$ in $Y^{\eta,\gamma}$ then descends by \eqref{eq:heat.bounded} and continuity of dilations in $Y^{\eta,\gamma}$.
With equation \eqref{eq:definition.semigroup.Y} in hand, in order to conclude the strong continuity of $S_0$ in $Y_\nabla^{\eta,\gamma}$ we only have to show strong continuity of $\xi \cdot \nabla (S_0 (\tau) f)$ with respect to the $Y^{\eta,\gamma}$ topology. Hence, fix any $f \in Y_\nabla^{\eta,\gamma}$; by change of variables $\zeta := e^{\tau/2}\xi$ we have
\begin{align}
    \|\xi \cdot \nabla (S_0 (\tau) f-f)\|_{Y^{\eta,\gamma}}
    &=
    \|\xi \cdot \nabla ( e^\tau( G_{\tau} \ast f)(e^{\tau/2} \cdot) - f)\|_{Y^{\eta,\gamma}}
    \\
    &\leq
    e^{\tau}\| e^{\tau/2}\xi \cdot \nabla ( G_{\tau} \ast f)(e^{\tau/2} \cdot) - e^{\tau/2}\xi \cdot \nabla f (e^{\tau/2} \cdot)\|_{Y^{\eta,\gamma}}
    \\
    &\quad+
    \| e^{3\tau/2} \xi \cdot \nabla f (e^{\tau/2} \cdot) - \xi \cdot \nabla f \|_{Y^{\eta,\gamma}}
    \\
    &\leq
    e^{\tau\left( 1-\frac{n}{2\gamma}\right)}
    \| \zeta \cdot \nabla (G_{\tau} \ast f)  -G_\tau \ast ( \zeta \cdot \nabla f) \|_{Y^{\eta,\gamma}}
    \\
    &\quad+
    e^{\tau\left( 1-\frac{n}{2\gamma}\right)}
    \| G_\tau \ast ( \zeta \cdot \nabla f) - \zeta \cdot \nabla f \|_{Y^{\eta,\gamma}}
    \\
    &\quad+
    \| e^{3\tau/2} \xi \cdot \nabla f (e^{\tau/2} \cdot) - \xi \cdot \nabla f \|_{Y^{\eta,\gamma}}.
\end{align}
Notice that the second term at the right-hand side goes to zero as $\tau \to 0$ by \eqref{eq:heat.bounded}, whereas the third term is infinitesimal by continuity of dilations in $Y^{\eta,\gamma}$, the fact that $f \in Y_\nabla^{\eta,\gamma}$, and dominated convergence.
As for the first term, observe that for every $\zeta \in \R^{n+2}$ it holds
\begin{align} \label{eq:commutator_1}
\zeta \cdot \nabla  &(G_\tau \ast f) (\zeta) -(G_\tau \ast ( \zeta \cdot \nabla f)) (\zeta)
\\
&= \int_{\R^{n+2}} (\zeta-\xi) \cdot \nabla G_\tau(\zeta-\xi) f(\xi) d\xi
+
(n+2) \int_{\R^{n+2}} G_\tau(\zeta-\xi) f(\xi) d\xi 
\\
&=- 2\alpha(\tau) \Delta(G_\tau \ast f)(\zeta).
\end{align}
Moreover, a change of variables shows that \eqref{eq:commutator_1} is infinitesimal as $\tau \to 0$, uniformly in $\zeta \in \R^{n+2}$, for every smooth $f$. We conclude by density of smooth functions in $Y^{\eta,\gamma}$, the bound \eqref{eq:heat.derivative.bounded}, and dominated convergence.

In order to conclude the proof of the lemma, it only remains to prove the bound \eqref{eq:growth.bound.Y}. 
First of all, we have by \eqref{eq:heat.bounded} and change of variables
\begin{align}
 \| S_0 (\tau) f \|_{Y^{\eta}} 
\lesssim 
e^{\omega_{\mathrm{ess}} \tau} \| G_\tau \ast f \|_{Y^{\eta}}
\lesssim 
e^{\omega_{\mathrm{ess}} \tau} \| f \|_{Y^{\eta}}.   
\end{align}
For the estimate in $Y^\gamma$, by \eqref{eq:heat.bounded} we can suppose $\tau>1$ without loss of generality. Using \eqref{eq:heat.ultracontractive} we get
\begin{align} \label{eq:bound.S0gamma}
 \| S_0 (\tau) f \|_{Y^{\gamma}} 
\lesssim 
e^{\tau \left( 1-\frac{n}{2\gamma} \right)} \| G_\tau \ast f \|_{Y^{\gamma}} 
\lesssim
e^{\tau \left( 1-\frac{n}{2\gamma} \right)} \alpha(\tau)^{-\frac{n}{2}\left( \frac{1}{\eta}-\frac{1}{\gamma} \right)} \| f \|_{Y^{\eta}}
\lesssim 
e^{\omega_{\mathrm{ess}} \tau} \| f \|_{Y^{\eta}}.
\end{align}
This completes the proof of the lemma in the space $Y^{\eta,\gamma}$.
As for the statement in $Y_\nabla^{\eta,\gamma}$, we additionally have to bound $\xi \cdot \nabla (S_0 (\tau) f)$ in $Y^{\eta,\gamma}$.
Let us discuss the bound in $Y^{\eta}$ first: by change of variable and formula \eqref{eq:commutator_1} we obtain
\begin{align}
\| \xi \cdot \nabla (S_0 (\tau) f) \|_{Y^{\eta}}
&=
e^{\omega_{\mathrm{ess}} \tau} \| \zeta \cdot \nabla (G_\tau \ast f) \|_{Y^\eta}
\\
&\lesssim
e^{\omega_{\mathrm{ess}} \tau} \| G_\tau \ast (\zeta \cdot \nabla f) \|_{Y^\eta}
+
e^{\omega_{\mathrm{ess}} \tau} \alpha(\tau) \| \Delta (G_\tau \ast f) \|_{Y^\eta}
\lesssim
e^{\omega_{\mathrm{ess}} \tau}\| f \|_{Y^{\eta}_\nabla},
\end{align}
where the last inequality follows by \eqref{eq:heat.bounded} and \eqref{eq:heat.derivative.bounded}.
Similar computations permit to give a bound on $\| \xi \cdot \nabla (S_0 (\tau) f) \|_{Y^{\gamma}}$ for $\tau \leq 1$. On the other hand, for $\tau>1$ we can invoke ultracontractivity estimates similarly to what we have done to obtain \eqref{eq:bound.S0gamma}.  
\end{proof}

By the explicit expression \eqref{eq:definition.semigroup.Y} and \autoref{lem:heat_semigroup_weighted}, it is easy to show the following smoothing and ultracontractivity properties for $S_0$. 
\begin{lemma}\label{properties_freesemigrup}
    Let $1\leq \eta\leq \gamma<\infty$. Then for every $f\in Y^{\eta,\gamma}$ and $\tau\in (0,1]$ it holds
    \begin{align} \label{reg_order_0_S0}
        \||\nabla|^kS_0(\tau)f\|_{Y^{\eta,\gamma}} 
        \lesssim
        \tau^{-k/2}
        \| f \|_{Y^{\eta,\gamma}},
        \quad
        k =1,2,3. 
    \end{align}
    Moreover, if $\eta\leq \eta',\ \gamma\leq \gamma'$ are such that $\frac{1}{\eta}-\frac{1}{\eta'}=\frac{1}{\gamma}-\frac{1}{\gamma'}$ then we also have 
    \begin{align} \label{ultracontractivity_free}
        \|S_0(\tau)f\|_{Y^{\eta',\gamma'}}
        \lesssim
        \tau^{-\frac{n}{2}\left(\frac{1}{\eta}-\frac{1}{\eta'}\right)}
        \| f \|_{Y^{\eta,\gamma}}, \quad \forall\tau\in(0,1].
    \end{align}
\end{lemma}

\autoref{properties_freesemigrup} above implies that $D(\mathcal{L}_0)$ with its graph norm is embedded in ${Y_1^{\eta,\gamma}}.$
\begin{lemma}
In the same setting of \autoref{lem:L_0}, for every $f \in D(\mathcal{L}_0) \subset Y^{\eta,\gamma}$ it holds
\begin{align} \label{eq:bound.coerciveY1}
    \| f \|_{Y_1^{\eta,\gamma}} \lesssim \| f \|_{Y^{\eta,\gamma}} + \| \mathcal{L}_0 f \|_{Y^{\eta,\gamma}}.
\end{align}
\end{lemma}
\begin{proof}
Let $\omega_{\mathrm{ess}}$ be given by \eqref{eq:growth.bound.Y} and choose $\lambda > \omega_{\mathrm{ess}}$. Then the resolvent
\begin{align}
    R(\lambda,\mathcal{L}_0) g = \int_0^\infty e^{-\lambda \tau} S_0(\tau) g d\tau
\end{align}
is well-defined in $Y^{\eta,\gamma}$ and satisfies $\| R(\lambda,\mathcal{L}_0) g \|_{Y^{\eta,\gamma}_1} \lesssim \| g \|_{Y^{\eta,\gamma}}$.
We conclude by choosing $g := (\mathcal{L}_0-\lambda)f$.
\end{proof}

We now move to introduce perturbations of the operator $(\mathcal{L}_0,D(\mathcal{L}_0))$.

Let $u_\alpha$ be the unique solution to \eqref{eq:ODE} corresponding to a fixed $\alpha>0$, let $\Omega_\alpha(y) := u_\alpha(|y|)$ for $y \in \R^{n+2}$, and define the operator
\begin{align} \label{eq:def.mathcalKalpha}
    \tilde{\mathcal{K}}_\alpha :f 
    &\mapsto 
    4n \Omega_\alpha f+ 2 \xi \cdot \nabla ( \Omega_\alpha f),
\quad
D(\tilde{\mathcal{K}}_\alpha)=\Srad.
\end{align}

\begin{lemma} \label{lem:K_alpha}
Fix $\eta\in [1,\infty)$ and $\gamma \in [\eta,\infty)$.  For every $\alpha>0$, the operator $(\tilde{\mathcal{K}}_\alpha,D(\tilde{\mathcal{K}}_\alpha))$ is closable in $Y^{\eta,\gamma}$ and its closure $(\mathcal{K}_\alpha,D(\mathcal{K}_\alpha))$ is relatively compact with respect to $(\mathcal{L}_0,D(\mathcal{L}_0))$.
  A similar statement holds true replacing the space $Y^{\eta,\gamma}$ with $Y_\nabla^{\eta,\gamma}$.
\end{lemma}
\begin{proof}
Closability of $(\tilde{\mathcal{K}}_\alpha,D(\tilde{\mathcal{K}}_\alpha))$ follows as in the proof of \autoref{lem:L_0}, by recalling the properties of the profile $u_\alpha$ entailed by \autoref{prop:Omega_alpha>0}.
Let us show that $(\mathcal{K}_\alpha,D(\mathcal{K}_\alpha))$ is relatively compact with respect to $(\mathcal{L}_0,D(\mathcal{L}_0))$.

For the statement in the $Y^{\eta,\gamma}$ topology, let $\{f_k\}_{k \in \NN}\subset D(\mathcal L_0)\subset Y^{\eta,\gamma}$ be a bounded sequence with $\|\mathcal{L}_0 f_k\|_{Y^{\eta,\gamma}} \leq 1$, then we want to prove $\mathcal{K}_\alpha f_k \to g$ for some $g \in Y^{\eta,\gamma}$ with respect to the $Y^{\eta,\gamma}$ topology.

Denote 
\begin{align} \label{eq:definition.hk}
    h_k := \mathcal{L}_0 f_k - \frac{\xi}{2} \cdot \nabla f_k - f_k;
\end{align}
If $r:=|\xi|$ and $f_{k,\mathrm{rad}}$ (resp. $h_{k,\mathrm{rad}}$) denotes the radial part of $f_k$ (resp. $h_k$), then we have
\begin{align} \label{eq:definition.hk.rad}
    h_{k,\mathrm{rad}} = f_{k,\mathrm{rad}}'' + \frac{n+1}{r} f_{k,\mathrm{rad}}' = \frac{(r^{n+1}  f_{k,\mathrm{rad}}')'}{r^{n+1}}. 
\end{align}
By the expression above we recover, after simple algebraic manipulations, 
\begin{align}
    \frac{f_{k,\mathrm{rad}}'(r)}{r} = \int_0^1 h_{k,\mathrm{rad}}(rt) t^{n+1} dt,
\end{align}
and by Minkowski inequality we get, for every fixed $R>0$, the uniform bound in $k \in \NN$
\begin{align}
    \left\| \frac{f_{k,\mathrm{rad}}'}{r}\right\|_{L^\gamma((0,R),r^{n-1}dr)}
    \lesssim
    \| h_{k,\mathrm{rad}} \|_{L^\gamma((0,R),r^{n-1}dr)} \lesssim 1,
\end{align}
where the second inequality descends by \eqref{eq:definition.hk}, the uniform bounds on $\| f_k \|_{Y^\gamma}, \| \mathcal{L}_0f_k \|_{Y^\gamma}$, and \eqref{eq:bound.coerciveY1}. By \eqref{eq:definition.hk.rad} we deduce as well $\| f_{k,\mathrm{rad}}'' \|_{L^\gamma((0,R),r^{n-1}dr)} \lesssim 1$ uniformly in $k$.

Applying the isomorphism $\iota$ and arguing similarly to \cite[Proposition 2.4]{GlHoLaLu25}, we deduce by Rellich-Kondrachov Theorem that $\iota f_k$ converges in $ W^{1,\gamma}_{loc}(\R^n,d\theta)$. Since $\eta \leq \gamma$, the convergence also takes place in $ W^{1,\eta}_{loc}(\R^n,d\theta)$.
Using the uniform bounds on $\|f_k\|_{Y^{\eta,\gamma}},\|f_k\|_{Y_1^{\eta,\gamma}}$ and the decay properties of $u_\alpha, u_\alpha'$ at infinity, we obtain the existence of $\tilde g =: \iota g \in X^{\eta,\gamma}$ such that
\begin{align}
    \| \mathcal{K}_\alpha f_k-g\|_{Y^{\eta,\gamma}}
    =
    \| \iota \mathcal{K}_\alpha f_k-\iota g\|_{X^{\eta,\gamma}}
    \to 0.
\end{align}

Let us move to the proof of the relative compactness in the $Y_\nabla^{\eta,\gamma}$ topology. 
Assume now $\{f_k\}_{k \in \NN}\subset D(\mathcal L_0)\subset Y_\nabla^{\eta,\gamma}$ is bounded with $\|\mathcal{L}_0f_k\|_{Y_\nabla^{\eta,\gamma}} \leq 1$.
By the previous step, we already have a candidate limit $g \in Y^{\eta,\gamma}$, and to conclude the proof of the lemma it only remains to show $\xi \cdot \nabla g \in Y^{\eta,\gamma}$ and the convergence
$\| \xi \cdot \nabla (\mathcal{K}_\alpha f_k-g)\|_{Y^{\eta,\gamma}} \to 0$, or equivalently $\| \theta \cdot \nabla (\iota \mathcal{K}_\alpha f_k-\iota g)\|_{X^{\eta,\gamma}}
    \to 0$.

To do this, observe that
\begin{align} \label{eq:commutator}
    \mathcal{L}_0(\xi \cdot \nabla f_k) = \xi \cdot \nabla (\mathcal{L}_0 f_k) +2\Delta f_k. 
\end{align}
The first term on the right-hand side is bounded in $Y^{\eta,\gamma}$ by the assumption $\|\mathcal{L}_0f_k\|_{Y_\nabla^{\eta,\gamma}} \leq 1$. On the other hand, by \eqref{eq:definition.hk} we also have $\|\Delta f_k\|_{Y^{\eta,\gamma}} \lesssim \| \mathcal{L}_0f_k \|_{Y^{\eta,\gamma}} + \| f_k \|_{Y_\nabla^{\eta,\gamma}} \lesssim 1$ and therefore $\ell_k := \xi \cdot \nabla f_k$ satisfies the uniform bound in $k \in \NN$:
\begin{align}
    \| \ell_k \|_{Y^{\eta,\gamma}} + \| \mathcal{L}_0 \ell_k \|_{Y^{\eta,\gamma}} \lesssim 1.
\end{align}
By the same arguments as above, $\iota \ell_k$ converges in $ W^{1,\eta}_{loc} \cap W^{1,\gamma}_{loc}(\R^n,d\theta)$. Hence, the following terms
\begin{align} \label{eq:nonremaining.term}
    \theta \cdot \nabla (\iota \Omega_\alpha \iota f_k) , \quad
    \iota \Omega_\alpha \theta \cdot \nabla (\theta \cdot \nabla (\iota f_k)), \quad
    (\theta \cdot \nabla (\iota \Omega_\alpha)) (\theta \cdot \nabla (\iota f_k)),
\end{align}
appearing in the expression of $\theta \cdot \nabla (\iota \mathcal{K}_\alpha f_k)$
converge in $X^{\eta,\gamma}$.
For the remaining term
\begin{align} \label{eq:remaining.term}
    \iota f_k \theta \cdot \nabla (\theta \cdot \nabla (\iota \Omega_\alpha))
\end{align}
we need to separately control the decay of $\theta \cdot \nabla (\theta \cdot \nabla (\iota \Omega_\alpha))$ when $|\theta| \to \infty$, for which we use the following argument.
First of all, observe that using the ODE satisfied by $u_\alpha$ we have
\begin{align}
  \rho  (\rho u_\alpha')'
  =
 - n\rho u_\alpha' - 2n\rho^2  u_\alpha^2 - 2\rho^3 u_\alpha u_\alpha'
 - \rho^2 u_\alpha-\frac{\rho^3}{2}u_\alpha'.
\end{align}
The first three terms at the right-hand side above decay as $\rho^{-2}$ when $\rho \to \infty$ by \autoref{prop:Omega_alpha>0} and \autoref{lem:bound.rho.m.Omega};
to verify decay of the remainder $\rho^2 u_\alpha +\frac{\rho^3}{2}u_\alpha'$ we argue as in \autoref{ssec:tail}.
Let $z(t) := e^{2t} u_\alpha(e^t)$; then by \autoref{prop:Omega_alpha>0}, \autoref{lem:bound.rho.m.Omega}, and using \eqref{eq:ODE}, we know that $z(t)$ and $\dot{z}(t)=2e^{2t} u_\alpha(e^t)+e^{3t} u_\alpha'(e^t)$ are bounded uniformly in $t \in \R$; in addition, $z(t)$ solves the (non-autonomous) ODE
\begin{align}
\ddot{z} + (n-4+2z) \dot{z} -  2(n-2)(z-z^2) +\frac{e^{2t}}{2} \dot{z}=0.
\end{align}
Using the analog of \eqref{eq:dotz0} with $h(t) := e^{2t}/2$, $g(t):=(n-4+2z(t))\dot{z}(t)-2(n-2)(z(t)-z^2(t))$, and $H(t) := \exp \left( \frac{e^{2t}-1}{4}\right)$, and recalling that $|g(t)|$ is bounded uniformly in $t \in \R$, by L'H\^opital's rule we have the bound
\begin{align}
    |\dot{z}(t)| \lesssim \frac{1+\int_0^t \exp (\frac{e^{2s}-1}{4})ds}{\exp (\frac{e^{2t}-1}{4})} \to 0,
    \quad
    \mbox{as } t \to \infty.
\end{align}
In particular this implies that $\theta \cdot \nabla (\theta \cdot \nabla (\iota \Omega_\alpha))$ is infinitesimal when $|\theta| \to \infty$ and therefore the remaining term \eqref{eq:remaining.term} does, in fact, converge in $X^{\eta,\gamma} $ when $k \to \infty$.

Putting all together, $\theta \cdot \nabla (\iota \mathcal{K}_\alpha f_k)$ converges in $X^{\eta,\gamma}$ when $k \to \infty$. It only remains to show that the limit is $\theta \cdot \nabla (\iota g)$, but this is implied by the fact that $\iota \mathcal{K}_\alpha f_k \to \iota g$ strongly in $X^{\eta,\gamma}$ thanks to the previous steps of the proof.
\end{proof}
\begin{rmk}\label{rmk_derivative}
    As a byproduct of the proof above and \autoref{prop:Omega_alpha>0} we also obtain
    \begin{align*}
      \lim_{\rho\rightarrow +\infty}  \rho^3 u_\alpha' =-2 \lim_{\rho\rightarrow +\infty}  \rho^2 u_\alpha=-2\ell_{\alpha}.
    \end{align*}
\end{rmk}

Putting \autoref{lem:L_0} and \autoref{lem:K_alpha} together, we deduce the following:
\begin{prop} \label{cor:S_alpha}
    The operator $L_\alpha=\mathcal{L}_0+\mathcal{K}_\alpha$ with domain $D(L_\alpha)=D(\mathcal{L}_0)$ generates a strongly continuous semigroup $S_\alpha$ on $Y^{\eta,\gamma}$ which is a compact perturbation of $S_0$. A similar statement holds true replacing the space $Y^{\eta,\gamma}$ with $Y_\nabla^{\eta,\gamma}$.
\end{prop}
\begin{proof}
Generation of a strongly continuous semigroup on $Y^{\eta,\gamma}$ descends from the Miyadera–Voigt perturbation Theorem \cite[Corollary III.3.16]{EnNa00} as soon as we prove that $\mathcal{K}_\alpha : D(\mathcal{L}_0) \to Y^{\eta,\gamma}$ is bounded and there exist $t_0>0$ and $q \in (0,1)$ such that
\begin{align}
    \int_0^{t_0} \| \mathcal{K}_\alpha S_0(\tau) f \|_{Y^{\eta,\gamma}} d\tau \leq q\|f\|_{Y^{\eta,\gamma}},
    \quad
    \forall f \in D(\mathcal{L}_0).
\end{align}
To check this we can use \eqref{eq:bound.coerciveY1}, the decay properties of $u_\alpha,u'_\alpha$, and the bounds \eqref{eq:growth.bound.Y} and \eqref{reg_order_0_S0}. 

Concerning the proof of the statement in the $Y^{\eta,\gamma}_{\nabla}$ topology, we adopt the same strategy by replacing $Y^{\eta,\gamma}$ with $Y^{\eta,\gamma}_{\nabla}$ in the above: here the bound needed to apply the Miyadera–Voigt perturbation Theorem comes from decomposing $\xi \cdot \nabla (\mathcal{K}_\alpha S_0(\tau)f)$ as in \eqref{eq:nonremaining.term}-\eqref{eq:remaining.term} and invoking the decay properties of $u_\alpha,u'_\alpha$. The only nontrivial quantity we need to control is
\begin{align}
    \| \Omega_\alpha \xi \cdot \nabla ( \xi \cdot \nabla S_0(\tau)f) \|_{Y^{\eta,\gamma}} 
    \lesssim
    \|  \xi \cdot \nabla (G_\tau \ast f) \|_{Y_1^{\eta,\gamma}}.
\end{align}
For this we recall the commutator property \eqref{eq:commutator_1}, the smoothing estimate \eqref{reg_order_0_S0}: 
\begin{align}
  \|  \xi \cdot \nabla (G_\tau \ast f) \|_{Y_1^{\eta,\gamma}}
  &\lesssim
  \|  S_0(\tau)( \xi \cdot \nabla f) \|_{Y_1^{\eta,\gamma}}
  +
  \alpha(\tau) \| \Delta S_0 (\tau)f\|_{Y_1^{\eta,\gamma}}
  \lesssim
  \tau^{-1/2} \|f\|_{Y_\nabla^{\eta,\gamma}} ,
  \quad
  \forall \tau \in (0,1).
\end{align}

Let us show that $S_\alpha$ is a compact perturbation of $S_0$. Fix $\tau>0$; using the mild formulation
    \begin{align}
        S_0(\tau)f -S_\alpha(\tau) f = -\int_0^\tau S_\alpha(\tau-s) \mathcal{K}_\alpha S_0(s)f ds
    \end{align}
and \cite[Theorem C.7]{EnNa00}, it suffices to show that $S_\alpha(\tau-s) \mathcal{K}_\alpha S_0(s)$ is compact for every $s \in (0,\tau)$. 

Since $S_\alpha$ is a strongly continuous semigroup, it is also bounded, and thus it is sufficient that for $\|f_k\| \leq 1$ (with respect to either $Y^{\eta,\gamma}$ or $Y^{\eta,\gamma}_\nabla$ norm) then $\mathcal{K}_\alpha S_0(s) f_k$ converges in the relevant topology. 
For every $s>0$ it holds $S_0(s)f_k \in D(\mathcal{L}_0)$ with $\|\mathcal{L}_0 S_0(s)f_k\| \lesssim_s \|f_k\| \lesssim_s 1$ (by the explicit formula \eqref{eq:definition.semigroup.Y} for the semigroup) and therefore $\mathcal{K}_\alpha S_0(s)f_k$ converges by the relative compactness stated in \autoref{lem:K_alpha}.
\end{proof}

\begin{rmk}
The operator $L_\alpha$ corresponds to the linearization of the right-hand side of \eqref{eq:Omega} around the $\tau$-independent profile $\Omega_\alpha(y)=u_\alpha(|y|)$: indeed, for $f \in D(L_\alpha)$, we have the explicit expression
\begin{align}\label{def_lalpha}
    L_\alpha f = \Delta f + \frac{\xi}{2} \cdot \nabla f +f+ 4n \Omega_\alpha f+ 2 \xi \cdot \nabla ( \Omega_\alpha f),
\end{align}
or equivalently, interpreting $f$ as a radial profile:
\begin{align} \label{eq:L_alpha.radial}
  L_\alpha f=   f'' + \left( \frac{n+1}{\rho}+ \frac{\rho}{2} \right) f' + f 
+ 4n u_\alpha f+ 2 \rho u_\alpha' f + 2 \rho u_\alpha f'.
\end{align}
\end{rmk}

\subsection{Existence of unstable eigenvalues}\label{subsec_unstable_eigen}

This subsection is devoted to the proof of existence of unstable eigenvalues for the linearized operator $L_\alpha$ given by previous \autoref{cor:S_alpha}. The results here presented are valid for both spaces $Y^{\eta,\gamma}$ and $Y_\nabla^{\eta,\gamma}$ on which the operator $L_\alpha$ can be defined according to previous \autoref{lem:L_0} and \autoref{lem:K_alpha}.  

We are in a similar setting as \cite{JiSv15}.
By \autoref{lem:L_0}, \autoref{cor:S_alpha}, and \cite[Corollary 2.11]{EnNa00}, for every $\omega>\omega_{\mathrm{ess}}$ the set
\begin{align} \label{eq:spectrum.finite}
 \sigma_{\mathrm{Re}\geq \omega}(L_\alpha) :=   \sigma(L_\alpha) \cap \left\{ \lambda \in \CC \,:\, \mathrm{Re}(\lambda) \geq \omega \right\}
\end{align}
has finite cardinality, and in particular the resolvent set satisfies
\begin{align} \label{eq:resolvent.non.empty}
    \rho(L_\alpha) \cap \left\{ \lambda \in \CC : \mathrm{Re}(\lambda) \geq \omega \right\}\neq \varnothing.
\end{align}

Since by \cite[Chapter IV, Theorem 5.26]{Kato} and \cite[Proposition 2.2]{EnNa00} we have
\begin{align}
  \sigma_{\mathrm{ess}}(L_\alpha) := \{ \lambda \in \CC \,:\, L_\alpha-\lambda \mbox{ is not Fredholm} \} = \sigma_{\mathrm{ess}}(\mathcal{L}_0) \subset \left\{ \lambda \in \CC \,:\, \mathrm{Re}(\lambda) \leq \omega_{\mathrm{ess}} \right\}  
\end{align}
and the Fredholm index of $L_\alpha-\lambda$ is constant when $\lambda$ varies in a connected component of the complementary of $\sigma_{\mathrm{ess}}(L_\alpha)$, from \eqref{eq:resolvent.non.empty} we deduce that: 
\begin{align} \label{eq:spectrum.made.of.eigen}
   \forall \omega>\omega_{\mathrm{ess}}, \quad \mbox{every } \lambda \in \sigma_{\mathrm{Re}\geq \omega} (L_\alpha) \mbox{ is an eigenvalue of }L_\alpha.
\end{align}

Moreover, again \cite[Corollary 2.11]{EnNa00} informs us that the growth bound $\omega_0$ of the semigroup $S_\alpha$ is equal to the maximum between the essential growth bound of the semigroup $S_\alpha$ (which is less than or equal to $\omega_{\mathrm{ess}}$ since $S_\alpha$ is a relatively compact perturbation of $S_0$) and the spectral bound of the operator $L_\alpha$ (defined as $\sup \{\mathrm{Re}(\lambda): \lambda \in \sigma(L_\alpha)\}$).
Therefore, we deduce that for every fixed $\omega>\omega_{\mathrm{ess}}$ the growth bound of $S_\alpha$ can be controlled by
\begin{align} \label{eq:growth.bound}
    \omega_0 \leq \omega \vee \max \left\{ \mathrm{Re}(\lambda) \,:\,
    \lambda \in \sigma_{\mathrm{Re}\geq \omega}(L_\alpha) \mbox{ is an eigenvalue of } L_\alpha  \right\}.
\end{align}

Given the above discussion, we turn our focus to the study of eigenvalues of $L_\alpha$. We have the following:
\begin{lemma} \label{lem:exponential.decay.eigenfunctions}
Fix $\eta\in [1,n/2)$, $\gamma \in [\eta,\infty)$, and $\omega \geq 0>\omega_{\mathrm{ess}}$, and suppose $\lambda \in \sigma_{\mathrm{Re}\geq \omega}(L_\alpha)$ is an eigenvalue of $L_\alpha$ with eigenfunction $\bar f \in Y^{\eta,\gamma}$. Then $\bar f \in L^2_\pi(\R^{n+2})$ where $L^2_\pi(\R^{n+2})$ has been defined in the proof of \autoref{lem:principal_eigen}.
\end{lemma}
\begin{proof}
Let $\bar f \in Y^{\eta,\gamma}$ be an eigenfunction relative to $\lambda$ and denote $f$ its radial part, i.e. $f(r) := \bar f(r e_1)$ for $e_1\in \R^{n+2}$ of unit norm; then  
\begin{align} \label{eq:eigenvalue.equation.f}
  f'' + \left( \frac{n+1}{\rho}+ \frac{\rho}{2} \right) f' + f 
+ 4n u_\alpha f+ 2 \rho u_\alpha' f + 2 \rho u_\alpha f'
=
\lambda f
\end{align}
where the equality is meant in $\L(\R_+,\rho^{n-1}d\rho)$.
Since $f' = \int f''$ in the sense of distributions, \eqref{eq:eigenvalue.equation.f} implies that $f' \in L_{loc}^{\eta,\gamma}((0,\infty),\CC)$. In particular there exists $\rho_0 >0$ such that $f(\rho_0), f'(\rho_0) \in \CC$, and thus the various integral formulations of the ODE used so far give $f \in C^\infty((0,\infty),\CC)$. In particular, by local existence and uniqueness there are at most two $\CC$-linearly independent solutions of \eqref{eq:eigenvalue.equation.f}.

\emph{Step 1: behavior as $\rho \downarrow 0$}.
Next we analyze the behavior of $f$ close to $\rho =0$, applying the Frobenius method.  
Since $u_\alpha$ is real analytic in a neighborhood of $0$ by \autoref{prop:local.existence.uniqueness}, $0$ is a regular singular point of \eqref{eq:eigenvalue.equation.f}. 
The roots of the associated indicial polynomial are $0$ and $-n$, and therefore Fuchs Theorem gives two $\CC$-linearly independent solutions behaving for $\rho \downarrow 0$ like 
\begin{align}
    f_1(\rho) &\to 1, \quad f_1'(\rho) \to 0,
    \\
    f_2(\rho)\rho^{n} &\to 1, \quad f_2'(\rho)\rho^{n+1} \to -n.
\end{align}
Notice that $f_2 \notin \L(\R_+,\rho^{n-1}d\rho)$ because of its behavior close to $\rho=0$ and therefore $f$ is a complex multiple of $f_1$. 

\emph{Step 2: behavior as $\rho \uparrow \infty$}.
We apply the transformation already encountered in the proof of \autoref{prop:f.zero}:
\begin{align}
    {f}_\star (\rho) 
    := 
     {f}(\rho) \exp \left( \frac12 \int_{1}^\rho \left( \frac{n+1}{s}+\frac{s}{2} +2 s u_\alpha (s) \right)ds \right),
\end{align}
which solves the ODE
\begin{align}
   {f}_\star'' & = v f_\star ,
   \quad
   \\
   v(\rho) &:=
\frac{\rho^2}{16}
+ \frac{n-2}{4}+\frac{\rho^2 u_\alpha}{2}+\lambda
-
(3n-2)u_\alpha 
-
 \rho u_\alpha'
-
\frac{n+1}{2\rho^2} 
+
\frac{(n+1)^2}{4\rho^2} 
+
\rho^2u_\alpha^2 .
\end{align}
We intend to apply \cite[Theorem 3.1]{LeLeTrWi25} to determine the asymptotic behavior of $f_\star$ for $\rho \to \infty$. To do this, let $\rho_0>0$ be large enough so that $\mathrm{arg}(v(\rho)) \in (-\pi,\pi)$ for every $\rho \geq \rho_0$ and define the complex branch of the $n$-th root of $v(\rho)$ as $v^{1/n}(\rho) := |v(\rho)|^{1/n} e^{i \mathrm{arg}(v(\rho))/n}$, according to the convention in \cite{LeLeTrWi25}.

By \cite[Theorem 3.1]{LeLeTrWi25}, equation \eqref{eq:eigenvalue.equation.f} has two $\CC$-linearly independent solutions
\begin{align}
    f_{\star,1}(\rho) &= \frac{1}{v^{1/4}(\rho)} \exp\left( \int_{\rho_0}^\rho v^{1/2}(s)ds\right) (1+\epsilon_1(\rho)),
    \\
    f_{\star,2}(\rho) &= \frac{1}{v^{1/4}(\rho)} \exp\left( -\int_{\rho_0}^\rho v^{1/2}(s)ds\right) (1+\epsilon_2(\rho)),
\end{align}
where the errors $\epsilon_1$, $\epsilon_2$ are controlled by
\begin{align}
    |\epsilon_1(\rho)| \leq 2\exp \left( \int_{\rho_0}^\rho \left|
    \frac{1}{v^{1/4}(s)} \frac{d^2}{ds^2}\left( \frac{1}{v^{1/4}(s)}\right)
    \right|ds\right)-2,
    \\
    |\epsilon_2(\rho)| \leq 2 \exp \left( \int_\rho^\infty \left|
    \frac{1}{v^{1/4}(s)} \frac{d^2}{ds^2}\left( \frac{1}{v^{1/4}(s)}\right)
    \right|ds\right)-2.
\end{align}
Up to taking a larger value of $\rho_0$, by the decaying properties of $u_\alpha$ we can additionally suppose
\begin{align}
    \int_{\rho_0}^\infty \left|
    \frac{1}{v^{1/4}(s)} \frac{d^2}{ds^2}\left( \frac{1}{v^{1/4}(s)}\right)
    \right|ds
    \leq C\int_{\rho_0}^\infty \frac{ds}{s^3} \ll 1,
\end{align}
so that $\sup_{\rho \geq \rho_0}|\epsilon_j(\rho)| \leq 1/2$ and the asymptotic behaviors 
\begin{align}
    |f_{\star,1}(\rho)| &\sim \frac{1}{v^{1/4}(\rho)} \exp\left( \int_{\rho_0}^\rho v^{1/2}(s)ds\right),
    \\
    |f_{\star,2}(\rho)| &\sim \frac{1}{v^{1/4}(\rho)} \exp\left( -\int_{\rho_0}^\rho v^{1/2}(s)ds\right),
\end{align}
are exact as $\rho \to \infty$.
Moreover, by \autoref{prop:Omega_alpha>0} and Taylor expansion we have for every fixed $a>0$
\begin{align}
    \mathrm{Re}(v^{1/2}(\rho)) \geq \frac{\rho}{4} + \frac{1-a}{\rho}\left( \frac{n-2}{2}+\ell_\alpha+2\mathrm{Re}(\lambda)\right) 
\end{align}
for every $\rho$ sufficiently large, and therefore
\begin{align}
\left|
\frac{1}{v^{1/4}(\rho)}\exp \left( \int_{\rho_0}^\rho v^{1/2}(s)ds \right) 
\right|
&\gtrsim 
\rho^{(1-a)\left( \frac{n-2}{2}+\ell_\alpha+2\mathrm{Re}(\lambda)\right)-\frac12} e^\frac{\rho^2}{8},
\\
\left|\frac{1}{v^{1/4}(\rho)}\exp \left( -\int_{\rho_0}^\rho v^{1/2}(s)ds \right)\right| 
&\lesssim 
\rho^{-(1-a)\left( \frac{n-2}{2}+\ell_\alpha+2\mathrm{Re}(\lambda)\right)-\frac12} e^{-\frac{\rho^2}{8}}.
\end{align}
In particular
\begin{align}
|f_{\star,1}(\rho)| \exp \left(- \frac12 \int_{1}^\rho \left( \frac{n+1}{s}+\frac{s}{2} +2 s u_\alpha (s) \right)ds \right)
\gtrsim 
\rho^{-2+2\mathrm{Re}(\lambda) -a\left( \frac{n-2}{2}+2\ell_\alpha+2\mathrm{Re}(\lambda)\right)}
\end{align}
as $\rho \to \infty$, and it is not in $\L(\R_+,\rho^{n-1}d\rho)$ for $0<a\ll 1$ because $\mathrm{Re}(\lambda) \geq 0$ and $\eta<n/2$; we conclude that $f_\star$ must be a multiple of $f_{\star,2}$ and thus $f$ behaves for large $\rho$ as 
\begin{align} \label{eq:behavior.f.infty}
   | f(\rho)| \lesssim \rho^{-\frac{n+2}{2}-(1-a)\left( \frac{n-2}{2}+2\ell_\alpha+2\mathrm{Re}(\lambda)\right)} e^{-\frac{\rho^2}{4}}.
\end{align}
By \eqref{eq:behavior.f.infty} above one immediately gets $\bar{f} \in L^2_\pi(\R^{n+2})$ and the proof is complete.
\end{proof}

\begin{rmk}
The same argument given in Step 2 above, together with the explicit error bound on $f_{\star,2}'$ stated in \cite[Theorem 3.1]{LeLeTrWi25}, yields the control on $f'(\rho)$ for $\rho$ large enough:
\begin{align} \label{eq:behavior.f'.infty}
  |f'(\rho)| \lesssim |v^{1/2}(\rho)||f(\rho)| 
  \lesssim \rho^{-\frac{n}{2}-(1-a)\left( \frac{n-2}{2}+2\ell_\alpha+2\mathrm{Re}(\lambda)\right)} e^{-\frac{\rho^2}{4}}.
\end{align}
\end{rmk}

Let us introduce the space 
\begin{align}
    L^2_{\pi,\mathrm{rad}} := \left\{ f : [0,\infty) \to \CC \, : \, f(|\cdot|) \in L^2_\pi(\R^{n+2
    })\right\}.
\end{align}
A second (and useful) point of view is looking at $L_\alpha$ as a linear operator acting on radial profiles $f \in L^2_{\pi,\mathrm{rad}}$ via formula \eqref{eq:L_alpha.radial}.
Indeed, since $n \geq 3$, the proof of previous \autoref{lem:exponential.decay.eigenfunctions} immediately implies that the extrema $0,\infty$ are \emph{limit-points} for the formal Sturm-Liouville operator \eqref{eq:L_alpha.radial}, according to Weyl’s classification. Therefore, \eqref{eq:L_alpha.radial} has a self-adjoint realization on $L^2_{\pi,\mathrm{rad}}$ that we denote by $(L_\alpha^\pi,D(L_\alpha^\pi))$.

The previous \autoref{lem:exponential.decay.eigenfunctions} implies the following:
\begin{cor} \label{cor:eigenvalues}
Fix $\eta \in [1,n/2)$, $\gamma \in [\eta,\infty)$, and suppose $\lambda \in \CC$ with $\mathrm{Re}(\lambda) \geq 0$.
Then $\lambda$ is an eigenvalue of $L_\alpha$ if and only if it is an eigenvalue of $L_\alpha^\pi$.
In particular, the eigenvalues of $L_\alpha$ with non-negative real part are independent of the particular choice of $\eta \in [1,n/2)$ and $\gamma \in [\eta,\infty)$ defining the spaces $Y^{\eta,\gamma}$ and $Y_\nabla^{\eta,\gamma}$.
\end{cor}
\begin{proof}
By previous lemma, the radial profile $f$ of an eigenfunction $\bar f$ of $L_\alpha$ relative to the eigenvalue $\lambda$ satisfies $L_\alpha f = L^\pi_\alpha f =\lambda f\in L^2_{\pi,\mathrm{rad}}$, thus we have $f \in D(L^\pi_\alpha)$ and therefore $f$ is an eigenfunction of $L^\pi_\alpha$ with eigenvalue $\lambda$.
For the converse implication, it suffices to show that given an eigenfunction $f$ of $L_\alpha^\pi$ it holds $f(|\cdot|) \in Y_\nabla^{\eta,\gamma}$.
The same argument used in Step 1 of the proof of the lemma above ensures that $|f|$, $|f'|$ are bounded in a neighborhood of zero; therefore, for $g$ equal to either $f$ or $\rho f'$ and for every $q \in [1,\infty)$ we have
\begin{align}
\| g \|^q_{Y^q} 
&\lesssim 
\int_0^\infty \rho^{n-1} |g(\rho)|^q d\rho
\lesssim
\int_0^1 |g(\rho)|^q d\rho
+
\int_1^\infty \rho^{n-1}|g(\rho)|^qd\rho
\\
&\lesssim
1
+
\int_1^\infty \rho^{n-1}|g(\rho)|^{q-1}\pi^{-1/2}(\rho) |g(\rho)| \pi^{1/2}(\rho) d\rho
\\
&\lesssim
1
+
\left(\int_1^\infty \rho^{2n-2}|g(\rho)|^{2q-2}(\rho) \pi^{-1}(\rho) d\rho \right)^{1/2}
\left(\int_1^\infty |g(\rho)|^2 \pi(\rho) d\rho\right)^{1/2}
< \infty,
\end{align}
where the last expression is finite by the asymptotic behaviors \eqref{eq:behavior.f.infty} and \eqref{eq:behavior.f'.infty}.
\end{proof}

The following theorem states the existence of unstable eigenvalues for the linearized operator $L_\alpha$. It holds for both spaces $Y^{\eta,\gamma}$ and $Y_\nabla^{\eta,\gamma}$.

\begin{theorem} \label{thm:instability}
Fix $n \in \{3,\dots,9\}$ and $\eta\in [1,n/2)$, $\gamma \in [\eta,\infty)$. For every $\bar\lambda> 0$ there exists $\bar \alpha \in (0,\infty)$ such that $L_{\bar \alpha}$ has a real positive eigenvalue $\lambda \in (0,\bar \lambda]$, that is \emph{maximal} in the sense that $\sigma(L_{\bar \alpha}) \subset \{ \mu \in \CC \, : \, \mathrm{Re}(\mu) \leq \lambda\}$. 
Moreover, the eigenfunction corresponding to $\lambda$ decays at infinity exponentially fast.
\end{theorem}
\begin{proof}
For $\alpha>0$ every $\lambda \in \sigma_{\mathrm{Re}\geq 0}(L_\alpha)$ is an eigenvalue of $L_\alpha$ by \eqref{eq:spectrum.made.of.eigen}, irrespective of the choice of $Y^{\eta,\gamma}$ or $Y_\nabla^{\eta,\gamma}$ topology.
Moreover, since $L_\alpha^\pi$ is self-adjoint we know that $\sigma(L_\alpha^\pi) \subset \RR$.
Thus, by \autoref{cor:eigenvalues} it is sufficient to show that there exists $\bar \alpha$ such that $L_{\bar \alpha}^\pi$ has at least a real positive eigenvalue $\lambda \in (0,\bar\lambda]$ and no eigenvalue larger than $\lambda$.

Sturm-Liouville oscillation theory  \cite[Theorem 1.2 and Corollary 2.4]{GeSiTe96} and \autoref{prop:f.zero} applied with $k=1$ guarantee the existence of $\alpha_0 \in (0,\infty)$ such that $L_{\alpha_0}^\pi$ has at least a positive eigenvalue. By \autoref{cor:eigenvalues} and the discussion above, eigenvalues of $L_{\alpha_0}^\pi$ cannot accumulate to $+\infty$; therefore,  there exists the maximum eigenvalue $\lambda_{\max}^{\alpha_0}$ and it satisfies $0 < \lambda^{\alpha_0}_{\max} < \infty$.

We distinguish two cases: if $\lambda_{\max}^{\alpha_0} \leq \bar\lambda$ then $\bar \alpha := \alpha_0$ and there is nothing to prove. Otherwise, let us consider the operator $L_\alpha^\pi - \bar\lambda$;
For $\alpha=\alpha_0$ we have by assumption that $L_{\alpha_0}^\pi - \bar\lambda$ has at least a strictly positive eigenvalue.
By Sturm-Liouville theory we get the following: the unique solution of the system
\begin{align}  
    \begin{dcases}
    \mathcal{L}_0 f  - \bar\lambda f + \mathcal{K}_{\alpha_0} f = 0 ,
    \\
    f(0) = 1,
    \\
    f'(0) = 0,
\end{dcases}
\end{align}
admits at least a positive zero $\rho = \rho(\alpha_0,\bar\lambda) > 0$.
Our next goal is to find $\bar \alpha\in (0,\infty)$ such that $L_{\bar \alpha}^\pi - \bar\lambda$ has eigenvalue $0$ and no strictly positive eigenvalue.

\emph{Step 1}.
For $\alpha \in (0,\infty)$ let us consider the system 
\begin{align} \label{eq:system_f,alpha}
    \begin{dcases}
    \mathcal{L}_0 f  - \bar\lambda f + \mathcal{K}_{\alpha} f = 0 ,
    \\
    f(0) = 1,
    \\
    f'(0) = 0,
\end{dcases}
\end{align}
and set
\begin{align}
    A = A^{\bar\lambda} := \{ \alpha>0 : f \mbox{ solution of } \eqref{eq:system_f,alpha} \mbox{ admits a positive zero} \}
    \neq \varnothing.
\end{align}
The fact that $A \neq \varnothing$ is implied by the fact that $\alpha_0 \in A$. Next we want to show that $\inf A >0$.

Recall the definition of $\mathcal{V}_\alpha$ from \autoref{sec:ODE} and notice that $\mathcal{K}_\alpha f = 2 \Omega_\alpha\xi \cdot \nabla f + \mathcal{V}_\alpha f$, when we interpret $\mathcal{V}_\alpha$ as acting on radial functions. 
By \autoref{lem:principal_eigen} and \eqref{eq:lower.bound.barmu} in the proof thereof, we know that the system
\begin{align}
    \begin{dcases}
    \mathcal{L}_0 f  - \bar\lambda f + \mathcal{K}_\alpha f + (\mu-1) \mathcal{V}_{\alpha} f = 0 ,
    \\
    f(0) = 1,
    \\
    f'(0) = 0,
\end{dcases}
\end{align}
has principal eigenvalue 
\begin{align}
   \bar\mu =  \bar{\mu}^{\alpha}_{\bar\lambda} \geq \frac{1}{16n \alpha} > 1,
   \qquad
   \forall \alpha < \frac{1}{16n}.
\end{align}
In particular, by point $iii$) of \autoref{lem:principal_eigen} the solution of the system \eqref{eq:system_f,alpha} cannot have positive zeros if $\alpha < \frac{1}{16n}$, meaning $\inf A \geq \frac{1}{16n}>0$.

\emph{Step 2}. Let us denote $\bar \alpha := \inf A \in (0,\infty)$. We claim that the operator $L_{\bar \alpha}^\pi - \bar\lambda$ has eigenvalue $0$ and no strictly positive eigenvalue. 

Let us show first that no strictly positive eigenvalue exists. For if that were the case, then by Sturm-Liouville theory the unique solution $f_{\bar \lambda}$ to the system \eqref{eq:system_f,alpha} with $\alpha = \bar \alpha$ would have a positive zero, contradicting the minimality of $\bar \alpha$ (by continuity of \eqref{eq:system_f,alpha} with respect to $\alpha$, there would be another $\bar \alpha' \in (0,\bar \alpha)$ for which the unique solution to \eqref{eq:system_f,alpha} with $\alpha = \bar \alpha'$ would have a positive zero, thus $\bar\alpha' \in A$ against the assumption that $\bar \alpha = \inf A$).

In order to conclude, we must show that $0$ is an eigenvalue for the operator $L_{\bar \alpha}^\pi - \bar \lambda$, or equivalently that $f_{\bar \lambda}$ is, in fact, an eigenfunction.
By Sturm-Liouville theory and \autoref{lem:principal_eigen} it suffices to show that the principal eigenvalue of the system
\begin{align}  
    \begin{dcases}
    \mathcal{L}_0 f  - \bar \lambda f + \mathcal{K}_\alpha f + (\mu-1) \mathcal{V}_{\alpha} f  = 0 ,
    \\
    f(0) = 1,
    \\
    f'(0) = 0,
\end{dcases}
\end{align}
is exactly equal to $\bar \mu = 1$: Indeed, this would mean that $f_{\bar \lambda}$ satisfies \eqref{eq:decay_f} with suitable bounds near $\rho=0$, and is therefore in the domain of the operator $L_{\bar \alpha}^\pi-\bar \lambda$. 

Since $f_{\bar \lambda}$ has no positive zeros, we know by point $ii$) of \autoref{lem:principal_eigen} that $\bar\mu \geq 1$. 
Suppose by contradiction that $\bar{\mu} > 1$; then by point $iii$) of \autoref{lem:principal_eigen} and the asymptotic behavior when $\rho \to \infty$ proven in Step 2 in the proof of \autoref{lem:exponential.decay.eigenfunctions} for the solution of the ODE \eqref{eq:eigenvalue.equation.f}, it must be
\begin{align}
f_{\bar \lambda}(\rho_0) 
>0,
\quad
f'_{\bar \lambda}(\rho_0) + \frac{\rho_0}{2} f_{\bar \lambda}(\rho_0)>0,
\end{align}
for every $\rho_0 \gg 1$, where the second inequality comes from the explicit expression for $f_{\star,2}(\rho)$ and the error bounds $\sup_{\rho \geq \rho_0}|\epsilon_j(\rho)| \leq 1/4$ and $\sup_{\rho \geq \rho_0}v^{-1/2}(\rho)|\epsilon_j'(\rho)| \leq 1/2$ for $j \in \{1,2\}$. 

Now we can argue as in Step 4 of \autoref{prop:Omega_alpha>0}: Take $\rho_0 \gg 1$ large enough so that $2(n-2)u_\alpha(\rho) \leq 4(n-2) \rho^{-2}< n/2+\bar{\lambda}$ for every $\rho \geq \rho_0$, and take $A \ni \alpha>\bar \alpha$ sufficiently close to $\bar \alpha$ such that the unique solution $f$ of \eqref{eq:system_f,alpha} satisfies   
\begin{align}
f(\rho) 
>0,
\quad
\forall \rho \leq \rho_0,
\quad
\mbox{and}
\quad
f'(\rho_0) + \frac{\rho_0}{2} f(\rho_0)>0;
\end{align}
For $\rho > \rho_0$ consider the inequality:
\begin{align}
   \rho^{n+1} f'(\rho) &=  \rho_0^{n+1} \left( f'(\rho_0)+ \rho_0 \left( \frac12 +2 u_\alpha(\rho_0) \right) f(\rho_0)\right)
   -
   \rho^{n+2} \left( \frac12 +2 u_\alpha(\rho) \right) f(\rho) 
\\
\quad&-
\int_{\rho_0}^\rho s^{n+1} \left( -\bar{\lambda}-\frac{n}{2} + 2(n-2)u_\alpha(s) \right) f(s) ds
\\
&>-
   \rho^{n+2} \left( \frac12 +2  u_\alpha(\rho) \right) f(\rho).
 \end{align}
From this we deduce $f'(\rho)>0$ for the first $\rho > \rho_0$ in which $f(\rho)=0$, which is absurd.
Therefore $\bar{\mu}=1$ and the proof is complete.
\end{proof}

\begin{rmk}
Notice that the operator $A^{-1}L_{\bar\alpha} A$ corresponds to the linearization of the right-hand side of \eqref{eq:Xi} around the profile $\bar{\Xi}_\alpha := A^{-1}[\bar\Omega_\alpha]$. 
By \autoref{lem:isomorphism}, the map $A : X^{\eta,\gamma} \to Y^{\eta,\gamma}_{\nabla}$ is an isomorphism when $\eta>1$ and therefore \autoref{thm:instability} implies \autoref{thm:instability.intro} in the introduction.
\end{rmk}

We conclude this section with the following corollary containing semigroup and smoothing estimates for $S_{\bar\alpha}$ analog to \eqref{eq:growth.bound.Y}, \eqref{reg_order_0_S0}, and \eqref{ultracontractivity_free} given in \autoref{lem:L_0} and \autoref{properties_freesemigrup}.  
\begin{cor} \label{prop:smoothing}
    Under the same assumptions of \autoref{properties_freesemigrup} and \autoref{thm:instability}, for every fixed $\delta>0$ we have the bounds, valid for every $\tau>0$:
    \begin{align}
    \label{eq:bound.semigroup}
    \| S_{\bar \alpha}(\tau) f \|_{Y^{\eta,\gamma}} 
    &\lesssim 
    e^{(\lambda+\delta) \tau} \|f\|_{Y^{\eta,\gamma}},
    \\
    \label{smoothing}
    \|S_{\bar\alpha}(\tau)f\|_{Y^{\eta,\gamma}_1} 
    &\lesssim
    \frac{e^{(\lambda+\delta)\tau}}{\tau^{1/2}} 
    \|f\|_{Y^{\eta,\gamma}},
    \\
    \label{ultracontractivity}
    \|S_{\bar\alpha}(\tau)f\|_{Y^{\eta',\gamma'}} & \lesssim\frac{e^{(\lambda+\delta)\tau}}{\tau^{\frac{n}{2}\left(\frac{1}{\eta}-\frac{1}{\eta'}\right)}}
    \|f\|_{Y^{\eta,\gamma}}.
    \end{align}
\end{cor}

\begin{proof}
The growth bound \eqref{eq:bound.semigroup} follows by \autoref{thm:instability} by recalling \eqref{eq:growth.bound}.
Bound \eqref{smoothing} for small values of $\tau\leq \tau_*\leq 1$ follows by the corresponding bound \eqref{reg_order_0_S0} for $S_0$ and the mild formula
    \begin{align*}
S_{\bar\alpha}(\tau) f=S_0(\tau)f+\int_0^{\tau} S_0(\tau-s)
\mathcal{K}_{\bar\alpha} [S_{\bar\alpha}(s) f]ds,
    \end{align*}
while for large  $\tau\geq  \tau_* $ it holds
    \begin{align*}
\|S_{\bar\alpha}(\tau)f\|_{Y^{\eta,\gamma}_1}= \|S_{\bar\alpha}(\tau_*)S_{\bar\alpha}(\tau-\tau_*)f\|_{Y^{\eta,\gamma}_1}& \lesssim e^{(\lambda+\delta)\tau} \|f\|_{{Y}^{\eta,\gamma}}.
    \end{align*}
As for \eqref{ultracontractivity}, let us start assuming $\frac{n}{2}\left(\frac{1}{\eta}-\frac{1}{\eta'}\right)<1.$ By the corresponding ultracontractivity property \eqref{ultracontractivity_free} for $S_0$, as well as the bounds on $\bar{\Omega}_{\bar \alpha}$ coming from \autoref{prop:Omega_alpha>0} and \autoref{lem:bound.rho.m.Omega}, we have for every $\tau\leq 1$
\begin{align*}
    \|S_{\bar\alpha}(\tau)f\|_{Y^{\eta',\gamma'}}
    &\lesssim 
    \frac{\|f\|_{Y^{\eta,\gamma}}}{\tau^{\frac{n}{2}\left(\frac{1}{\eta}-\frac{1}{\eta'}\right)}}
    +
    \int_0^\tau \frac{\|S_{\bar\alpha}(s)f\|_{ Y^{\eta,\gamma}_1}}{(\tau-s)^{\frac{n}{2}\left(\frac{1}{\eta}-\frac{1}{\eta'}\right)}} ds
    \\ 
    &\lesssim
    \frac{\|f\|_{L^{\eta,\gamma}}}{\tau^{\frac{n}{2}\left(\frac{1}{\eta}-\frac{1}{\eta'}\right)}}
    +
    \int_0^\tau \frac{\|f\|_{Y^{\eta,\gamma}}}{s^{1/2}(\tau-s)^{\frac{n}{2}\left(\frac{1}{\eta}-\frac{1}{\eta'}\right)}}ds
    \lesssim  
    \frac{\|f\|_{Y^{\eta,\gamma}}}{\tau^{\frac{n}{2}\left(\frac{1}{\eta}-\frac{1}{\eta'}\right)}}.
\end{align*}
In case of $\frac{n}{2}\left(\frac{1}{\eta}-\frac{1}{\eta'}\right)\geq 1$, one can successively perform the above steps for a finite number of intermediate values $\eta<\eta_i<\eta'$ and $\gamma<\gamma_i<\gamma'$. The case $\tau\geq 1$ is treated similarly as above.
\end{proof}

\section{Construction of Ancient Solutions} \label{sec:ancient}
Throughout this section we assume, even when not specified, that
\begin{align}\label{hp: coefficients 1}
    1\leq \hat{q}<\frac{n}{2}<\hat{r}<n,\quad \hat{r}\geq 2\hat{q}.
\end{align}
As outlined in \autoref{sec:strategy}, our goal is to construct two solutions of \eqref{eq:Omega} corresponding to the same initial datum at $\tau=-\infty$. One of these solutions is time-independent and is given by
\begin{align*}
    \Omega_1:=\bar{\Omega}_{\alpha},
\end{align*}
where $\bar{\Omega}_{\alpha}(y)=u_{\alpha}(|y|)$ and $u_{\alpha}$ solves the self-similar stationary equation \eqref{eq:ODE}. If, in addition, $\alpha$ is suitably chosen, then \autoref{thm:instability} applies and the linearized operator $L_{\alpha}$ admits positive eigenvalues.
Let $\lambda_\alpha$ be the corresponding maximal positive eigenvalue and $\bar{\Omega}^{lin}$ an associated eigenfunction. The second solution of \eqref{eq:Omega} that we aim to construct has the form
\begin{align*}
\Omega_2:=\bar{\Omega}_{\alpha}+\Psi:=\bar{\Omega}_{\alpha}+\Omega^{lin}+\Omega^{per},
\end{align*}
where $\Omega^{lin}=e^{\lambda_{\alpha}\tau}\bar{\Omega}^{lin}$.
This ansatz implies that $\Psi$ must solve
\begin{align}\label{eq:perturb}
\begin{dcases}
        \partial_{\tau}\Psi =L_{\alpha}\Psi+N(\Psi),
        \\
        \lim_{\tau\rightarrow-\infty}\Psi(\tau) =0,
    \end{dcases}    
\end{align}
or, equivalently, that $\Omega^{per}$ solves
\begin{align}\label{ancient_PDE}
    \begin{dcases}
        \partial_{\tau}\Omega^{per}=L_{\alpha}\Omega^{per}+N(\Omega^{per})+N(\Omega^{lin})+V(\Omega^{per}),
        \\
        \lim_{\tau\rightarrow-\infty}\Omega^{per}(\tau)=0,
    \end{dcases}
\end{align}
where $L_{\alpha}$ is defined in \eqref{def_lalpha} and $N,V$ in \eqref{nonlinearity_ancient}, respectively.
A function $\Psi$ solving \eqref{eq:perturb} and satisfying
\begin{align*}
    \|\Psi(\tau)\|_{Y_{\nabla}^{\hat{q},\hat{r}}}\rightarrow 0,
    \quad\text{as }\tau\rightarrow-\infty,
\end{align*}
will be called an \emph{ancient solution} of equation \eqref{eq:perturb}. Since 
\begin{align*}
    \|\Omega^{lin}(\tau)\|_{{Y_{\nabla}^{\hat{q},\hat{r}}}}=e^{\lambda_{\alpha}\tau}\|\bar\Omega^{lin}\|_{{Y_{\nabla}^{\hat{q},\hat{r}}}},
\end{align*}
it suffices to construct $\Omega^{per}\neq -\Omega^{lin}$ solving \eqref{ancient_PDE} and satisfying
\begin{align*}
    \|\Omega^{per}(\tau)\|_{Y_{\nabla}^{\hat{q},\hat{r}}}\rightarrow 0,
    \quad\text{as }\tau\rightarrow-\infty.
\end{align*}
This is the content of \autoref{prop:fixed_point_ancient}. Before stating and proving this result, we recall some properties of $\bar{\Omega}_{\alpha}$ and $\bar{\Omega}^{lin}$ and introduce some notation.\\
From the results of \autoref{subsec:selfsimilarprofile} and \autoref{subsec_unstable_eigen}, we know that 
\begin{align*}
    \bar{\Omega}_{\alpha}, \ \bar{\Omega}^{lin}\in C^2_{loc}(\R^{n+2}).
\end{align*}
Moreover,  
\begin{align}\label{eq:decay_ss}
 \bar{\Omega}_{\alpha}(\xi)=O(|\xi|^{-2}),
 \quad \mbox{and} \quad|\nabla\bar{\Omega}_{\alpha}(\xi)|=O(|\xi|^{-3}), \quad\text{for } |\xi|\rightarrow +\infty, 
\end{align}
while both $\bar{\Omega}^{lin}(\xi)$ and $|\nabla\bar{\Omega}^{lin}(\xi)|$ decay exponentially fast as $|\xi|\to +\infty$. Moreover, \autoref{cor:eigenvalues} gives that the maximal positive eigenvalue $\lambda_{\alpha}$ and the eigenfunction $\bar{\Omega}^{lin}$ are independent of the choice of $\hat{q},\hat{r}$, provided that $1\leq \hat{q}<\frac{n}{2}<\hat{r}$.
From now on, we drop the subscript $\alpha$ to simplify the notation. 

For $T \in \R$ and parameters $0<\delta\ll \lambda$, we introduce the space
\begin{align*}
\bar{Z}^T := \Big\{ f &\in C((-\infty,T];Y^{\hat{q},\hat{r}}_1)\cap C((-\infty,T];Y_{\nabla}^{\hat{q},\hat{r}}):\ |\xi|f \in C((-\infty,T];Y^{\hat{q},\infty}),\\
& \sup_{t\leq T} e^{-2(\lambda-\delta)t}
\left(
\|f(t)\|_{Y^{\hat{q},\hat{r}}_1}
+\||\xi|f(t)\|_{Y^{\hat{q},\infty}}
+\|f(t)\|_{Y_{\nabla}^{\hat{q},\hat{r}}}
\right)
<+\infty
\Big\}.
\end{align*}

This is a Banach space when endowed with its natural norm, denoted by $\|\cdot\|_{\bar{Z}^T}$. We also denote by $B^{\epsilon,T}$ the closed ball in $\bar{Z}^T$ of radius $\epsilon$. The following result holds.
\begin{prop}\label{prop:fixed_point_ancient}
Assume \eqref{hp: coefficients 1}, and let $\alpha>0$ be such that the corresponding expander $\bar{\Omega}$ yields a linearized operator 
$L:D(L)\subset Y^{\hat{q},\hat{r}}\rightarrow Y^{\hat{q},\hat{r}}$ with maximal positive eigenvalue $\lambda>0$.
Then, for every $\delta \in (0,\tfrac{\lambda}{10})$, there exists $\bar{\epsilon}=\bar{\epsilon}(\delta)>0$ such that, for all $\epsilon \in (0,\bar{\epsilon})$, there exists $\bar{T}=\bar{T}(\epsilon,\delta)<0$ with the property that equation \eqref{ancient_PDE} admits a unique solution in $B^{\epsilon,T}$ for every $T<\bar{T}$.
\end{prop}
\begin{proof}
    We construct the solution as a fixed point of the map $\Gamma$ defined as follows: given $u\in B^{\epsilon,T} $, we define $v:=\Gamma u$ as the solution of the linear PDE 
    \begin{align}\label{linearized_ancient_PDE}
        \begin{dcases}
            \partial_\tau v=Lv+N(u)+N(\Omega^{lin})+V(u),
            \\
            \lim_{\tau\rightarrow -\infty}v(\tau)=0.
        \end{dcases}
    \end{align}
 \emph{Step 1}.
 Our first goal is to show that $\Gamma$ maps $B^{\epsilon,T}$ in itself. 
 We begin by estimating the term $N(u)+N(\Omega^{lin})+V(u)$. Using the assumption $2\hat{q}\leq \hat{r}$, the regularity of $\Omega^{lin}$ (cf. the discussion before the statement), Young’s inequality, and interpolation, we obtain
    \begin{align*}
    \|(N(u)+N(\Omega^{lin})+V(u))(r)\|_{Y^{\hat{q}}}& \lesssim \|u(r)\|_{Y^{2\hat{q}}}^2+\|\xi\cdot\nabla u(r)\|_{Y^{2\hat{q}}}^2+ e^{2\lambda r}\\ & \lesssim\|u(r)\|_{Y^{\hat{q}}}^2+\|u(r)\|_{Y^{\hat{r}}}^2+\|\xi\cdot\nabla u(r)\|_{Y^{\hat{q}}}^2+\|\xi\cdot\nabla u(r)\|_{Y^{\hat{r}}}^2+ e^{2\lambda r}\\ & \lesssim e^{2\lambda r}+ \epsilon^2 e^{4(\lambda-\delta)r}.
    \end{align*}
    Next, define $r^*,\hat{r}'$ as
    \begin{align}\label{def_auxuliary_parameters}
        \frac{1}{r^*}:=\frac{1}{\hat{r}}-\frac{1}{n},
        \quad 
        \frac{1}{\hat{r}'}:=\frac{1}{r^*}+\frac{1}{\hat{r}}.
    \end{align}
    By the Sobolev embedding \autoref{weighted_sobolev}, we also have
    \begin{align*}
    \|(N(u)+N(\Omega^{lin})+V(u))(r)\|_{Y^{\hat{r}'}}& \lesssim \|u(r)\|_{Y^{\hat{q},\hat{r}}_1}\|\xi\cdot\nabla u(r)\|_{Y^{\hat{r}}} +\|u(r)\|_{Y^{\hat{q},\hat{r}}_1}\|u(r)\|_{Y^{\hat{r}}}+e^{2\lambda r}
    \\ & \lesssim e^{2\lambda r}+ \epsilon^2 e^{4(\lambda-\delta)r}.
    \end{align*}
    
Combining these bounds with the semigroup representation: 
\begin{align}\label{eq:semigroup_form_1}
        v(\tau)=\int_{-\infty}^{\tau}S(\tau-r)\left[N(u)+N(\Omega^{lin})+V(u)\right](r)dr,
    \end{align}
and using smoothing and ultracontractivity properties of the semigroup given respectively by \eqref{smoothing} and \eqref{ultracontractivity}, we obtain for $s:=\frac{n}{2}\left(\frac{1}{\hat{r}}-\frac{1}{n}\right)<\frac{1}{2}$ the following chain of inequalities:
\begin{align*}
    \|v(\tau)\|_{Y^{\hat{q},\hat{r}}_1}
    &\lesssim 
    \int_{-\infty}^{\tau}\frac{e^{\frac{(\lambda+\delta)(\tau-r)}{2}}}{(\tau-r)^{1/2}}\left\|S\left(\frac{\tau-r}{2}\right)\left[N(u)+N(\Omega^{lin})+V(u)\right](r) \right\|_{Y^{\hat{q},\hat{r}}} dr 
    \\ 
    &\lesssim 
    \int_{-\infty}^{\tau}\frac{e^{(\lambda+\delta)(\tau-r)}}{(\tau-r)^{1/2+s}}\left(e^{2\lambda r}+\epsilon^2 e^{4(\lambda-\delta)r}\right)dr
    \\ 
    &\leq
    \int_{-\infty}^{\tau -1} 
    e^{(\lambda+\delta)(\tau-r)} \left(e^{2\lambda r}+\epsilon^2 e^{4(\lambda-\delta)r}\right)dr
    +
    \int_{\tau-1}^{\tau} \frac{e^{(\lambda+\delta)(\tau-r)}}{(\tau-r)^{1/2+s}}\left(e^{2\lambda r}+\epsilon^2 e^{4(\lambda-\delta)r}\right)dr
    \\ 
    &=:
    I_1(\tau)+I_2(\tau).
\end{align*}
One trivially has 
\begin{align}\label{estimate_1_smooth_fixed}
    I_1(\tau)&\lesssim e^{2\lambda \tau}+\epsilon^2 e^{4(\lambda-\delta)\tau}\leq (e^{2\delta T}+\epsilon^2)e^{2(\lambda-\delta)\tau}.
\end{align}
Moreover, since $s<1/2$, by H\"older inequality it holds
\begin{align}\label{estimate_2_smooth_fixed}
    I_2(\tau)& \lesssim e^{2\lambda \tau}+\epsilon^2 e^{4(\lambda-\delta)\tau}\leq (e^{2\delta T}+\epsilon^2)e^{2(\lambda-\delta)\tau}.
\end{align}
In conclusion, we have just proved the bound 
\begin{align}\label{estimate_1_ball}
  e^{-2(\lambda-\delta)\tau} \|v(\tau)\|_{Y^{\hat{q},\hat{r}}_1}&\lesssim (e^{2\delta T}+\epsilon^2).
\end{align} 

To capture spatial decay, we introduce the weighted quantity $h:=|\xi|v$. A direct computation shows that $h$ satisfies
\begin{align}\label{linearized_ancient_PDE_aux1}
        \begin{dcases}
            \partial_\tau h=\left(L-\frac{1}{2}\right)h-2\frac{\xi}{\lvert \xi\rvert}\cdot\nabla v-{(n+1)}\frac{v}{\lvert \xi\rvert}-2|\xi|\bar{\Omega}v+|\xi|\left(N(u)+N(\Omega^{lin})+V(u)\right),
            \\
            \lim_{\tau\rightarrow -\infty}h(\tau)=0.
        \end{dcases}
    \end{align} 
Using the mild formulation for $h,$ we write
\begin{align}\label{eq:semigroup_form_2}
    h(\tau)=\int_{-\infty}^{\tau} e^{-\frac{1}{2}(\tau-r)}S(\tau-r)\left[-2\frac{\xi}{\lvert \xi\rvert}\nabla v-(n+1)\frac{v}{\lvert \xi\rvert}-2|\xi|\bar{\Omega}v+|\xi|\left(N(u)+N(\Omega^{lin})+V(u)\right)\right](r)dr.
\end{align}
Thanks to the properties of $\bar{\Omega}$ (cf. equation \eqref{eq:decay_ss}), the bound \eqref{estimate_1_ball} just obtained on $v$, and Hardy's inequality (\autoref{weighted_sobolev}), we can control:
\begin{align}\label{estimate_1_weight}
\left\|\left(2\frac{\xi}{\lvert \xi\rvert}\cdot\nabla v+(n+1)\frac{v}{\lvert \xi\rvert}+2|\xi|\bar{\Omega}v\right)(r)\right\|_{Y^{\hat{q},\hat{r}}}& \lesssim \|v(r)\|_{Y^{\hat{q},\hat{r}}_1}\leq (e^{2\delta T}+\epsilon^2)e^{2(\lambda-\delta)r}.
\end{align}
On the other hand, the weighted nonlinear terms satisfy, arguing similarly as above, 
\begin{align}\label{eq:weighted_nonlinearity}
\left\||\xi|\left(N(u)+N(\Omega^{lin})+V(u)\right)(r) \right\|_{Y^{\hat{q},\hat{r}}} 
&\lesssim  
\||\xi|u(r)\|_{Y^{\hat{q},\infty}}^2
+
\|u(r)\|_{Y^{\hat{q},\hat{r}}}^2
+
\|\xi\cdot\nabla u(r)\|_{Y^{\hat{q},\hat{r}}}^2
+
e^{2\lambda r}\notag
\\ 
& \lesssim 
\epsilon^2 e^{4(\lambda-\delta) r}+e^{2\lambda r}.
\end{align}
Now we are ready to estimate the $Y^{\hat{q},\hat{r}}_1$ and $Y^{\infty}$ norms of $h$.
Using again the smoothing properties \eqref{smoothing} and arguing as in \eqref{estimate_1_smooth_fixed} and \eqref{estimate_2_smooth_fixed}, we conclude
\begin{align*}
    \|h(\tau)\|_{Y^{\hat{q},\hat{r}}_1}& 
    \lesssim
    \int_{-\infty}^{\tau} e^{-\frac{1}{2}(\tau-r)}\left\|S(\tau-r)\left[-2\frac{\xi}{\lvert \xi\rvert}\nabla v-(n+1)\frac{v}{\lvert \xi\rvert}-2|\xi|\bar{\Omega}v \right](r)\right\|_{Y^{\hat{q},\hat{r}}_1} dr
    \\ &+\int_{-\infty}^{\tau}e^{-\frac{1}{2}(\tau-r)}\left\|S(\tau-r)\left[|\xi|(N(u)+N(\Omega^{lin})+V(u))(r) \right]\right\|_{Y^{\hat{q},\hat{r}}_1} dr 
    \\ & \lesssim
    (e^{2\delta T}+\epsilon^2)\int_{-\infty}^{\tau} \frac{e^{(\lambda+\delta-\frac{1}{2})(\tau-r)}e^{2(\lambda-\delta)r}}{(\tau-r)^{1/2}} dr+\int_{-\infty}^{\tau}\frac{e^{(\lambda+\delta-\frac{1}{2})(\tau-r)}\left(e^{2\lambda r}+\epsilon^2 e^{4(\lambda-\delta)r}\right)}{(\tau-r)^{1/2}}dr\\
    & \lesssim
    (e^{2\delta T}+\epsilon^2)\int_{-\infty}^{\tau-1} e^{(\lambda+\delta-\frac{1}{2})(\tau-r)}e^{2(\lambda-\delta)r} dr+\int_{-\infty}^{\tau-1}e^{(\lambda+\delta-\frac{1}{2})(\tau-r)}\left(e^{2\lambda r}+\epsilon^2 e^{4(\lambda-\delta)r}\right) dr
    \\ & + (e^{2\delta T}+\epsilon^2)\int_{\tau-1}^{\tau} \frac{e^{(\lambda+\delta-\frac{1}{2})(\tau-r)}e^{2(\lambda-\delta)r}}{(\tau-r)^{1/2}} dr+\int_{\tau-1}^{\tau}\frac{e^{(\lambda+\delta-\frac{1}{2})(\tau-r)}\left(e^{2\lambda r}+\epsilon^2 e^{4(\lambda-\delta)r}\right)}{(\tau-r)^{1/2}}dr\\
    &\lesssim (e^{2\delta T}+\epsilon^2)e^{2(\lambda-\delta)\tau}+e^{2\lambda\tau}+\epsilon^2e^{4(\lambda-\delta)\tau}.
\end{align*}
 Therefore 
\begin{align}\label{estimate_2_ball}
  e^{-2(\lambda-\delta)\tau} \|h(\tau)\|_{Y^{\hat{q},\hat{r}}_1}&\lesssim (e^{2\delta T}+\epsilon^2).
\end{align}
Lastly, by similar computations and \eqref{ultracontractivity} we have for $s:=\frac{n}{2\hat{r}}<1$ the bound
\begin{align*}
  \|h(\tau)\|_{Y^{\infty}}& 
    \lesssim
    (e^{2\delta T}+\epsilon^2)\int_{-\infty}^{\tau} \frac{e^{(\lambda+\delta-\frac{1}{2})(\tau-r)}e^{2(\lambda-\delta)r}}{(\tau-r)^{s}} dr+\int_{-\infty}^{\tau}\frac{e^{(\lambda+\delta-\frac{1}{2})(\tau-r)}\left(e^{2\lambda r}+\epsilon^2 e^{4(\lambda-\delta)r}\right)}{(\tau-r)^{s}}dr\\
    &\lesssim (e^{2\delta T}+\epsilon^2)e^{2(\lambda-\delta)\tau}+e^{2\lambda\tau}+\epsilon^2e^{4(\lambda-\delta)\tau}. 
\end{align*}

The previous bounds on $h$ and equation \eqref{estimate_1_ball} imply, in particular:
\begin{align}\label{estimate_3_ball}
     e^{-2(\lambda-\delta)\tau} 
     \left( \||\xi|v(\tau)\|_{Y^{\hat{q},\hat{r}}}
     +
     \|\xi\cdot\nabla v(\tau)\|_{Y^{\hat{q},\hat{r}}}
     +
      \||\xi|v(\tau)\|_{Y^{\infty}}
     \right)
     &\lesssim (e^{2\delta T}+\epsilon^2).
\end{align}
The time continuity of $v$ and $h$ can be established by analogous arguments, based on the mild formulations \eqref{eq:semigroup_form_1}-\eqref{eq:semigroup_form_2} and the Lebesgue and Vitali convergence theorems. We omit the easy details.\\
Combining the previous estimates \eqref{estimate_1_ball}, \eqref{estimate_2_ball}, and \eqref{estimate_3_ball}, we conclude that $\Gamma$ maps $B^{\epsilon,T}$ in itself, provided $\epsilon>0$ is sufficiently small and $T<0$ is sufficiently negative.

\emph{Step 2}.
We are left to show that $\Gamma$ is a contraction, possibly taking smaller $\epsilon$ and more negative $T$. This follows by arguments similar to the ones described above. More precisely, let $u_1,u_2\in B^{\epsilon,T}$ and set $v_1:=\Gamma u_1$, $v_2:=\Gamma u_2$, $h_1:=|\xi|v_1$, $h_2:=|\xi|v_2$.  Then 
\begin{align*}
(v_1-v_2)(\tau)
=
\int_{-\infty}^{\tau}S(\tau-r)\left[N(u_1)-N(u_2)+V(u_1-u_2)\right](r)dr.
\end{align*}
    First, since $u_1,u_2\in B^{\epsilon,T}$, proceeding as before we obtain for $\hat{r}'$ defined in \eqref{def_auxuliary_parameters}:
    \begin{align*}
    \|(N(u_1)-N(u_2)+V(u_1-u_2))(r)\|_{Y^{\hat{q}}}
    &\lesssim 
    \|u_1-u_2\|_{\bar{Z}^T} e^{2(\lambda-\delta)r}\left(\epsilon e^{2(\lambda-\delta)r}+e^{\lambda r}\right),
    \\
    \|(N(u_1)-N(u_2)+V(u_1-u_2))(r)\|_{Y^{\hat{r}'}}
    &\lesssim 
    \|u_1-u_2\|_{\bar{Z}^T} e^{2(\lambda-\delta)r}\left(\epsilon e^{2(\lambda-\delta)r}+e^{\lambda r}\right).
    \end{align*}
    Combining these bounds with the smoothing and ultracontractivity properties of the semigroup, cf. \eqref{smoothing}, \eqref{ultracontractivity}, we obtain   for $s:=\frac{n}{2}\left(\frac{1}{\hat{r}}-\frac{1}{n}\right)<\frac{1}{2}$
\begin{align*}
    \|(v_1-v_2)(\tau)\|_{Y^{\hat{q},\hat{r}}_1}
    & \lesssim 
    \int_{-\infty}^{\tau} \frac{e^{\frac{(\lambda+\delta)(\tau-r)}{2}}}{(\tau-r)^{1/2}} \left\|S\left(\frac{\tau-r}{2}\right)(N(u_1)-N(u_2)+V(u_1-u_2))(r) \right\|_{Y^{\hat{q},\hat{r}}} dr 
    \\ & \lesssim 
    \|u_1-u_2\|_{\bar{Z}^T}\int_{-\infty}^{\tau}\frac{e^{(\lambda+\delta)(\tau-r)}}{(\tau-r)^{1/2+s}} e^{2(\lambda-\delta)r}\left(\epsilon e^{2(\lambda-\delta)r}+e^{\lambda r}\right)dr
    \\ &\leq
    \|u_1-u_2\|_{\bar{Z}^T}\left(\int_{-\infty}^{\tau} e^{(\lambda+\delta)(\tau-r)} e^{2(\lambda-\delta)r}\left(\epsilon e^{2(\lambda-\delta)r}+e^{\lambda r}\right)dr\right.\\ & \qquad \qquad \qquad \left.+\int_{\tau-1}^{\tau}\frac{e^{(\lambda+\delta)(\tau-r)}}{(\tau-r)^{1/2+s}} e^{2(\lambda-\delta)r}\left(\epsilon e^{2(\lambda-\delta)r}+e^{\lambda r}\right)dr\right)\\ &=:\|u_1-u_2\|_{\bar{Z}^T}\left(J_1(\tau)+J_2(\tau)\right).
\end{align*}
The two terms $J_1(\tau), J_2(\tau)$ above can be treated as in \eqref{estimate_1_smooth_fixed} and \eqref{estimate_2_smooth_fixed}, leading to
\begin{align}\label{estimate_1_contraction}
  e^{-2(\lambda-\delta)\tau} \|(v_1-v_2)(\tau)\|_{Y^{\hat{q},\hat{r}}_1}&\lesssim \left(\epsilon e^{2(\lambda-\delta)\tau}+e^{\lambda \tau}\right)\|u_1-u_2\|_{\bar{Z}^T}.
\end{align} 
Lastly, we study the time behavior of $h_1-h_2$. Using the mild formulation, we have
\begin{align*}
    (h_1-h_2)(\tau)&=\int_{-\infty}^{\tau} e^{-\frac{1}{2}(\tau-r)}S(\tau-r)\bigg[-2\frac{\xi}{\lvert \xi\rvert}\nabla (v_1-v_2)-(n+1)\frac{v_1-v_2}{\lvert \xi\rvert}-2|\xi|\bar{\Omega}(v_1-v_2)\\ &\qquad\qquad \qquad\qquad \qquad\qquad+|\xi|\left(N(u_1)-N(u_2)+V(u_1-u_2)\right)\bigg](r)dr.
\end{align*}
Arguing as above we can bound
\begin{multline}
\label{estimate_1_weight_contr}
\left\|\left(2\frac{\xi}{\lvert \xi\rvert}\nabla (v_1-v_2)+(n+1)\frac{v_1-v_2}{\lvert \xi\rvert}+2|\xi|\bar{\Omega}(v_1-v_2)\right)(r)\right\|_{Y^{\hat{q},\hat{r}}}\\ \lesssim \|(v_1-v_2)(r)\|_{Y^{\hat{q},\hat{r}}_1} \lesssim e^{2(\lambda-\delta)r}\left(\epsilon e^{2(\lambda-\delta)r}+e^{\lambda r}\right)\|u_1-u_2\|_{\bar{Z}^T},
\end{multline}
and
\begin{align}\label{control_1_weighet_term_contr}
\||\xi|\left(N(u_1)-N(u_2)+V(u_1-u_2)\right)(r) \|_{Y^{\hat{q},\hat{r}}}&\lesssim  \|u_1-u_2\|_{\bar{Z}^T} e^{2(\lambda-\delta)r}\left(\epsilon e^{2(\lambda-\delta)r}+e^{\lambda r}\right).
\end{align}
Therefore, by the smoothing property \eqref{smoothing} and the bounds \eqref{estimate_1_weight_contr}, \eqref{control_1_weighet_term_contr} we can control the forcing terms appearing in the mild formulation for $h_1-h_2$ obtaining
\begin{align*}
    \|(h_1-h_2)(\tau)\|_{Y^{\hat{q},\hat{r}}_1}
    &\lesssim
    \|u_1-u_2\|_{\bar{Z}^T}\int_{-\infty}^{\tau} \frac{e^{(\lambda+\delta-\frac{1}{2})(\tau-r)}}{(\tau-r)^{1/2}}  e^{2(\lambda-\delta)r}\left(\epsilon e^{2(\lambda-\delta)r}+e^{\lambda r}\right) dr\\ & +\|u_1-u_2\|_{\bar{Z}^T}\int_{-\infty}^{\tau}\frac{e^{(\lambda+\delta-\frac{1}{2})(\tau-r)}}{(\tau-r)^{1/2}}e^{2(\lambda-\delta)r}\left(\epsilon e^{2(\lambda-\delta)r}+e^{\lambda r}\right)dr\\ &\lesssim \|u_1-u_2\|_{\bar{Z}^T} e^{2(\lambda-\delta)\tau}\left(\epsilon e^{2(\lambda-\delta)\tau}+e^{\lambda \tau}\right)\\ &+\|u_1-u_2\|_{\bar{Z}^T}\int_{\tau-1}^{\tau} \frac{e^{(\lambda+\delta-\frac{1}{2})(\tau-r)}}{(\tau-r)^{1/2}}e^{2(\lambda-\delta)r}\left(\epsilon e^{2(\lambda-\delta)r}+e^{\lambda r}\right)  dr\\ &\lesssim \|u_1-u_2\|_{\bar{Z}^T}e^{2(\lambda-\delta)\tau}\left(\epsilon e^{2(\lambda-\delta)\tau}+e^{\lambda \tau}\right).
\end{align*}
Therefore 
\begin{align} \label{estimate_2_contraction}
  e^{-2(\lambda-\delta)\tau} \|(h_1-h_2)(\tau)\|_{Y_1^{\hat{q},\hat{r}}}&\lesssim \|u_1-u_2\|_{\bar{Z}^T} \left(\epsilon e^{2(\lambda-\delta)\tau}+e^{\lambda \tau}\right).
\end{align}
The latter and equation \eqref{estimate_1_contraction} imply
\begin{align}
\label{estimate_3_contraction}
   e^{-2(\lambda-\delta)\tau} \|(h_1-h_2)(\tau)\|_{Y^{\hat{q},\hat{r}}}&+e^{-2(\lambda-\delta)\tau} \|\xi\cdot\nabla (v_1-v_2)(\tau)\|_{Y^{\hat{q},\hat{r}}}\notag
   \\
   &\lesssim \|u_1-u_2\|_{\bar{Z}^T} \left(\epsilon e^{2(\lambda-\delta)\tau}+e^{\lambda \tau}\right). 
\end{align}
Lastly, by \eqref{ultracontractivity} and similar computations
\begin{align}
\label{estimate_4_contraction}
    e^{-2(\lambda-\delta)\tau}\|(h_1-h_2)(\tau)\|_{Y^{\infty}}\lesssim \|u_1-u_2\|_{\bar{Z}^T} \left(\epsilon e^{2(\lambda-\delta)\tau}+e^{\lambda \tau}\right).
\end{align}
Collecting the estimates \eqref{estimate_1_contraction}, \eqref{estimate_2_contraction}, \eqref{estimate_3_contraction}, \eqref{estimate_4_contraction},  we conclude that $\Gamma$ is a contraction on  $B^{\epsilon,T}$, provided $\epsilon>0$ is sufficiently small and $T<0$ sufficiently negative. Therefore, by the Banach fixed point theorem, a unique solution exists in $B^{\epsilon,T}$, concluding the proof.
\end{proof}
The different asymptotic behavior of $\Omega^{lin}$ and $\Omega^{per}$ as $\tau\rightarrow-\infty$ implies that $\Psi\neq 0$, and consequently $\Omega_1\neq \Omega_2$. As a corollary of \autoref{prop:fixed_point_ancient} we now derive the main result of this section; we formulate it in the form of a theorem to be invoked later on.	\begin{theorem}\label{thm:existence_ancient_solutions}
		Assume \eqref{hp: coefficients 1}, and let $\alpha>0$ be such that the corresponding expander $\bar{\Omega}$ yields a linearized operator $L: \mathcal{D}(L) \subset Y^{\hat{q},\hat{r}} \rightarrow  Y^{\hat{q},\hat{r}}$ with maximal positive eigenvalue $\lambda>0$. Then, for every sufficiently small $\eps>0$ and $p\in [\hat{q},\hat{r}]$ there exist $T<0$ and an ancient solution $\Psi \in C((-\infty,T],Y^{\hat{q},\hat{r}})$ to \eqref{eq:perturb} satisfying
		\begin{align*}
        \Psi\in C((-\infty,T];Y^{\hat{q},\hat{r}}_1\cap Y^{\hat{q},\hat{r}}_\nabla ),\quad \ |\xi|\Psi\in C((-\infty,T];Y^{\hat{q},\infty}),\\ 
			\sup_{\tau\leq T }e^{-\frac{\lambda}{2}\tau}\left(\lVert \Psi(\tau)\rVert_{Y^{\hat{q},\hat{r}}_1}+\lVert |\xi|\Psi(\tau)\rVert_{Y^{\hat{q},\hat{r}}_1}+\lVert |\xi|\Psi(\tau)\rVert_{Y^{\infty}}\right)<\eps,
		\end{align*}
		and 
		\begin{align*}
			\lVert \Psi(\tau)\rVert_{Y_{\nabla}^p}> e^{\lambda\tau}\frac{\lVert \bar{\Omega}^{lin}\rVert_{Y_{\nabla}^p}}{2},
            \quad
            \mbox{ for }\tau\in (-\infty,T].
		\end{align*}
	\end{theorem}
	\begin{proof}
		Let us choose $\delta,\frac{\eps}{2},T$ such that \autoref{prop:fixed_point_ancient} holds. Furthermore, by possibly choosing more negative $T$, we ensure that
		\begin{align}\label{cond 1 barT}
			e^{\frac{\lambda}{2}T}&< \frac{\eps}{2\left(\lVert \bar{\Omega}^{lin}\rVert_{Y_1^{\hat{q},\hat{r}}}+\lVert |\xi|\bar{\Omega}^{lin}\rVert_{Y^{\hat{q},\hat{r}}_1}+\lVert |\xi|\bar{\Omega}^{lin}\rVert_{Y^{\infty}}\right)},\\
			e^{(\lambda-2\delta) T}&< \frac{\lVert {\bar{\Omega}}^{lin}\rVert_{Y_{\nabla}^p}}{\eps}. \label{cond 2 barT}
		\end{align}
		Let us now define for $\tau\in (-\infty,T]$
		\begin{align*}
			\Psi(\tau):=\Omega^{lin}(\tau)+\Omega^{per}(\tau).
		\end{align*}
		In particular, we have that $\Psi\in C((-\infty,T], Y^{\hat{q},\hat{r}})$, $\Psi$ solves \eqref{eq:perturb}, and has the additional claimed regularity. Moreover, by the time decaying properties of $\Omega^{lin}, \Omega^{per}$ and the choice of $T$ in \eqref{cond 1 barT}, it follows that
		\begin{align*}
			\left(\lVert \Psi(\tau)\rVert_{Y_1^{\hat{q},\hat{r}}}+\lVert |\xi|\Psi(\tau)\rVert_{Y_1^{\hat{q},\hat{r}}}+\lVert |\xi|\Psi(\tau)\rVert_{Y^{\infty}}\right)<\eps e^{\frac{\lambda}{2}\tau},
            \quad \text{for } \tau\in(-\infty,T].
		\end{align*}
		In addition, by \eqref{cond 2 barT}, the time decaying properties of $\Omega^{per}$, and interpolation, we have
		\begin{align*}
			\lVert \Omega^{per}(\tau)\rVert_{Y_{\nabla}^p}<\frac{\eps}{2}e^{(\lambda-2\delta)\tau}e^{\lambda\tau}<\frac{\lVert \bar{\Omega}^{lin}\rVert_{Y_{\nabla}^p}}{2}e^{\lambda\tau},
            \quad \text{for } \tau\in(-\infty,T].  
		\end{align*}
		Therefore, recalling the definition of $\Omega^{lin}$:
		\begin{align*}
			\lVert \Psi(\tau)\rVert_{Y_{\nabla}^p}&\geq \lVert \Omega^{lin}(\tau)\rVert_{Y_{\nabla}^p}-\lVert \Omega^{per}(\tau)\rVert_{Y_{\nabla}^p}   > \frac{\lVert \bar{\Omega}^{lin}\rVert_{Y_{\nabla}^p}}{2}e^{\lambda\tau},
		\end{align*}
		for  $\tau\in (-\infty,T]$. This completes the proof.
	\end{proof}

\subsection{Proof of \autoref{thm:critical}} 
\label{ssec:proof.critical}
    
The previous \autoref{thm:existence_ancient_solutions} directly implies non-uniqueness for \eqref{eq:KS_physical} in the scaling-critical space $L^{\infty}([0,T'];L^{\frac{n}{2},\infty}(\R^n))$, where $T'>0$ and $L^{\frac{n}{2},\infty}(\R^n)$ is the Lorentz space. 
Namely, it allows to construct an initial datum $c_0\in L^{\frac{n}{2},\infty}(\R^n)$ and two distinct functions $c_1,c_2\in L^{\infty}([0,T'];L^{\frac{n}{2},\infty}(\R^n))$ satisfying \begin{align} \label{eq:regularity.nonlinearity}
    c_i\nabla(-\Delta)^{-1}c_i \in L^1([0,T'];L^1_{loc}(\R^n)),
    \quad
    i \in \{1,2\},
\end{align}
and \eqref{eq:KS_physical} in the sense of distributions on $(0,T)\times \R^n$. Moreover the initial datum is attained in the sense that for all $\phi\in C^{\infty}_c(\R^n)$:
\begin{align*}
    \lim_{t\rightarrow 0}\int_{\R^n}c_1(t,x)\phi(x)dx=\lim_{t\rightarrow 0}\int_{\R^n}c_2(t,x)\phi(x)dx=\int_{\R^n}c_0(x)\phi(x)dx.
\end{align*}

Assume \eqref{hp: coefficients 1}, take $\bar{\alpha}>0$ such that \autoref{thm:existence_ancient_solutions} applies, and denote by $u_{\bar{\alpha}}(|\cdot|)$ the profile associated with $\bar{\Omega}$, i.e. $\bar\Omega := u_{\bar{\alpha}}(|\cdot|)$.
Define $T':=e^T$, the initial condition
\begin{align*}
    c_0(x):=\frac{2(n-2)\ell_{\bar{\alpha}}}{|x|^2},
\end{align*}
and for $t\in (0,T']$
\begin{align*}
    c_1(t,x)&:=\frac{2}{t}\left(nu_{\bar{\alpha}}\left(\frac{|x|}{\sqrt{t}}\right)+\frac{|x|}{\sqrt{t}}u_{\bar{\alpha}}'\left(\frac{|x|}{\sqrt{t}}\right)\right),
    \quad 
    c_2(t,x)=c_1(t,x)+\frac{1}{t}A^{-1}\left[\Psi\left(\log t,\frac{  \cdot
    }{\sqrt{t}}\right) \right](x).
\end{align*}
By \autoref{thm:existence_ancient_solutions}, the function $\Psi$ satisfies 
\begin{align*}
    \Psi\in C((-\infty,T];Y^{\hat{q},\hat{r}}_{\nabla}),
    \quad \|\Psi(\tau)\|_{Y^{\hat{q},\hat{r}}_\nabla}\leq e^{\lambda\tau/2}, 
\end{align*}
and therefore
\begin{align*}
    \frac{1}{t}A^{-1}\Psi\left(\log t,\frac{\cdot}{\sqrt{t}}\right)\in C([0,T'];X^{\hat{q},n/2}) \cap C((0,T'];X^{\hat{r}}),
    \quad
    \quad  \lim_{t\rightarrow 0}\frac{1}{t}\left\|A^{-1}\Psi\left(\log t,\frac{\cdot}{\sqrt{t}}\right)\right\|_{L^{n/2}(\R^n)}=0.
\end{align*}
Combined with \eqref{eq:decay_ss}, this implies 
\begin{align*}
    c_1,c_2\in L^{\infty}([0,T'];L^{\frac{n}{2},\infty}(\R^n))\cap C((0,T'];L^{\hat{r}}(\R^n)).
\end{align*}
Also, it holds
\begin{align*}
    \|c_1(t)\|_{L^{\hat{r}}}+\|c_2(t)\|_{L^{\hat{r}}} &\lesssim \frac{1}{t^{1-\frac{n}{2\hat{r}}}}.
\end{align*}
By standard properties of Riesz transforms and Riesz potential of order 1, cf. \cite[Chapters 3-4]{stein1970singular}, $\nabla(-\Delta)^{-1}$ is a linear bounded operator between $L^{\hat{r}}(\R^n)$ and $L^{\frac{\hat{r}n}{n-\hat{r}}}(\R^n)$. The latter implies \eqref{eq:regularity.nonlinearity} since $\frac{n}{2}<\hat{r}<n.$
Lastly, if $\phi\in C^{\infty}_c(\R^n)$ is radially symmetric with radial part $\varphi$, then using integration by parts and dominated convergence we get
\begin{align*}
    \int_{\R^n}c_1(t,x)\phi(x) dx
    &=
    \frac{2\omega_{n-1}}{t}\int_0^{+\infty }\left(nu_{\bar{\alpha}}\left(\frac{r}{\sqrt{t}}\right)+\frac{r}{\sqrt{t}}u_{\bar{\alpha}}'\left(\frac{r}{\sqrt{t}}\right)\right) \varphi(r) r^{n-1}dr
    \\ 
    & = 
    -\frac{2\omega_{n-1}}{t}\int_0^{+\infty }u_{\bar{\alpha}}\left(\frac{r}{\sqrt{t}}\right)\varphi'(r)r^{n}dr
    \\ 
    & \rightarrow  
    -2\ell_{\bar\alpha}\omega_{n-1}\int_0^{+\infty }\varphi'(r)r^{n-2}dr
    =
    \int_{\R^n}c_0(x)\phi(x)dx,\quad \text{as }t\rightarrow 0. 
\end{align*}
The general case of $\phi\in C^{\infty}_c(\R^n)$ follows by observing that 
\begin{align*}
    \int_{\R^n}c_1(t,x)\phi(x) dx&=\int_{\R^n}c_1(t,x)\bar{\phi}(x) dx, 
\end{align*}
where $\bar{\phi}(x)$ is the average of $\phi$ over the sphere of radius $|x|$, namely
\begin{align}
    \bar{\phi}(x)&:=\frac{1}{\omega_{n-1}}\int_{\mathbb{S}^{n-1}}\phi(|x|\hat{\theta})\sigma(d\hat{\theta}),
\end{align}
that clearly is radially symmetric and satisfies $\bar\phi \in C^\infty_c(\R^n)$.

\section{Localization procedure} \label{sec:localization}
From the results of the previous section, it is possible to construct two solutions to the problem
\begin{align*}
   \begin{dcases}
        \partial_t w=\Delta w+2nw^2+y\cdot\nabla (w^2), \quad \text{in } (t,y) \in [0,e^T] \times \R^{n+2},\\
        w(0,y)=\frac{\ell_{\alpha}}{|y|^2},
   \end{dcases}
\end{align*}
living in the space $C([0,e^T];L^1_{loc}(\R^{n+2},|y|^{-2}dy))$.
One solution is given by
\begin{align*}
    \tilde{w}_1(t,y):=\bar{w}(t,y):=\frac{1}{t}\bar{\Omega}\left(\frac{y}{\sqrt{t}}\right),
\end{align*}
while the second one has the form
\begin{align*}
    \tilde{w}_2(t,y)=\bar{w}(t,y)+\frac{1}{t}\Psi\left(\log t,\frac{y}{\sqrt{t}}\right),
\end{align*}
where $\bar{\Omega}$ is the self-similar profile associated with an unstable eigenvalue (see \autoref{thm:instability}), and $\Psi$ is the corresponding ancient solution provided by \autoref{thm:existence_ancient_solutions}.
However, for $q<\frac{n}{2}$, neither these solutions nor the initial datum belong to $Y^q_{\nabla}$ due to their insufficient decay as $|y|\to+\infty$. Consequently, the associated ${c}_i := A^{-1}[\tilde{w}_i]$ do not lie in $L^q(\R^n)$.

As outlined in \autoref{sec:strategy}, the goal of this section is to localize $\tilde{w}_1$ and $\tilde{w}_2$ by suppressing their tails at the initial time. This is carried out in \autoref{subsec:localization_PDE}, where the problem is addressed by studying a nonlinear equation of the form \eqref{eq:w} with singular lower-order terms. Finally, in \autoref{subsec:end_proof}, we complete the proof of \autoref{thm:non.uniqueness} by showing that this localization procedure yields two distinct mild $L^q$-solutions to \eqref{eq:KS_physical} arising from the same initial datum.
\subsection{A Keller-Segel--type equation with singular lower order terms}\label{subsec:localization_PDE}
We recall that the results of \autoref{subsec:selfsimilarprofile} yield $\bar{\Omega}\in C^2_{loc}(\R^{n+2})$ and the decay properties \eqref{eq:decay_ss}:
\begin{align*}
    \bar{\Omega}(\xi)=O(|\xi|^{-2}),\quad 
    |\nabla\bar{\Omega}(\xi)|=O(|\xi|^{-3}),
    \quad \text{as } |\xi|\to+\infty.
\end{align*}
In this section we introduce three indices $q_a, r$ and $\sigma$ satisfying 
\begin{align}\label{condition_localization}
    \frac{nr}{2n-r}<q_a<\frac{n}{2}<r< \frac{2n}{3},\quad \sigma=\frac{nr}{n-r}.
\end{align}
Note that this choice implies $\sigma>n$.
We employ the notation $q_a$ for the supercritical exponent, since it is just an auxiliary parameter of this section that should not be confused with $q$ of \autoref{thm:non.uniqueness}.\\ 
Choosing $\hat{q}=1$ and $\hat{r}=(2\hat{q})\vee r$, \autoref{thm:existence_ancient_solutions} provides an ancient solution $\Psi$ satisfying
\begin{align*}
        \Psi\in & C((-\infty,T];Y^{q_a,r}_1\cap Y^{q_a,r}_\nabla ),\quad \ |\xi|\Psi\in C((-\infty,T];Y^{q_a,\infty}),\\ 
			\sup_{\tau\leq T }& e^{-\frac{\lambda}{2}\tau} \left(\lVert \Psi(\tau)\rVert_{Y_1^{q_a,r}}+\lVert |\xi|\Psi(\tau)\rVert_{Y_1^{q_a,r}}+\lVert |\xi|\Psi(\tau)\rVert_{Y^{\infty}}\right)<\eps,
\end{align*}
for every $\varepsilon>0$ sufficiently small and $T<0$ sufficiently negative.\\
The two solutions $\tilde{w}_1(t), \tilde{w}_2(t)$ introduced above converge locally in $Y^{q}_{\nabla}$ to
\begin{align*}
    \tilde{w}_0(y)=\frac{\ell_{\alpha}}{|y|^2},
    \quad \text{as } t\to 0.
\end{align*}
However, this profile does not belong to $Y^q_{\nabla}$. To enforce integrability, we localize $\tilde{w}_0$ by means of a smooth modulation. Let $\chi \in C^{\infty}([0,+\infty))$ satisfying $\chi(\rho)=1$ for $\rho \leq 1$, and define
\begin{align*}
    w_0(y):=\chi(|y|)\tilde{w}_0(y),
    \quad
    h_0:=\tilde{w}_0-w_0.
\end{align*}
We require that
\begin{align} \label{eq:assumptions.localization}
    w_0\in Y^q_{\nabla},\quad 
  h_0\in Y^r_{\nabla}\cap Y_1^r,
  \quad
  \mbox{and}
  \quad
  |y|h_0\in Y^{\sigma},  
\end{align}
which is true for instance if $\chi$ is compactly supported. We denote by $w_1$ and $w_2$ the two solutions of \eqref{eq:w} that we aim to construct, and we adopt the ansatz
\begin{align*}
    w_i = \tilde{w}_i - h_i, \qquad i\in\{1,2\}.
\end{align*}
This ansatz forces each $h_i$ to solve
\begin{align}\label{eq_reminder_trunc}
\begin{dcases}
\partial_t h_i=\Delta h_i+4n\bar{w}h_i+2y\cdot\nabla(\bar{w}h_i)+f_i(h_i),
\\
h_i(0)=h_0,
\end{dcases}
\end{align}
where 
\begin{align*}
    \bar{w}(t,y)&:=\frac{1}{t}\bar{\Omega}\left(\frac{y}{\sqrt{t}}\right),
    \\
f_i(h)&:=4n\psi_i h+2y\cdot\nabla(\psi_i h)-2nh^2-y\cdot \nabla (h^2),
\\
\psi_1(t,y)&:=0,
\\
\psi_2(t,y)&:=\frac{1}{t}\Psi\!\left(\log t,\frac{y}{\sqrt{t}}\right).
\end{align*}

The construction of solutions to \eqref{eq_reminder_trunc} is nontrivial due to the singular behavior of $\bar{w}$ at $t=0$. In particular, local well-posedness is tied to the ``strength'' of the instability associated with the eigenvalue of $L_{\alpha}$, see \eqref{Item instabilty} below. From now on, we drop the subscript $\alpha$ to simplify the notation. 
We construct solutions to \eqref{eq_reminder_trunc} via a fixed-point argument. The starting point for this procedure is the following proposition.

\begin{prop}\label{prop_linear_localization}
   Assume \eqref{condition_localization}. 
   Let $\alpha>0$ be such that the corresponding expander $\bar{\Omega}$ yields a linearized operator $L: \mathcal{D}(L) \subset Y^{q_a,r} \rightarrow  Y^{q_a,r}$ with maximal positive eigenvalue $\lambda$ satisfying
		\begin{align}
			\label{Item instabilty} {\lambda}<1-\frac{n}{2r}.  
		\end{align}
		Assume further that $h_0\in Y_1^r\cap Y_{\nabla}^r$, $|y|h_0\in Y^\sigma$, and $f:(0,1]\times \R^{n+2} \rightarrow \R$ is such that
\begin{align*}
        M:=\sup_{t\in (0,1)} \left(\lVert f(t)\rVert_{Y^r}t+\lVert |y|f(t)\rVert_{Y^r}t^{\frac{1}{2}}+\lVert f(t)\rVert_{Y^{q_{a}}}t^{1+\frac{n}{2r}-\frac{n}{2q_a}}+\lVert |y|f(t)\rVert_{Y^{q_{a}}}t^{\frac{1}{2}+\frac{n}{2r}-\frac{n}{2q_a}}\right) &<\infty,
        \\
			\operatorname{lim}_{t\rightarrow 0} \left(\lVert f(t)\rVert_{Y^r}t+\lVert |y|f(t)\rVert_{Y^r}t^{\frac{1}{2}}+\lVert f(t)\rVert_{Y^{q_{a}}}t^{1+\frac{n}{2r}-\frac{n}{2q_a}}+\lVert |y|f(t)\rVert_{Y^{q_{a}}}t^{\frac{1}{2}+\frac{n}{2r}-\frac{n}{2q_a}}\right)&=0,
		\end{align*}
and the following functions, extended by continuity at time $t=0$, satisfy
		\begin{align*}
            tf&\in C([0,1];Y^r), & \hspace{-2cm}t^{\frac{1}{2}}|y|f&\in C([0,1];Y^r),
            \\  
            t^{1+\frac{n}{2r}-\frac{n}{2q_a}} f&\in C([0,1];Y^{q_a}), 
            & \hspace{-2cm} 
            t^{\frac{1}{2}+\frac{n}{2r}-\frac{n}{2q_a}} |y|f&\in C([0,1];Y^{q_a}).
		\end{align*}
        Then there exists a unique $h\in C([0,1],Y_{\nabla}^r)$ solving the Cauchy problem 
		\begin{align}\label{linear pde}
			\begin{dcases}
				\partial_t h=\Delta h+4n\bar{w}h+2y\cdot\nabla(\bar{w}h)+f,\\
h(0)=h_0.
			\end{dcases} 
		\end{align}
		Furthermore, $h$ satisfies
        \begin{align*}
            t^{\frac{1}{2}}h\in C([0,1];{ Y_1^r}),\quad |y|h\in C([0,1],Y^{\sigma}),\quad t^{\frac{n}{2\sigma}}|y|h\in  C([0,1],Y^{\infty}),
        \end{align*}
        and 
		\begin{align}\label{en_est_lin_1}
			\lVert h\rVert_{ L^{\infty}([0,1];Y^r_{\nabla})}+\lVert |y| h\rVert_{ L^{\infty}([0,1];Y^\sigma)}&+\operatorname{sup}_{t\in (0,1)}t^{\frac{1}{2}}\lVert h(t)\rVert_{ Y_1^r} +\operatorname{sup}_{t\in (0,1)}t^{\frac{n}{2\sigma}}\lVert |y| h(t)\rVert_{ Y^{\infty}}
            \\ 
            &\lesssim M+\lVert h_0\rVert_{Y_1^{r}}+\lVert y\cdot\nabla h_0\rVert_{Y^r}+\lVert |y|h_0\rVert_{Y^{\sigma}},
            \\
			\operatorname{lim}_{t\rightarrow 0} t^{\frac{n}{2\sigma}}\lVert |y| h(t)&\rVert_{ Y^{\infty}}+t^{\frac{1}{2}}\lVert h(t)\rVert_{ Y^r_1}=0.
		\end{align}
\end{prop}
\begin{proof}
We look for solutions of the form
\begin{align*}
    h(t,y)=e^{\Delta t}h_0(y)+\phi\left(\log t,\frac{y}{\sqrt{t}}\right).
\end{align*}
Define
\begin{align*}
     f(t,y)=:\frac{1}{t}g\left(\log t,\frac{y}{\sqrt{t}}\right),
     \quad 
     \omega(t):=e^{\Delta t}h_0,
     \quad 
     \omega(t,y)=:b\left(\log t,\frac{y}{\sqrt{t}}\right).
\end{align*}

\emph{Step 1: Preliminary bounds on $\omega$ and $|y|\omega$}. 
By \autoref{lem:heat_semigroup_weighted}, we have 
\begin{align*}
    \omega\in C([0,1];Y_1^{r}),\quad \|\omega(t)\|_{Y_1^{r}}\leq \|h_0\|_{Y_1^{r}}.
\end{align*}
We now estimate $y\cdot\nabla \omega$ in $Y^r$. For this reason, define  $\theta(t)=\iota(y\cdot\nabla \omega(t))$ where the isomorphism $\iota$ has been introduced in \autoref{sec:spectral}. Then $\theta(t)=x\cdot\nabla \iota\omega(t)$ and it solves the equation on $\R^n$
\begin{align*}
    \begin{dcases}
        \partial_t \theta=\Delta \theta-2n\frac{x}{|x|^2}\cdot\nabla\iota \omega, \\
        \theta(0,x)=\iota(y\cdot\nabla h_0)(x).
    \end{dcases}
\end{align*}
Standard heat estimates on $\R^n$, H\"older's inequality on Lorentz spaces and the previous bound on $\|\omega(t)\|_{Y^r_1}$ yield
\begin{align*}
    \|\theta(t)\|_{X^r}&\lesssim \|\iota(y\cdot\nabla h_0)\|_{X^r}+\int_0^t \frac{\|\frac{x}{|x|^2}\cdot\nabla\iota \omega(s)\|_{L^{\frac{nr}{n+r},\infty}}}{(t-s)^{1/2}}ds\\ & \lesssim
     \|\iota(y\cdot\nabla h_0)\|_{X^r}+\int_0^t \frac{\|\iota \omega(s)\|_{W^{1,r}}}{(t-s)^{1/2}}ds
    \\ & \lesssim \|y\cdot\nabla h_0\|_{Y^r}+\int_0^t \frac{\|\omega(s)\|_{Y_1^{r}}}{(t-s)^{1/2}}ds \lesssim \|y\cdot\nabla h_0\|_{Y^r}+ \|h_0\|_{Y_1^{r}},
\end{align*}
and $\theta\in C([0,1];X^r)$. Therefore $y\cdot\nabla\omega\in C([0,1];Y^{r})$ and 
\begin{align}\label{eq:evolution_xgradientheat}
    \|y\cdot\nabla\omega(t)\|_{Y^{r}}
    & \lesssim 
    \|y\cdot\nabla h_0\|_{Y^{r}}+\|h_0\|_{Y_1^{r}}.
\end{align}

Also, $|y| \omega$ solves
\begin{align*}
\begin{dcases}
    \partial_t (|y| \omega)=\Delta (|y| \omega)-2\frac{y}{|y|} \cdot\nabla\omega-\frac{n+1}{|y|}\omega,\\
    |y| \omega(0,y)=|y| h_0(y),
\end{dcases}
\end{align*}
and by ultracontractivity of the heat semigroup and Hardy's inequality, cf. \autoref{lem:heat_semigroup_weighted} and \autoref{weighted_sobolev},
\begin{align}
    \||y|\omega(t)\|_{Y^{\sigma}} \label{eq:evolution_|y|heat.1}
    &\lesssim \||y| h_0\|_{Y^\sigma}+\int_0^t \frac{1}{(t-s)^{\frac{1}{2}}}\|h_0\|_{Y_1^{r}} ds
    \lesssim \||y| h_0\|_{Y^\sigma}+\|h_0\|_{Y_1^{r}},
    \\
    \||y|\omega(t)\|_{Y^{\infty}} \label{eq:evolution_|y|heat.2}
    & \lesssim 
    \frac{\||y| h_0\|_{Y^\sigma}}{t^{\frac{n}{2\sigma}}}+\int_0^t \frac{1}{(t-s)^{\frac{n}{2r}}}\|h_0\|_{Y_1^{r}} ds
    \lesssim \frac{1}{t^{\frac{n}{2\sigma}}}\left(\||y| h_0\|_{Y^\sigma}+\|h_0\|_{Y_1^{r}}\right).
\end{align}

\emph{Step 2: Bounds on $h$ and $|y|h$.}
After these preliminary computations, we can start to analyze the behavior of $\phi$. First note that it solves
\begin{align*}
    \partial_{\tau}\phi=(L-1)\phi+4n\bar{\Omega}b+2\xi\cdot\nabla(\bar{\Omega} b)+g,
    \quad(\tau,y)\in (-\infty,0]\times \R^{n+2}.
\end{align*}
Formally, a solution of the equation above is given by
\begin{align}\label{linear_mild_1}
    \phi(\tau)=\int_{-\infty}^\tau e^{-(\tau-s)}S(\tau-s)\left[4n\bar{\Omega}b+2\xi\cdot\nabla(\bar{\Omega} b)+g\right](s)ds.
\end{align}
In what follows, we show that the integral above, in fact, converges in $Y^{q_a,r}$ and then study further properties of the solution. By assumption,
\begin{align*}
    \|g(\tau)\|_{Y^{q_a,r}}&\leq M e^{-\frac{\tau n}{2r}}.
\end{align*}
Secondly, by scaling arguments we have for each $\sigma\in [1,+\infty]$
\begin{align*}
    \|b(\tau)\|_{Y^{\sigma}}
    =
    e^{-\frac{\tau n}{2\sigma}}\|\omega(e^{\tau})\|_{Y^\sigma},\quad 
    \|\xi \cdot\nabla b(\tau)\|_{Y^{\sigma}}
    =
    e^{-\frac{\tau n}{2\sigma}}\|y\cdot \nabla\omega(e^{\tau})\|_{Y^\sigma}.
\end{align*}
Therefore, by \autoref{lem:heat_semigroup_weighted} and \eqref{eq:evolution_xgradientheat}
\begin{align*}
    \|b(\tau)\|_{Y^{r}}
    \lesssim 
    e^{-\frac{\tau n}{2r}}\|h_0\|_{Y^r},\quad  
    \|\xi \cdot\nabla b(\tau)\|_{Y^{r}}
    \lesssim 
    e^{-\frac{\tau n}{2r}}(\|y\cdot\nabla h_0\|_{Y^{r}}+\|h_0\|_{Y_1^{r}}) .
\end{align*}
The computations above allow to study the forcings in the integral formulation \eqref{linear_mild_1}. On the one hand we have   
\begin{align*}
   \|[4n\bar{\Omega}b+2\xi\cdot\nabla(\bar{\Omega}b)](\tau)\|_{Y^r}
   &\lesssim 
   \|b(\tau)\|_{Y^{r}}+ \|\xi \cdot\nabla b(\tau)\|_{Y^{r}}
   \\ 
   &\lesssim 
   e^{-\frac{\tau n}{2r}}(\|y\cdot\nabla h_0\|_{Y^{r}}+\|h_0\|_{Y_1^{r}}),  
\end{align*}
while H\"older's inequality implies
\begin{align*}
    \|[4n\bar{\Omega}b+2\xi\cdot\nabla(\bar{\Omega}b)](\tau)\|_{Y^{q_a}}
    &\lesssim 
    \left(\|b(\tau)\|_{Y^{r}}+ \|\xi \cdot\nabla b(\tau)\|_{Y^{r}}\right)
    \left(\|\bar{\Omega}\|_{Y^{\frac{rq_a}{r-q_a}}}
    +
    \|\xi\cdot\nabla \bar{\Omega}\|_{Y^{\frac{rq_a}{r-q_a}}}\right)
    \\ 
    &\lesssim 
    e^{-\frac{\tau n}{2r}}(\|y\cdot\nabla h_0\|_{Y^{r}}+\|h_0\|_{Y_1^{r}}),
\end{align*}
where we can control the terms involving $\bar{\Omega}$ using \eqref{eq:decay_ss} and the fact that $\frac{rq_a}{r-q_a} > \frac{n}{2}$ by our assumptions on $q_a$. 
Using smoothing properties of $S$, \autoref{prop:smoothing}, and arguing as in the proof of \autoref{prop:fixed_point_ancient}, we obtain
\begin{align}\label{eq:estimate_loc_linear_1}
    \|\phi(\tau)\|_{Y_1^{q_a,r}}& \lesssim (M+\|y\cdot\nabla h_0\|_{Y^{r}}+\|h_0\|_{Y_1^{r}})e^{-\frac{\tau n}{2r}},\quad\text{for each }\tau\in (-\infty,0 ).
\end{align}
Returning to original variables, this yields
\begin{align*}
    \sup_{t\in (0,1)}\|h(t)-\omega(t)\|_{Y^{r}}
    +
    \sup_{t\in (0,1)}t^{\frac{1}{2}}\|h(t)-\omega(t)\|_{Y_1^{r}}& \lesssim M+\|y\cdot\nabla h_0\|_{Y^{r}}+\|h_0\|_{Y_1^{r}}
    .
\end{align*}
Moreover, by \autoref{lem:heat_semigroup_weighted},
\begin{align}\label{bound_1_linear}
    \sup_{t\in (0,1)}\|h(t)\|_{Y^{r}}
    +\sup_{t\in (0,1)}t^{\frac{1}{2}}\|h(t)\|_{Y_1^{r}}\lesssim M+\|y\cdot\nabla h_0\|_{Y^{r}}+\|h_0\|_{Y_1^{r}}.
\end{align}
On the other hand, one can readily check that $\Phi:=|\xi| \phi$ solves
\begin{align*}
    \partial_\tau \Phi
    &=
    \left(L-\frac{3}{2}\right)\Phi
    -
    2\frac{\xi}{\lvert \xi\rvert}\cdot\nabla \phi
    -
    (n+1)\frac{\phi}{\lvert \xi\rvert}
    -
    2|\xi|\bar{\Omega}\phi+|\xi|\left(4n\bar{\Omega}b+2\xi\cdot\nabla(\bar{\Omega}b)+g\right).
\end{align*}
Therefore, 
\begin{align*}
    \Phi(\tau)
    &=
    -\int_{-\infty}^\tau e^{-\frac{3}{2}(\tau-s)}S(\tau-s)\left[2\frac{\xi}{\lvert \xi\rvert}\cdot\nabla \phi+(n+1)\frac{\phi}{\lvert \xi\rvert}+2|\xi|\bar{\Omega}\phi\right](s) ds
    \\ 
    &\quad+
    \int_{-\infty}^\tau e^{-\frac{3}{2}(\tau-s)}S(\tau-s)\left[|\xi|\left(4n\bar{\Omega}b+2\xi\cdot\nabla(\bar{\Omega}b)+g\right)\right](s)ds,
\end{align*}
if the integral is well-defined in $Y^{q_a,r}$. We are left to study the forcings in the formula above. By \eqref{eq:estimate_loc_linear_1} and Hardy's inequality \autoref{weighted_sobolev} we have
\begin{align*}
    \left\|\left[2\frac{\xi}{\lvert \xi\rvert}\cdot\nabla \phi+(n+1)\frac{\phi}{\lvert \xi\rvert}+2|\xi|\bar{\Omega}\phi\right](\tau)\right\|_{Y^{q_a,r}}&\lesssim (M+\|y\cdot\nabla h_0\|_{Y^{r}}+\|h_0\|_{Y_1^{r}})e^{-\frac{\tau n}{2r}}.
\end{align*}
Moreover, by reasonings analogous to the previous ones
\begin{align*}
    \left\|\left[|\xi|(4n\bar{\Omega}b+2\xi\cdot\nabla(\bar{\Omega}b))\right](\tau)\right\|_{Y^{q_a,r}} 
    & \lesssim 
    \|b(\tau)\|_{Y_{\nabla}^{r}} (\||\xi|\bar{\Omega}\|_{Y^{\frac{rq_a}{r-q_a},\infty}}
    +
    \||\xi|\xi\cdot\nabla \bar{\Omega}\|_{Y^{\frac{rq_a}{r-q_a},\infty}})
    \\ 
    &\lesssim 
    e^{-\frac{\tau n}{2r}}(\|y\cdot\nabla h_0\|_{Y^{r}}+\|h_0\|_{Y_1^{r}}),
\end{align*}
since $\frac{r q_a}{r-q_a}>n.$ Also, by assumptions
\begin{align*}
    \||\xi|g(\tau)\|_{Y^{q_a,r}}&\leq M e^{-\frac{\tau n}{2r}}.
\end{align*}
Therefore, it is easy to check that the integral is well-defined and it holds
\begin{align}\label{eq:estimate_2_linear_eq}
    \|\Phi(\tau)\|_{Y^{\infty}}+\|\Phi(\tau)\|_{Y_1^{q_a,r}}& \lesssim (M+\|y\cdot\nabla h_0\|_{Y^{r}}+\|h_0\|_{Y_1^{r}})e^{-\frac{\tau n}{2r}}.
\end{align}
Together with equation \eqref{eq:estimate_loc_linear_1}, the latter also implies
\begin{align*}
    \|\xi\cdot\nabla \phi(\tau)\|_{Y^r}
    & \lesssim 
    (M+\|y\cdot\nabla h_0\|_{Y^{r}}+\|h_0\|_{Y_1^{r}})e^{-\frac{\tau n}{2r}}.
\end{align*}
Consequently,
\begin{align}\label{bound_2_linear}
    \sup_{t\in (0,1)}\|y\cdot \nabla h(t)\|_{Y^r}
    \lesssim 
    M+\|y\cdot\nabla h_0\|_{Y^{r}}+\|h_0\|_{Y_1^{r}}.
\end{align}
In addition, Sobolev embedding \autoref{weighted_sobolev}, scaling, and \eqref{eq:estimate_2_linear_eq}, yield
\begin{align*}
    \||y|h(t)-|y|\omega(t)\|_{Y^{\sigma}}
    &\lesssim 
    \left\||y|\phi\left(\log t,\frac{y}{\sqrt{t}}\right) \right\|_{Y_1^{r}}
    \lesssim M+\|y\cdot\nabla h_0\|_{Y^{r}}+\|h_0\|_{Y_1^{r}},
    \\
    \||y|h(t)-|y|\omega(t)\|_{Y^{\infty}}
    &\lesssim 
    \left\||y|\phi\left(\log t,\frac{y}{\sqrt{t}}\right)\right\|_{Y^{\infty}}
    \lesssim t^{\frac{1}{2}-\frac{n}{2r}}(M+\|y\cdot\nabla h_0\|_{Y^{r}}+\|h_0\|_{Y_1^{r}}).
\end{align*}

Combining the two above with \eqref{eq:evolution_|y|heat.1} and \eqref{eq:evolution_|y|heat.2} we obtain
\begin{align}\label{bound_3_linear}
    \sup_{t\in (0,1)} \||y|h(t)\|_{Y^{\sigma}}
    +
    \sup_{t\in (0,1)}t^{\frac{n}{2\sigma}} \||y|h(t)\|_{Y^{\infty}}
    &\lesssim M+\|y\cdot\nabla h_0\|_{Y^{r}}+\||y| h_0\|_{Y^\sigma}+\|h_0\|_{Y_1^{r}},
\end{align}
and relation \eqref{en_est_lin_1} follows by \eqref{bound_1_linear}, \eqref{bound_2_linear}, and \eqref{bound_3_linear}.

\emph{Step 3: Existence.}
Let us temporarily assume the additional regularity hypotheses $h_0\in C^{\infty}_c(\R^{n+2})$ and $f\in C^{\infty}_c((0,1]\times\R^{n+2})$. 
In this case, we have
\begin{align*}
    \sup_{t\in (0,1)}\|f(t)\|_{Y^{q_a,r}}+\sup_{t\in (0,1)}\||y|f(t)\|_{Y^{q_a,r}}&\lesssim 1,
\end{align*}
and therefore 
\begin{align} \label{eq:g.|xi|g}
     \|g(\tau)\|_{Y^{q_a,r}}\lesssim e^{\tau(1-\frac{n}{2q_a})},\ \||\xi|g(\tau)\|_{Y^{q_a,r}}&\lesssim e^{\frac{\tau}{2}\left(1-\frac{n}{q_a}\right)}.
\end{align}
Moreover, since $q_a>\frac{nr}{n+r}$ by our choice of parameters, we have 
\begin{align}
 1-\frac{n}{2q_a}>-\frac{n}{2r},
 \quad
 1-\frac{n}{q_a}>-\frac{n}{r}.
\end{align}

For the other term in \eqref{linear_mild_1} we can apply H\"older obtaining for some $\tilde{r}$ slightly larger than $r$ 
\begin{align*}
     \|[4n\bar{\Omega}b+2\xi\cdot\nabla(\bar{\Omega}b)](\tau)\|_{Y^{q_a,r}}&\lesssim e^{-\frac{\tau n}{2\tilde{r}}},\\
     \left\|\left[|\xi|(4n\bar{\Omega}b+2\xi\cdot\nabla(\bar{\Omega}b))\right](\tau)\right\|_{Y^{q_a,r}}&\lesssim e^{-\frac{\tau n}{2\tilde{r}}}.
\end{align*}
Repeating the estimates in Step 2 with an improved exponent
\begin{align*}
    \vartheta :=-\frac{ n}{2\tilde{r}}\wedge \frac{1}{2}\left(1-\frac{n}{q_a}\right)>-\frac{n}{2r},
\end{align*}
we obtain
\begin{align*}
\|\phi(\tau)\|_{Y_1^{q_a,r}}+\|\Phi(\tau)\|_{Y_1^{q_a,r}}+\|\Phi(\tau)\|_{Y^{\infty}}+\|\xi\cdot\nabla \phi(\tau)\|_{Y^r}& \lesssim e^{\vartheta \tau}.
\end{align*}
Therefore, it holds
\begin{align*}
    \|h(t)-\omega(t)\|_{Y^r}
    +
    \||y|h(t)-|y|\omega(t)\|_{Y^{\sigma}}
    +
    \|y\cdot\nabla h(t)-y\cdot\nabla\omega(t)\|_{Y^r}
    &\lesssim t^{\frac{n}{2r}+\vartheta}\rightarrow 0,
    \quad\text{as }t\rightarrow 0,
\end{align*}
Similarly, due to the smoothness of the initial data, one has:
\begin{align*}
    t^{\frac{1}{2}}\lVert h(t)\rVert_{ Y_1^r}&\lesssim t^{\frac{n}{2r}+\vartheta}+t^{\frac{1}{2}} \lVert \omega(t)\rVert_{ Y_1^r} \rightarrow 0, &\hspace{-2cm}\text{as }t\rightarrow 0,
    \\
    t^{\frac{n}{2\sigma}}\lVert |y| h(t)\rVert_{ Y^{\infty}}& \lesssim  t^{\frac{1}{2}+\frac{n}{2\sigma}+\vartheta}+t^{\frac{n}{2\sigma}}\||y|\omega(t)\|_{Y^{\infty}}\rightarrow 0, 
    &\hspace{-2cm}\text{as }t\rightarrow 0.
\end{align*}
The lines above guarantee the desired continuity at time $t=0$ under the additional assumptions $h_0\in C^{\infty}_c(\R^{n+2})$ and $f\in C^{\infty}_c((0,1]\times\R^{n+2})$. In addition, continuity for positive times follows from the fact that the terms $f$ and $\bar{w}$ are no longer singular for $t>0$. The general case follows by the uniform bounds proved in Step 2 and approximation.

\emph{Step 4: Uniqueness.} By linearity, we can suppose $h_0=f=0$ and want to show that 
\begin{align*}
    h(t)=\int_0^t e^{(t-s)\Delta}[4n\bar{w}h+2y\cdot\nabla(\bar{w}h)](s)ds
\end{align*}
is identically zero.
By H\"older's inequality and \autoref{lem:heat_semigroup_weighted} we have
\begin{align*}
    \|h(t)\|_{Y^{q_a}}
    &\lesssim 
    \int_0^t \left(\|h(s)\|_{Y^r}
    +
    \|y\cdot\nabla h(s)\|_{Y^r}\right) 
    \left(\|\bar{w}(s)\|_{Y^{\frac{rq_a}{r-q_a}}}
    +
    \|y\cdot\nabla \bar{w}(s)\|_{Y^{\frac{rq_a}{r-q_a}}}\right) ds.
\end{align*}
Since 
\begin{align*}
    \sup_{t\in (0,1)}\|h(t)\|_{Y^r}+\|y\cdot\nabla h(t)\|_{Y^r}\lesssim 1
\end{align*}
and 
\begin{align*}
    \|\bar{w}(s)\|_{Y_{\nabla}^{\frac{rq_a}{r-q_a}}}
    =
    s^{\frac{n(r-q_a)}{2rq_a}-1}\|\bar{\Omega}\|_{Y_{\nabla}^{\frac{rq_a}{r-q_a}}}
    \lesssim 
    s^{\frac{n(r-q_a)}{2rq_a}-1},
\end{align*}
with exponent $\frac{n(r-q_a)}{2rq_a}-1 > -1$, we obtain
\begin{align}\label{apriori_estimate_linear_sol}
     \|h(t)\|_{Y^{q_a}}&\lesssim t^{\frac{n(r-q_a)}{2rq_a}}.
\end{align}
Recall now that, since $h_0=0$, we have $h(t,y)=\phi\left(\log t,\frac{y}{\sqrt{t}}\right)$ as well as $b=0$. In addition, since $f=0$, we also have $g=0$ and thus from the mild formulation \eqref{linear_mild_1} we get  
\begin{align*}
    \phi(\tau)=e^{-(\tau-s)}S(\tau-s)\phi(s),
    \quad
    \mbox{for }-\infty<s<\tau<0.
\end{align*}
Combining the formula above, \eqref{apriori_estimate_linear_sol}, the uniform bound on $\|h(t)\|_{Y^r}$ and \autoref{prop:smoothing}, we finally obtain
\begin{align*}
    \|\phi(\tau)\|_{Y^{q_a,r}}
    &\lesssim e^{(\lambda+\delta-1)(\tau-s)}\|\phi(s)\|_{Y^{q_a,r}}
    \lesssim_{\tau} e^{-(\lambda+\delta-1+\frac{n}{2r})s}\rightarrow 0, \quad\text{as }s\rightarrow -\infty
\end{align*}
for some $0<\delta\ll 1$ sufficiently small that $\lambda+\delta-1+\frac{n}{2r}<0$, which exists due to condition \eqref{Item instabilty}. Thus $\phi\equiv 0$ and consequently also $h$ is identically zero, concluding the proof.
\end{proof}
\begin{rmk}
    One can readily verify that the proof of \autoref{prop_linear_localization} extends to the wider range of parameters $\frac{nr}{n+r}<q_a<\frac{n}{2}<r<n$. However, in the main result of this section, namely \autoref{thm:nonlinear_localization}, we will require the more restrictive condition \eqref{condition_localization}. Given the auxiliary role of \autoref{prop_linear_localization}, we have stated it in a slightly suboptimal form in order to avoid introducing additional parameter regimes.
\end{rmk}

Let us now fix $T'\in (0,1)$.
In order to run the fixed point argument we introduce the Banach space 
\begin{align*}
    Z^{T'}:=\bigg
    \{h &\in C([0,T'];Y_{\nabla}^r) \,:\, \\
    \quad& |y|h\in C([0,T'];Y^{\sigma}), \quad t^{\frac{n}{2\sigma}} |y| h\in C([0,T'];Y^{\infty}),\quad t^{\frac{1}{2}} h\in C([0,T'];Y^{r}_1)
    \\ 
    & 
    \lim_{t\rightarrow 0}t^{\frac{1}{2}}\lVert h(t)\rVert_{ Y_1^{r}}=  \lim_{t\rightarrow 0}t^{\frac{n}{2\sigma}}\lVert |y| h(t)\rVert_{ Y^{\infty}} =0 \qquad\bigg\},
\end{align*}
endowed with the natural norm 
\begin{align*}
    \|h(t)\|_{Z^{T'}}:=&\sup_{t\in (0,T')}\|h(t)\|_{Y_{\nabla}^r}+\sup_{t\in (0,T')}\||y| h(t)\|_{Y^\sigma} +\sup_{t\in (0,T')}t^{\frac{1}{2}}\| h(t)\|_{Y_1^{r}}+\sup_{t\in (0,T')}t^{\frac{n}{2\sigma}}\||y| h(t)\|_{Y^{\infty}}.
\end{align*}
The main result of this section reads as follows.
\begin{theorem}
   \label{thm:nonlinear_localization}
    Assume \eqref{condition_localization}. Let $\alpha>0$ be such that the corresponding expander $\bar{\Omega}$ yields a linearized operator $L: \mathcal{D}(L) \subset Y^{q_a,r} \rightarrow  Y^{q_a,r}$ with maximal positive eigenvalue $\lambda$ satisfying \eqref{Item instabilty}.
		Assume further that \begin{align*}
		    h_0\in Y_1^{r}\cap Y_{\nabla}^r,\quad |y|h_0\in Y^\sigma,
		\end{align*} 
        and $\psi(t,y)=\frac{1}{t}\Psi(\log t,\frac{y}{\sqrt{t}})$, where
        \begin{align*}
            \sup_{\tau\leq T }e^{-\kappa \tau}\left(\lVert \Psi(\tau)\rVert_{Y_1^{q_a,r}}+\lVert |y|\Psi(\tau)\rVert_{Y_1^{q_a,r}}+\lVert |y|\Psi(\tau)\rVert_{Y^{\infty}}\right)=:\epsilon<+\infty
        \end{align*}
        for some $\kappa>0$ and $T<0$. 
        Whenever the above displayed quantity is sufficiently small, there is $0<T'<e^T$ such that there exists a small ball in $Z^{T'}$ and a unique solution $h$ of \eqref{eq_reminder_trunc} therein. Moreover
        \begin{align}\label{estimate_trunc}
            \|h\|_{Z^{T'}}& \lesssim  \lVert h_0\rVert_{Y_1^{r}}+\lVert y\cdot\nabla h_0\rVert_{Y^r}+\lVert |y|  h_0\rVert_{Y^{\sigma}}.
        \end{align}
\end{theorem}
\begin{proof}
    We look for solutions of \eqref{eq_reminder_trunc} in the form
    \begin{align*}
        h(t)=\mathcal{T}[h_0,0](t)+\mathcal{T}[0,f(h)](t),
    \end{align*}
    where $\mathcal{T}[g,f]$ denotes the unique solution to \eqref{linear pde} with initial condition $g$ and singular forcing term $f$, as provided by \autoref{prop_linear_localization}. We begin by collecting bounds for the nonlinear terms. 
    
    Using H\"older’s inequality, interpolation, and Sobolev embedding \autoref{weighted_sobolev}, we can control the product $\psi h$ as follows:    
    \begin{align}
      \|\psi(t) h(t)\|_{Y^{q_a}}&\leq \|\psi(t)\|_{Y^{2q_a}}\|h(t)\|_{Y^{2q_a}}\\ & \leq \|\psi(t)\|_{Y^{r}}^{1-\frac{n}{r}+\frac{n}{2q_a}}\|\psi(t)\|_{Y^{\sigma}}^{\frac{n}{r}-\frac{n}{2q_a}}\|h(t)\|_{Y^{r}}^{1-\frac{n}{r}+\frac{n}{2q_a}}\|h(t)\|_{Y^{\sigma}}^{\frac{n}{r}-\frac{n}{2q_a}}\\ &  \lesssim \epsilon t^{\frac{n}{4q_a}-1} \|h(t)\|_{Y^r}^{1-\frac{n}{r}+\frac{n}{2q_a}}\|h(t)\|_{Y_1^{r}}^{\frac{n}{r}-\frac{n}{2q_a}},\label{auxiliary_estimates_fixed_point1}
      \\
       \|\psi(t) h(t)\|_{Y^{r}} \label{auxiliary_estimates_fixed_point2}
       &\lesssim 
       \|\psi(t)\|_{Y_1^{r}}\|h(t)\|_{Y^{n}}
       \lesssim 
       \epsilon t^{\frac{n}{2r}-\frac{3}{2}}\|h(t)\|_{Y^{r}}^{2-\frac{n}{r}}\|h(t)\|_{Y_1^{r}}^{\frac{n}{r}-1} , 
       \\
       \||y|\psi(t) h(t)\|_{Y^{q_a}}&\leq  \|\psi(t)\|_{Y^{q_a}}\||y|h(t)\|_{Y^\infty}\leq \epsilon t^{\frac{n}{2q_a}-1}\||y|h(t)\|_{Y^\infty} ,\label{auxiliary_estimates_fixed_point3}\\
       \||y|\psi(t) h(t)\|_{Y^{r}} \label{auxiliary_estimates_fixed_point4}&\leq \|\psi(t)\|_{Y^r}\||y|h(t)\|_{Y^\infty}\leq \epsilon t^{\frac{n}{2r}-1}\||y|h(t)\|_{Y^\infty} .
       \end{align}       
      Similarly, for the gradient term $y\cdot\nabla (\psi h) = y\cdot\nabla \psi h + y\cdot\nabla h \psi$, we have: 
       \begin{align}
        \|y\cdot\nabla\psi(t) h(t)\|_{Y^{q_a}}
        &\lesssim 
    \|\nabla\psi(t)\|_{Y^{q_a}}\||y|h(t)\|_{Y^{\infty}}\lesssim \epsilon t^{\frac{n}{2q_a}-\frac{3}{2}}\||y|h(t)\|_{Y^{\infty}}\label{auxiliary_estimates_fixed_point5},\\
       \|y\cdot\nabla\psi(t) h(t)\|_{Y^{r}}&\lesssim \|\nabla\psi(t)\|_{Y^{r}}\||y|h(t)\|_{Y^{\infty}}\lesssim \epsilon t^{\frac{n}{2r}-\frac{3}{2}}\||y|h(t)\|_{Y^{\infty}} \label{auxiliary_estimates_fixed_point6},\\
       \||y|y\cdot\nabla\psi(t) h(t)\|_{Y^{q_a}}&\lesssim \|y\cdot\nabla\psi(t)\|_{Y^{q_a}}\||y|h(t)\|_{Y^{\infty}}\lesssim \epsilon t^{\frac{n}{2q_a}-1}\||y|h(t)\|_{Y^{\infty}}\label{auxiliary_estimates_fixed_point7},\\
       \||y|y\cdot\nabla\psi(t) h(t)\|_{Y^{r}}& \lesssim \|y\cdot\nabla\psi(t)\|_{Y^{r}}\||y|h(t)\|_{Y^{\infty}}\lesssim \epsilon t^{\frac{n}{2r}-1}\||y|h(t)\|_{Y^{\infty}} \label{auxiliary_estimates_fixed_point8},\\
       \|y\cdot\nabla h(t)\psi(t) \|_{Y^{q_a}}&\leq \||y|\psi(t)\|_{Y^{\frac{r q_a}{r-q_a}}}\|\nabla h(t)\|_{Y^r}\lesssim \||y|\psi(t)\|^{\frac{r-q_a}{r}}_{Y^{q_a}}\||y|\psi(t)\|^{\frac{q_a}{r}}_{Y^{\infty}}\|\nabla h(t)\|_{Y^r}\notag\\ & \lesssim \epsilon t^{\frac{1}{2}\left(-1-\frac{n}{r}+\frac{n}{q_a}\right)}\|\nabla h(t)\|_{Y^r}\label{auxiliary_estimates_fixed_point9},\\
       \|y\cdot\nabla h(t) \psi(t) \|_{Y^{r}}&\lesssim \||y|\psi(t)\|_{Y^\infty}\|\nabla h(t)\|_{Y^r} \lesssim \epsilon t^{-1/2}\|\nabla h(t)\|_{Y^r}  \label{auxiliary_estimates_fixed_point10},\\
       \||y|y\cdot\nabla h(t)\psi(t) \|_{Y^{q_a}}&\leq \||y|\psi(t)\|_{Y^{\frac{r q_a}{r-q_a}}}\| y\cdot \nabla h(t)\|_{Y^r}\lesssim \epsilon t^{\kappa+\frac{1}{2}\left(-1-\frac{n}{r}+\frac{n}{q_a}\right)}\|y\cdot\nabla h(t)\|_{Y^r}\label{auxiliary_estimates_fixed_point11},\\
       \||y|y\cdot\nabla h(t)\psi(t) \|_{Y^{r}}&\lesssim \||y|\psi(t)\|_{Y^{\infty}}\|y\cdot\nabla h(t)\|_{Y^{r}}\lesssim \epsilon t^{\kappa-\frac{1}{2}} \|y\cdot\nabla h(t)\|_{Y^{r}} \label{auxiliary_estimates_fixed_point12}.
    \end{align}  
    Lastly, for the purely nonlinear contributions we have by H\"older's inequality, interpolation and Sobolev embedding \autoref{weighted_sobolev}:
     \begin{align}
\|h^2(t)\|_{Y^{q_a}}
&\lesssim 
\|h(t)\|^2_{Y^{2q_a}}\lesssim \|h(t)\|_{Y^r}^{2-\frac{2n}{r}+\frac{n}{q_a}}\|h(t)\|_{Y_1^{r}}^{\frac{2n}{r}-\frac{n}{q_a}}\label{auxiliary_estimates_fixed_point13},
\\
\|h^2(t)\|_{Y^{r}}
&\lesssim  
\|h(t)\|_{Y_1^{r}}\|h(t)\|_{Y^n}\lesssim \|h(t)\|_{Y^{r}}^{2-\frac{n}{r}}\|h(t)\|_{Y_1^{r}}^{\frac{n}{r}}  \label{auxiliary_estimates_fixed_point14},
\\
\||y|h^2(t)\|_{Y^{q_a}}
&\leq 
\|h(t)\|_{Y^r} \||y|h(t)\|_{Y^{\frac{q_ar}{r-q_a}}}  
\leq 
\|h(t)\|_{Y^r}\||y|h(t)\|^{\frac{n(r-q_a)}{q_a(n-r)}}_{Y^\sigma} \||y|h(t)\|^{1-\frac{n(r-q_a)}{q_a(n-r)}}_{Y^\infty} \label{auxiliary_estimates_fixed_point15},
\\
\||y|h(t)^2\|_{Y^{r}}
& \leq  
\|h(t)\|_{Y^{r}}\||y|h(t)\|_{Y^\infty},\label{auxiliary_estimates_fixed_point16}
\\ 
\|y\cdot\nabla h(t) h(t)\|_{Y^{q_a}}
&\leq 
\|\nabla h(t)\|_{Y^r} \||y|h(t)\|_{Y^{\frac{q_ar}{r-q_a}}}  
\leq 
\|\nabla h(t)\|_{Y^r}\||y|h(t)\|^{\frac{n(r-q_a)}{q_a(n-r)}}_{Y^\sigma} \||y|h(t)\|^{1-\frac{n(r-q_a)}{q_a(n-r)}}_{L^\infty},\label{auxiliary_estimates_fixed_point17}
\\
\|y\cdot\nabla h(t) h(t) \|_{Y^{r}}
&\leq  
\|\nabla h(t)\|_{Y^r}\||y|h(t)\|_{Y^{\infty}},\label{auxiliary_estimates_fixed_point18}
\\
\||y|y\cdot\nabla h(t) h(t) \|_{Y^{q_a}}
&\leq 
\|y\cdot\nabla h(t)\|_{Y^r} \||y|h(t)\|_{Y^{\frac{q_ar}{r-q_a}}} 
\leq 
\|y\cdot\nabla h(t)\|_{Y^r}\||y|h(t)\|^{\frac{n(r-q_a)}{q_a(n-r)}}_{Y^\sigma} \||y|h(t)\|^{1-\frac{n(r-q_a)}{q_a(n-r)}}_{Y^\infty},\label{auxiliary_estimates_fixed_point19}
\\ 
\||y|y\cdot\nabla h(t) h(t)\|_{Y^{r}}
&\leq 
\|y\cdot\nabla h(t)\|_{Y^r}\||y|h(t)\|_{Y^{\infty}} \label{auxiliary_estimates_fixed_point20}.
\end{align}
       After these preliminary computations, we can run a fixed point argument. Let $M:=\lVert \mathcal{T}[h_0,0]\rVert_{Z^{e^{T}}}$, and for $T'>0$ we denote by $B_{2M} \subset Z^{T'}$ the closed ball in $Z^{T'}$ with center $0$ and radius $2M$. 
       We are looking for $\epsilon>0$ and $T'>0$ small enough such that the map
		\begin{align*}
			\Gamma({h}):=\mathcal{T}[h_0,0]+\mathcal{T}[0,f(h)]
		\end{align*}
		is a contraction on $B_{2M}$. First we need to show that $\Gamma$ maps $B_{2M}$ into itself. Using the estimates above and the definition of the norm in $Z^{T'}$, we obtain
        for $h\in B_{2M}$:
		\begin{align*}
			t^{1+\frac{n}{2r}-\frac{n}{2{q_a}}}\lVert f(h(t))\rVert_{Y^{q_a}}&\lesssim \epsilon M+M^2(T')^{1-\frac{n}{2r}},\\
			t\lVert f(h(t))\rVert_{Y^r}& \lesssim \epsilon M+M^2(T')^{1-\frac{n}{2r}},\\
             t^{\frac{1}{2}+\frac{n}{2r}-\frac{n}{2{q_a}}}\lVert |y|f(h(t))\rVert_{Y^{q_a}}&\lesssim \epsilon M+M^2(T')^{1-\frac{n}{2r}},\\
			t^{\frac{1}{2}}\lVert|y| f(h(t))\rVert_{Y^r}& \lesssim \epsilon M+M^2(T')^{1-\frac{n}{2r}},
            \\
			\operatorname{lim}_{t\rightarrow 0} t^{1+\frac{n}{2r}-\frac{n}{2{q_a}}}\lVert f(h(t))\rVert_{Y^{q_a}}&=0,\quad \operatorname{lim}_{t\rightarrow 0} t\lVert f(h(t))\rVert_{Y^r}=0,\\
            \operatorname{lim}_{t\rightarrow 0} t^{\frac{1}{2}+\frac{n}{2r}-\frac{n}{2{q_a}}}\lVert |y|f(h(t))\rVert_{Y^{q_a}}&=0,\quad \operatorname{lim}_{t\rightarrow 0} t^{\frac{1}{2}}\lVert |y|f(h(t))\rVert_{Y^r}=0,
		\end{align*}
        as well as the required continuity of the forcing for positive times.
Hence, \autoref{prop_linear_localization} applies and yields, for some finite constant $C$:
		\begin{align*}
			\lVert  \Gamma(h)\rVert_{Z^{T'}}&\leq M +C\left(\eps M+\left(T'\right)^{1-\frac{n}{2r}}M^2\right).
		\end{align*}
        Choosing $\eps$ and $T'$ sufficiently small, we obtain 
\begin{align*}
    \Gamma(B_{2M})\subset B_{2M}.
\end{align*}
		Now we show that $\Gamma$ is a contraction in $B_{2M}$, possibly reducing $\eps$ and $T'$. For $h_1,\ h_2\in B_{2M}$, we observe that
		\begin{align*}
			\Gamma(h_1)-\Gamma(h_2)=\mathcal{T}[0,f(h_1)-f(h_2)].
		\end{align*} 
		The estimates above imply, thanks to the fact that $h_1,h_2\in B_{2M}$
		\begin{align*}
			t^{1+\frac{n}{2r}-\frac{n}{2{q_a}}}\lVert f(h_1(t))-f(h_2(t))\rVert_{Y^{q_a}}
            &\lesssim
            \left(\epsilon +\left(T'\right)^{1-\frac{n}{2r}}M \right) \| h_1-h_2 \|_{Z^{T'}},\\
			t\lVert f(h_1(t))-f(h_2(t))\rVert_{Y^r}
            & \lesssim
            \left(\epsilon +\left(T'\right)^{1-\frac{n}{2r}}M \right) \| h_1-h_2 \|_{Z^{T'}},\\
            t^{\frac{1}{2}+\frac{n}{2r}-\frac{n}{2{q_a}}}\lVert|y| \left(f(h_1(t))-f(h_2(t))\right)\rVert_{Y^{q_a}}
            &\lesssim
            \left(\epsilon +\left(T'\right)^{1-\frac{n}{2r}}M \right) \| h_1-h_2 \|_{Z^{T'}},\\
			t^{\frac{1}{2}}\lVert |y|\left(f(h_1(t))-f(h_2(t))\right)\rVert_{Y^r}
            & \lesssim
           \left(\epsilon +\left(T'\right)^{1-\frac{n}{2r}}M \right) \| h_1-h_2 \|_{Z^{T'}},\\
			\operatorname{lim}_{t\rightarrow 0} t^{1+\frac{n}{2r}-\frac{n}{2{q_a}}}\lVert f(h_1(t))-f(h_2(t))\rVert_{Y^{q_a}}&=0,\\ \operatorname{lim}_{t\rightarrow 0} t\lVert f(h_1(t))-f(h_2(t))\rVert_{Y^r}&=0,
            \\
            \operatorname{lim}_{t\rightarrow 0} t^{\frac{1}{2}+\frac{n}{2r}-\frac{n}{2{q_a}}}\lVert |y|\left(f(h_1(t))-f(h_2(t))\right)\rVert_{Y^{q_a}}&=0,
            \\ 
            \operatorname{lim}_{t\rightarrow 0} t^{\frac{1}{2}}\lVert |y|\left(f(h_1(t))-f(h_2(t))\right)\rVert_{Y^r}&=0.
		\end{align*}
		Applying again \autoref{prop_linear_localization}, we conclude
		\begin{align*}
			\lVert  \Gamma(h_1)-\Gamma(h_2)\rVert_{Z^{T'}}
            &\leq 
            C'\left(\epsilon+\left(T'\right)^{1-\frac{n}{2r}}M\right)\lVert h_1-h_2\rVert_{Z^{T'}},
		\end{align*}
		for some finite constant $C'$; up to choosing $\epsilon$ and $T'$ sufficiently small, the coefficient in front of $\lVert h_1-h_2\rVert_{Z^{T'}}$ is smaller than $1$, hence $\Gamma$ is a contraction. By the Banach fixed point theorem, $\Gamma$ admits a unique fixed point in $B_{2M}$, which provides the desired solution.
\end{proof}
\subsection{Proof of \autoref{thm:non.uniqueness}}\label{subsec:end_proof}
Let $1\leq q<\frac{n}{2}$ and $\beta$ such that
\begin{align*}
\left(\frac{n}{2q}-2\right)\vee \left(-\frac{2q}{q+n}\right)<\beta<\frac{2n}{3q}-2.
\end{align*}
Then set $r:=(2+\beta)q\in \left(\frac{n}{2},\frac{2n}{3}\right)$. By \autoref{thm:instability}, for each $\eta\in [1,\frac{n}{2})$ and $\gamma>\frac{n}{2}$ there is $\bar{\alpha}>0$ such that for the corresponding expander $\bar{\Omega}$ the operator $L_{\bar{\alpha}}:D(L_{\bar{\alpha}})\subset Y^{\eta,\gamma}\rightarrow Y^{\eta,\gamma}$ admits a maximal positive eigenvalue $\lambda_{\bar{\alpha}}$ for which 
\begin{align*}
    \lambda_{\bar{\alpha}}<1-\frac{n}{2r}.
\end{align*}
Moreover, $\lambda_{\bar{\alpha}}$ is independent of the particular choice of $\eta,
\gamma$. In particular we can choose both $(\eta,\gamma)=(\hat{q},\hat{r})$ and $(\eta,\gamma)=(q_a,r)$ for
\begin{align*}
    \hat{q}=1,\quad  \hat{r}=r\vee(2\hat{q}),\quad q_a\in \left(\frac{nr}{2n-r},\frac{n}{2}\right).
\end{align*}
Recall $\tilde{w}_0(y)=\frac{\ell_{\bar{\alpha}}}{|y|^2}$. 
Define $w_0, h_0: \R^{n+2}\rightarrow \R$ as above, namely:
\begin{align*}
    w_0(y):=\tilde{w}_0(y)\chi\left(|y|\right),
    \quad
    h_0:=\tilde{w}_0-w_0,
\end{align*}
for some cutoff function $\chi$ such that \eqref{eq:assumptions.localization} holds.

We can invoke \autoref{thm:existence_ancient_solutions} and \autoref{thm:nonlinear_localization}
to obtain 
\begin{align*}
   w_1(t,y)&=\frac{1}{t}\bar{\Omega}\left(\frac{y}{\sqrt{t}}\right)-h_1(t,y),\\
w_2(t,y)&=\frac{1}{t}\bar{\Omega}\left(\frac{y}{\sqrt{t}}\right)+\frac{1}{t}\Psi\left(\log t, \frac{y}{\sqrt{t}}\right)-h_2(t,y).
\end{align*}
Inverting the change of variables by \autoref{lem:isomorphism}, \autoref{rmk:isomorphism} we obtain the initial datum
\begin{align*}
    c_0:=A^{-1}[w_0]\in L^q_{rad}(\R^n),\quad c_0\neq 0,
\end{align*}
as well as the two functions 
\begin{align*}
  c_1(t)=A^{-1}[w_1(t)],\quad c_2(t)=A^{-1}[w_2(t)]. 
\end{align*}

One can easily check from the regularity of the self-similar profile $\bar{\Omega}, $ cf. \autoref{subsec:selfsimilarprofile}, \autoref{thm:existence_ancient_solutions}, \autoref{thm:nonlinear_localization} and the continuity of $A^{-1}$, cf. \autoref{lem:isomorphism}, \autoref{rmk:isomorphism}, that 
\begin{align}\label{properties_inversion_localization_1}
    \frac{1}{t}A^{-1}\bar{\Omega}\left(\frac{\cdot}{\sqrt{t}}\right)
    &\in  C((0,T'];L^r_{rad}(\R^n)),
    \quad 
    \frac{1}{t}\left\|A^{-1}\bar{\Omega}\left(\frac{\cdot}{\sqrt{t}}\right)\right\|_{L^r_{rad}}\lesssim\ \frac{1}{t^{1-\frac{n}{2r}}},\\ \ A^{-1}h_1(t),\ A^{-1}h_2(t)
    &\in
    C([0,T'];L^r_{rad}(\R^n)),\label{properties_inversion_localization_2}\\ 
     \frac{1}{t}A^{-1}{\Psi}\left(\log t,\frac{\cdot}{\sqrt{t}}\right)
     &\in
     C((0,T'];L^r_{rad}(\R^n))\cap C([0,T'];L^q_{rad}(\R^n)),\label{properties_inversion_localization_3}
     \\
    \frac{1}{t}\left\|A^{-1}{\Psi}\left(\log t,\frac{\cdot}{\sqrt{t}}\right)\right\|_{L^q_{rad}}& \lesssim \frac{1}{t^{1-\frac{n}{2q}}},
    \quad 
    \frac{1}{t}\left\|A^{-1}{\Psi}\left(\log t,\frac{\cdot}{\sqrt{t}}\right)\right\|_{L^r_{rad}} \lesssim \frac{1}{t^{1-\frac{n}{2r}}}\label{properties_inversion_localization_4}.
\end{align}
By the same arguments as in \autoref{ssec:proof.critical}, \eqref{eq:regularity.nonlinearity} holds and $c_1, c_2$ are distributional solutions of equation \eqref{eq:KS_physical}. We claim that they are, in fact, two different mild $L^q$-solutions on $[0,T']$ with the same initial datum $c_0$. For that, we have to show that $c_1, c_2\in C([0,T'];L^q(\R^n))$, they satisfy the mild formulation \eqref{eq:mild.formulation}, and $c_1\neq c_2$. We show the first two properties in the form of a lemma.
\begin{lemma}\label{lem:localization}
    $c_1,c_2\in C([0,T'];L^q(\R^n))$ and they both satisfy the mild formulation \eqref{eq:mild.formulation}.
\end{lemma}
\begin{proof}
    We just prove the claim for $c_2$, the other case being simpler. For simplicity we also drop the subscripts on $c_2, h_2$ since no ambiguity arises. 
    
    Fix a monotone function $\chi_0\in C^{\infty}_c([0,+\infty))$ such that $\chi_0\geq 0$ and
\begin{align*}
    \chi_0(\rho)=\begin{dcases}
        1\quad \text{if }\rho\leq 1,\\
        0\quad \text{if }\rho>2.
    \end{dcases}
\end{align*}
    For $x_0\in \R^n$ and $R>0$, define the cutoff 
    \begin{align*}
        \chi_{x_0,R}(x):=\chi_0\left(|x-x_0|/R\right).
    \end{align*}
    Using the properties of the self-similar profile $\bar{\Omega}$ (see \autoref{subsec:selfsimilarprofile}), the mapping properties of $A^{-1}$ (cf. \autoref{lem:isomorphism}, \autoref{rmk:isomorphism}), and \eqref{properties_inversion_localization_1}, \eqref{properties_inversion_localization_2}, \eqref{properties_inversion_localization_3}, \eqref{properties_inversion_localization_4}, we obtain
\begin{align}\label{eq:continuity_localization}
        c\chi_{x_0,R}\in C([0,T'];L^q(\R^n))\cap C((0,T'];L^r(\R^n))
    \end{align}
    and, for each $0<t_0<t\leq T'$, it holds
    \begin{align}\label{mild_eq_1}
     \chi_{x_0,R}c(t)
     &=
     e^{\Delta(t-t_0)}\chi_{x_0,R}c(t_0)+\int_{t_0}^t e^{\Delta(t-s)}[ \Delta\chi_{x_0,R}c(s)-2\div (\nabla \chi_{x_0,R}c(s))]ds\notag\\ & -\int_{t_0}^t e^{\Delta(t-s)}[\operatorname{div}\left(\chi_{x_0,R}c(s)\nabla(-\Delta)^{-1}c(s)\right)-\nabla\chi_{x_0,R}c(s)\cdot\nabla(-\Delta)^{-1}c(s)]ds  .  
    \end{align}
    By the regularization properties of the heat semigroup and the definition of $\chi$ we have

\begin{align*}
    \|c(t)\|_{L^q(B_R(x_0))}&\lesssim \|\chi_{x_0,R}c(t_0)\|_{L^q}+\int_{t_0}^t\left(1+\frac{1}{\sqrt{t-s}}\right)  \|c(s)\|_{L^q(B_{2R}(x_0))}ds\\ & +\int_{t_0}^t\left(1+\frac{1}{\sqrt{t-s}}\right)  \|\chi_{x_0,R}c(s)\nabla(-\Delta)^{-1}c(s)\|_{L^q}ds,
\end{align*}
for some hidden constant independent of $t_0, t, x_0, R$.
Let us further analyze the last term. By H\"older's inequality, Sobolev embedding, interpolation and equations \eqref{properties_inversion_localization_1}, \eqref{properties_inversion_localization_2}, \eqref{properties_inversion_localization_3}, \eqref{properties_inversion_localization_4} it holds, for $r_0\in (q,r)$ satisfying $\frac{1}{q}-\frac{1}{r}+\frac{1}{n}=\frac{1}{r_0}$ and $\theta=\frac{\beta}{1+\beta}+\frac{(2+\beta)q}{n(1+\beta)}\in (0,1)$,
\begin{align*}
    \|\chi_{x_0,R}c(s)\nabla(-\Delta)^{-1}c(s)\|_{L^q}&\leq \|\nabla(-\Delta)^{-1}c(s)\|_{L^{\frac{nr}{n-r}}}\|\chi_{x_0,R} c(s)\|_{L^{r_0}}\\ & \lesssim \|c(s)\|_{L^{r}}\|\chi_{x_0,R} c(s)\|_{L^{q}}^{\theta}\|\chi_{x_0,R} c(s)\|_{L^{r}}^{1-\theta}\\ & \leq \|\chi_{x_0,R} c(s)\|_{L^{q}}^{\theta}\|c(s)\|_{L^{r}}^{2-\theta}\\ & \lesssim \|\chi_{x_0,R} c(s)\|_{L^{q}}^{\theta} \frac{1}{s^{\left(1-\frac{n}{2r}\right)(2-\theta)}}.
\end{align*}
Since $r\in (\frac{n}{2},\frac{2n}{3})$ and $\theta\in (0,1)$, one can check
\begin{align}\label{eq:numerology}
    \left(1-\frac{n}{2r}\right)(2-\theta)+\frac{1}{2}<\frac{1}{4}(2-\theta)+\frac{1}{2}<1.
\end{align}
Therefore, from Young's inequality we obtain 
\begin{align*}
    \|c(t)\|_{L^q(B_R(x_0))}&\lesssim \|\chi_{x_0,R}c(t_0)\|_{L^q}+\int_{t_0}^t \frac{\|c(s)\|_{L^q(B_{2R}(x_0))}}{s^{\left(1-\frac{n}{2r}\right)(2-\theta)}\sqrt{t-s}}  ds\\ & +\int_{t_0}^t \frac{1}{s^{\left(1-\frac{n}{2r}\right)(2-\theta)}\sqrt{t-s}}ds.
\end{align*}
Since $c_0\in L^q$ and \eqref{eq:continuity_localization}, \eqref{eq:numerology} hold we can let $t_0\rightarrow 0$:
\begin{align*}
    \|c(t)\|_{L^q(B_R(x_0))}&\lesssim 1+\|c_0\|_{L^q}+\int_0^t \frac{\|c(s)\|_{L^q(B_{2R}(x_0))}}{s^{\left(1-\frac{n}{2r}\right)(2-\theta)}\sqrt{t-s}}  ds.
\end{align*}
Now, let us introduce the function
\begin{align*}
    c_R(t):=\sup_{x_0\in \R^n}\|c(t)\|_{L^q(B_R(x_0))}.
\end{align*}
Obviously $\|c(s)\|_{L^q(B_{2R}(x_0))}\lesssim c_R(s)$, and therefore
\begin{align*}
    \|c(t)\|_{L^q(B_R(x_0))} &\lesssim 1+\|c_0\|_{L^q}+\int_0^t \frac{c_R(s)}{s^{\left(1-\frac{n}{2r}\right)(2-\theta)}\sqrt{t-s}}  ds.
\end{align*}
Taking the supremum in $x_0$ in the expression above and applying Gr\"onwall's inequality we obtain
\begin{align*}
    c_R(t)&\lesssim  1+\|c_0\|_{L^q}.
\end{align*}
Letting $R\rightarrow +\infty$ this implies $c\in L^{\infty}([0,T'];L^q(\R^n))$.
Since we already know that $c \in C((0,T];L^r(\R^n))$, the latter implies that $c\in \mathcal{B}([0,T];L^q(\R^n)).$ Computations above and the continuity of $c(t)\chi_{x_0,R}$ in $L^q$ also imply that equation \eqref{mild_eq_1} can be extended up to time $t_0=0$, namely for each $t\in [0,T]$ it holds
\begin{align}\label{eq_mild_2}
    \chi_{x_0,R}c(t)&=e^{\Delta t}\chi_{x_0,R}c_0+\int_{0}^t e^{\Delta(t-s)}[ \Delta\chi_{x_0,R}c(s)-2\div (\nabla \chi_{x_0,R}c(s))]ds\notag\\ & -\int_{0}^t e^{\Delta(t-s)}[\operatorname{div}\left(\chi_{x_0,R}c(s)\nabla(-\Delta)^{-1}c(s)\right)-\nabla\chi_{x_0,R}c(s)\cdot\nabla(-\Delta)^{-1}c(s)]ds.  
\end{align}
We now exploit this formulation to show \eqref{eq:mild.formulation}, by letting $R\rightarrow+\infty$. Indeed, one trivially has
\begin{align*}
    \chi_{x_0,R}c(t)\rightarrow c(t),\ \chi_{x_0,R}c_0\rightarrow c_0,\quad \mbox{in } L^q,\quad\mbox{as }R\rightarrow +\infty.
\end{align*}
Moreover, since $c\in L^\infty(0,T;L^q(\R^n))$, the regularizing properties of the heat semigroup and previous computations imply that, as $R \to +\infty$, the following convergences hold in $L^q$ for each $s\in (0,t)$:
\begin{align*}
     &e^{\Delta(t-s)}[ \Delta\chi_{x_0,R}c(s)-2\div (\nabla \chi_{x_0,R}c(s))+\nabla\chi_{x_0,R}c(s)\cdot\nabla(-\Delta)^{-1}c(s)]\rightarrow 0,
     \\
     &
     e^{\Delta(t-s)}[\operatorname{div}\left(\chi_{x_0,R}c(s)\nabla(-\Delta)^{-1}c(s)\right)]\rightarrow e^{\Delta(t-s)}[\operatorname{div}\left(c(s)\nabla(-\Delta)^{-1}c(s)\right)],
\end{align*}
as well as the bounds
\begin{align}
    \|e^{\Delta(t-s)}[ \Delta\chi_{x_0,R}c(s)-2\div (\nabla \chi_{x_0,R}c(s))+\nabla\chi_{x_0,R}c(s)\cdot\nabla(-\Delta)^{-1}c(s)]\|_{L^q}
    &\lesssim \frac{1+\|c(s)\|_{L^q}}{s^{\left(1-\frac{n}{2r}\right)(2-\theta)}}, 
    \\ 
    \|e^{\Delta(t-s)}[\operatorname{div}\left(\chi_{x_0,R}c(s)\nabla(-\Delta)^{-1}c(s)\right)]\|_{L^q}
    &\lesssim 
     \frac{1+\|c(s)\|}{s^{\left(1-\frac{n}{2r}\right)(2-\theta)}\sqrt{t-s}}.
\end{align}
The right-hand sides above are in $L^1(0,t)$ as functions of $s$; therefore, we can pass to the limit in \eqref{eq_mild_2} as $R\rightarrow+\infty$ by dominated convergence theorem and recover the mild formulation \eqref{eq:mild.formulation}.
\\
We are left to show the continuity of $c$ in $L^q(\R^n).$ Let $0\leq t_0\leq t\leq T $; equation \eqref{eq:mild.formulation} implies
\begin{align*}
     c(t)-   c(t_0)=&\left(e^{\Delta t}c_0-e^{\Delta t_0}c_0\right) -\int_{t_0}^t e^{\Delta(t-s)}\operatorname{div}\left(c(s)\nabla(-\Delta)^{-1}c(s)\right)ds\\ & -\int_0^{t_0} (e^{\Delta(t-s)}-e^{\Delta(t_0-s)})\operatorname{div}\left(c(s)\nabla(-\Delta)^{-1}c(s)\right)ds ,   
    \end{align*}
as an equality in $L^q.$  Taking the $L^q$-norm of the expression above, by analogous considerations to the ones employed to obtain the uniform bound on the $L^q$-norm, we get
\begin{align*}
    \|c(t)-c(t_0)\|_{L^q}&\leq \|e^{\Delta t}c_0-e^{\Delta t_0}c_0\|_{L^q}+\left\|\int_{t_0}^t e^{\Delta(t-s)}\operatorname{div}\left(c(s)\nabla(-\Delta)^{-1}c(s)\right) ds\right\|_{L^q}\\ &+\left\|\int_0^{t_0} (e^{\Delta(t-s)}-e^{\Delta(t_0-s)})\operatorname{div}\left(c(s)\nabla(-\Delta)^{-1}c(s)\right) ds\right\|_{L^q}\\ & \lesssim  \|e^{\Delta t}c_0-e^{\Delta t_0}c_0\|_{L^q}+\int_{t_0}^t \frac{\|c(s)\|^{\theta}_{L^q}}{s^{\left(1-\frac{n}{2r}\right)(2-\theta)}\sqrt{t-s}} ds\\ & +\left\|\int_0^{t_0} (e^{\Delta(t-s)}-e^{\Delta(t_0-s)})\operatorname{div}\left(c(s)\nabla(-\Delta)^{-1}c(s)\right) ds\right\|_{L^q}.
\end{align*}
Since we already proved that $c\in L^{\infty}([0,T];L^q(\R^n))$, the continuity follows from the last inequality and Vitali convergence theorem. 
\end{proof}
We are left to show that $c_1\neq c_2$. Since $r>1$, $A$ is continuous from $L^r(\R^n)$ to $Y_{\nabla}^r$ thanks to \autoref{lem:isomorphism}. Therefore from \autoref{thm:existence_ancient_solutions} and \autoref{thm:nonlinear_localization} we have the lower bound
\begin{align*}
   \|c_1(t)-c_2(t)\|_{L^r}&\gtrsim \frac{1}{t^{1-\frac{n}{2r}}}\|\Psi(\log t)\|_{Y^{r}_{\nabla}} -\|h_1(t)-h_2(t)\|_{Y^{r}_{\nabla}}\\ & \gtrsim
    \frac{1}{t^{1-\frac{n}{2r}-\lambda_{\bar{\alpha}}}}\lVert \bar{\Omega}^{lin}\rVert_{Y_{\nabla}^r} -1\rightarrow +\infty,\quad \text{as }t\rightarrow 0,
\end{align*}
and the proof is complete.
\begin{rmk}
    For $q\in (1,\frac{n}{2})$, \autoref{lem:localization} and the fact that $A$ is a linear continuous map between $L^q_{rad}(\R^n)$ and $Y^q_{\nabla}$, together with the computations above, imply that $w_1,\ w_2$ are two different $Y^q$-mild solutions of equation \eqref{eq:w} in a sense analogous to \autoref{def_mild_sol}. In case of $q=1$ one can redo computations similar to those of this section directly on $w_1,\ w_2$ to show that $w_1,\ w_2\in C([0,T];Y^1)$ and they are two different $Y^1$-mild solutions of equation \eqref{eq:w} in a sense analogous to \autoref{def_mild_sol}.
\end{rmk}

\bibliography{biblio}{}
\bibliographystyle{alpha}

\end{document}